\documentclass[a4paper, 12pt]{article}

\usepackage{xcolor}
\usepackage{bbm}
\usepackage{amsmath,amssymb,mathrsfs}
\usepackage{amsthm} 
\usepackage{amsbsy} 
\usepackage{graphicx}
\usepackage{ulem}  

\usepackage{wasysym,pifont} 

\usepackage{enumerate}

\newtheorem*{Claim}{Claim}

\newtheorem{Lemma}{Lemma}

\newtheorem{Theorem}[Lemma]{Theorem}

\theoremstyle{definition}

\newtheorem{case}{Case}

\DeclareMathOperator{\R}{\mathbb{R}}

\newcommand{\ReSt}[2]{\left( #1 \mid #2 \right)}

\newcommand{\dfix}[2]{\mathfrak{#1}_{#2}}
\newcommand{\dproj}[1]{\mathrm{proj}\big[ #1 \big]} 

\newcommand{\drm}[2]{\mathrm{#1} \hspace{0.4mm} #2} 
\newcommand{\dsgn}[1]{\pmb{sgn} \hspace{0.7mm} \pmb{#1}} 
\newcommand{\signdu}[2]{\mathrm{sgn} \hspace{0.2mm} [#2] (#1)}
\newcommand{\dAM}[1]{\mathbf{A}( #1 )} 
 
\newcommand{\dbipart}[2]{\left[#1_{#2'},#1_{#2''}\right]}

\newcommand{\dubipart}[3]{\big[#1_{#2},#1_{#3}\big]}
\newcommand{\dumultipart}[3]{\left[#1_{#2},\cdots,#1_{#3}\right]}
\newcommand{\dusubgroup}[3]{\langle #1_{#2},\cdots,#1_{#3} \rangle}
\newcommand{\duspan}[2]{ \mathrm{span} \hspace{0.2mm} \left\{ #1:#2 \right\}}

\title{\upshape\bf On Graph Isomorphism Problem }

\author{ Wenxue Du \\
{\small  \it School of Mathematical Science, Anhui University, Hefei,
230601, China} \\
{\footnotesize E-mail:  wenxuedu@gmail.com }
}

\date{}

\begin{document}
\maketitle

\begin{abstract} 
Let $G$ and $H$ be two simple graphs. A bijection $\phi:V(G)\rightarrow V(H)$ is called an {\it isomorphism} between $G$ and $H$ if $(\phi\hspace{0.5mm} v_i)(\phi\hspace{0.5mm} v_j)\in E(H)$ $\Leftrightarrow$ $v_i v_j\in E(G)$ for any two vertices $v_i$ and $v_j$ of $G$. In the case that $G = H$, we say $\phi$ an automorphism of $G$ and denote the group consisting of all automorphisms of $G$ by $\mathrm{Aut}\hspace{0.5mm} G$. As well-known, the problem of determining whether or not two given graphs are isomorphic is called {\it Graph Isomorphism Problem} (GI). One of key steps in resolving GI is to work out the partition $\Pi^*_G$ of $V(G)$ composed of orbits of $\drm{Aut}{G}$. By means of geometric features of $\Pi^*_G$ and combinatorial constructions such as the multipartite graph $\dumultipart{\Pi^*}{t_1}{t_s}$, where $t_1,\ldots,t_s$ are vertices of $G$ constituting an orbit of $\drm{Aut}{G}$ and $\Pi^*_{t_i}$ ($i=1,\ldots,s$) is the partition comprised of orbits of the stabilizer $( \mathrm{Aut}\hspace{0.5mm} G )_{t_i}$, we can reduce the problem of determining $\Pi_G^*$ to that of working out a series of partitions of $V(G)$ each of which consists of orbits of a stabilizer that fixes a sequence of vertices of $G$, and thus the determination of the partition $\Pi^*_v$ is a critical transition.  

On the other hand, we have for a given subspace $U \subseteq \mathbb{R}^n$ a permutation group $\mathrm{Aut}\hspace{0.5mm} U$ which is defined as $\{ \sigma \in S_n : \sigma \hspace{0.5mm} U = U \}$. As a matter of fact, $\mathrm{Aut}\hspace{0.5mm} G = \cap_{\lambda \in \mathrm{spec} \hspace{0.3mm} \mathbf{A}(G) } \mathrm{Aut}\hspace{0.5mm} V_{\lambda}$, where $V_{\lambda}$ is the eigenspace of the adjacency matrix $\mathbf{A}(G)$ corresponding to $\lambda$, and moreover we can obtain a good approximation $\Pi[ \oplus V_{\lambda} ; v ]$ to $\Pi_v^*$ by analyzing a decomposition of $V_{\lambda}$ resulted from the division of $V_{\lambda}$ by subspaces $\{ \dproj{ V_{\lambda} }( \pmb{e}_v )^{\perp} : v \in V(G) \}$, where $\dproj{ V_{\lambda} }( \pmb{e}_v )$ denotes the orthogonal projection of the vector $\pmb{e}_v$ onto $V_{\lambda}$ and $\dproj{ V_{\lambda} }( \pmb{e}_v )^{\perp}$ stands for the orthogonal complement of the subspace spanned by $\dproj{ V_{\lambda} }( \pmb{e}_v )$ in $V_{\lambda}$. In fact, there is a close relation among subspaces spanned by cells of the equitable partition $\Pi[ \oplus V_{\lambda} ; v ]$ of $G$, which enables us to determine $\Pi_v^*$ more efficiently. In virtue of that, we devise a deterministic algorithm solving GI in time $n^{ O( \log n ) }$, which is equal to $2^{ O\left( \log^2 n \right) }$.
\end{abstract}

\vspace{3mm}

\section{Introduction}

Let $G$ and $H$ be two simple graphs. A bijective map $\phi$ from $V(G)$ to $V(H)$ is called an {\it isomorphism} between $G$ and $H$ if $(\phi\hspace{0.5mm} v_i)(\phi\hspace{0.5mm} v_j) \in E(H)$ $\Leftrightarrow$ $v_i v_j\in E(G)$ for any two vertices $v_i$ and $v_j$ of $G$. In the case that there is such an isomorphism between $G$ and $H$, we say that $G$ and $H$ are isomorphic, which is denoted by $G \cong H$. The problem of determining whether or not two given graphs are isomorphic is called {\it Graph Isomorphism Problem} (GI).

One of striking facts about GI is the following established by Whitney in 1930s.

\begin{Theorem}\label{Thm-WhitneyIsomorphism} 
Two connected graphs are isomorphic if and only if their line graphs are isomorphic, with a single exception: $K_3$ and $K_{1,3}$, which are not isomorphic but both have $K_3$ as their line graph.
\end{Theorem}

Clearly, the relation above offers a reduction of GI from general graphs to a special class of graphs --- line graphs, which accounts only for a small fraction of all graphs. This fact suggests that GI may not be very hard. In fact, GI is well solved from practical point of view and there are a number of efficient algorithms available \cite{McP}. Even from worst-case point of view, GI may not be as hard as NP-complete problems. As a matter of fact GI is not NP-complete unless the polynomial hierarchy collapses to its second level \cite{BHZ,Schonig}. On the other hand, however, we have no efficient algorithm so far for general graphs in worst-case analysis, while for restricted graph classes there are efficient algorithms, for instance, for graphs with bounded degree \cite{Luks} and for graphs with bounded eigenvalue multiplicity \cite{BaGrMu}. L. Babai \cite{Babai} recently declared an algorithm resolving GI for any graph of order $n$ within time $\exp\big\{ (\log n)^{O(1)} \big\}$ in worst-case analysis. In the present paper, we develop a machinery for GI from geometric point of view, which enables us devise a deterministic algorithm solving GI for any graph of order $n$ within time $2^{ O\left( \log^2 n \right) }$ in worst-case analysis. 

In the case that two graphs $G$ and $H$ involved are the same, an isomorphism is called an {\it automorphism} of $G$. Clearly, all automorphisms of $G$ form a group under composition of maps, which is denoted by $\drm{Aut}{G}$. Suppose the vertex set $V(G)$ is $\{1,2,\ldots,n\}$ abbreviated to $[n]$. Then a bijective map $\phi$ on $V(G)$ is a permutation of $[n]$, and thus the automorphism group $\drm{Aut}{G}$ is a permutation group of $[n]$. 

There is a natural action of $\drm{Aut}{G}$ on the vertex set $[n]$: $I v = v$, where $I$ is the identity of $\drm{Aut}{G}$, and $\gamma (\sigma v) = (\gamma \sigma) v$ for any two permutations $\gamma$ and $\sigma$ in $\drm{Aut}{G}$. Accordingly, we can obtain a subset $\{ \sigma v : \sigma \in \drm{Aut}{G} \}$ of $[n]$, which is called an {\it orbit} of $\drm{Aut}{G}$. Obviously, the orbits of $\drm{Aut}{G}$ constitute a partition of $[n]$, which is denoted by $\Pi_G^*$, and each orbit is called a cell of $\Pi_G^*$. One can readily see that for any subgroup $\mathfrak{S}$ of $\drm{Aut}{G}$ we have a partition of $[n]$ consisting of orbits of $\mathfrak{S}$.

Suppose $G$ and $H$ are isomorphic and $\phi$ is an isomorphism between $G$ and $H$. It is easy to see that $\phi$ induces a bijection between cells of $\Pi_G^*$ and of $\Pi_H^*$. Apparently, if $\drm{Aut}{G}$ is trivial, {\it i.e.,} there is only one permutation, the identity, in $\drm{Aut}{G}$, then the bijection from $\Pi_G^*$ to $\Pi_H^*$ is actually equal to $\phi$. So let us consider more interesting cases and assume $\drm{Aut}{G}$ possesses at least one non-trivial orbit. 

We take a vertex $u_1$ from a non-trivial orbit of $\drm{Aut}{G}$. Then there is exactly one vertex $v_1$ of $H$ corresponding to $u_1$ through $\phi$, and accordingly  $\phi$ induces a bijection between cells of $\Pi_{u_1}^*$ and of $\Pi_{v_1}^*$, where $\Pi_{u_1}^*$ and $\Pi_{v_1}^*$ are two partitions of $V(G)$ and $V(H)$, respectively, consisting of orbits of $( \drm{Aut}{G} )_{u_1}$ and of $( \drm{Aut}{H} )_{v_1}$, and $( \drm{Aut}{G} )_{u_1}$ stands for the subgroup of $\drm{Aut}{G}$ which is defined as $\{ \gamma \in \drm{Aut}{G} : \gamma \hspace{0.5mm} u_1 = u_1 \}$ and called the {\it stabilizer of $u_1$ in $\drm{Aut}{G}$.} Moreover, if $( \drm{Aut}{G} )_{u_1}$ is non-trivial, we could choose another vertex $u_2$ from a non-trivial orbit of $( \drm{Aut}{G} )_{u_1}$. Then we can get a vertex $v_2 = \phi\hspace{0.5mm} u_2$ of $H$ so that there is a bijection between $\Pi_{u_1,u_2}^*$ and $\Pi_{v_1,v_2}^*$ induced also by $\phi$, where $\Pi_{u_1,u_2}^*$ and $\Pi_{v_1,v_2}^*$ are two partitions of $V(G)$ and $V(H)$, respectively, consisting of orbits of $( \drm{Aut}{G} )_{u_1,u_2}$ and of $( \drm{Aut}{H} )_{v_1,v_2}$, and $( \drm{Aut}{G} )_{u_1,u_2}$ called the {\it stabilizer of the sequence $u_1,u_2$ in $\drm{Aut}{G}$} is defined as $\{ \gamma \in \drm{Aut}{G} : \gamma \hspace{0.5mm} u_i = u_i, i = 1,2 \}$. Clearly, we can continue this process until the stabilizer of the sequence $u_1,\ldots,u_s$ is trivial, {\it i.e.,} $( \drm{Aut}{G} )_{u_1,\ldots,u_s} = \{ I \}$. 

Conversely, if we have those two groups of partitions $\Pi_{G}^*,\Pi_{u_1}^*,\ldots,\Pi_{u_1,\ldots,u_s}^*$ and $\Pi_{H}^*,\Pi_{v_1}^*,\ldots,\Pi_{v_1,\ldots,v_s}^*$ and know the corresponding relations between cells of partitions in each pair $(\Pi^*_G,\Pi^*_H)$, $(\Pi^*_{u_1},\Pi^*_{v_1})$, $\cdots$, $(\Pi^*_{u_1,\ldots,u_s},\Pi^*_{v_1,\ldots,v_s})$, then we can easily decide whether $G$ is isomorphic to $H$ or not and in the case of being isomorphic work out an isomorphism from $G$ to $H$. 

In the next part, we shall explore some geometric features of $\Pi_G^*$ that show us how to reduce the problem of determining $\Pi_G^*$ to that of working out a series of partitions of $[n]$ each of which consists of orbits of a stabilizer that fixes a sequence of vertices of $G$, and thus the determination of the partition $\Pi^*_v$ is a critical transition.

The {\it adjacency matrix} of $G$, denoted by $\mathbf{A}(G)$, is a $n\times n$ (0,1)-matrix where each entry $a_{ij}$ of the matrix is equal to 1 if and only if the two vertices $v_i$ and $v_j$ are adjacent in $G$. In the second part, we will reveal some of geometric features of $\drm{Aut}{G}$ by means of the decomposition $\oplus V_{\lambda} = \R^n$, where $V_{\lambda}$ is the eigenspace of $\dAM{G}$ corresponding to the eigenvalue $\lambda$. In virtue of that, we could build a partition $\Pi[ \oplus V_{\lambda};v ]$, which is a good approximation to $\Pi^*_{v}$.

\subsection{Geometric Features of $\Pi_{G}^*$}

Let $\Pi$ be a partition of $[n]$ with cells $C_1,\ldots,C_t$, which is said to be {\it equitable} if for any vertex $v$ in $C_i$, the number of neighbors of $v$ in $C_j$ is a constant $b_{ij}$ $(1\leq i,j \leq t)$, {\it i.e.,} the number of neighbors in every cell is independent of the vertex $v$. Clearly, if $\mathfrak{S}$ is a subgroup of $\drm{Aut}{G}$ then the partition of $[n]$ consisting of orbits of $\mathfrak{S}$ is an equitable one. On the other hand, we can construct a direct graph $G / \Pi$ from $G$ and its equitable partition $\Pi$, which is called the {\it quotient graph} of $G$ over $\Pi$. The vertex set of $G / \Pi$ is composed of cells of $\Pi$ and there are $b_{ij}$ arcs $(1\leq i,j \leq t)$ from the $i$th vertex to the $j$th vertex of $V(G  / \Pi)$. 

For each cell $C_i$ ($i=1,\ldots,t$) of the partition $\Pi$, one can build a vector $\pmb{R}_{C_i}$, or abbreviated to $\pmb{R}_i$, to indicate $C_i$, that is called the {\it characteristic vector} of $C_i$, such that the $k$th coordinate ($1\leq k\leq n$) of the vector is 1 if $k$ belongs to $C_i$ otherwise it is 0. By means of  characteristic vectors, we can define the {\it characteristic matrix} $\mathbf{R}_{\Pi}$ of $\Pi$ as $(\pmb{R}_1 \pmb{R}_2 \cdots \pmb{R}_t)$. It is not difficult to verify that a partition $\Pi$ of $[n]$ is equitable if and only if the column space of $\mathbf{R}_{\Pi}$ is $\dAM{G}$-invariant (see \cite{GodRoy} for details).

As well-known, if the partition $\Pi$ involved is equitable, there is a close relation between eigenvalues and eigenvectors of $\mathbf{A}(G)$ and that of $\mathbf{A}(G / \Pi)$. To be precise, $\drm{spec}{ \dAM{G / \Pi} } \subseteq \drm{spec}{ \dAM{G} }$, and if $\pmb{x}_{\lambda}$ is an eigenvector of $\mathbf{A}(G / \Pi)$, corresponding to the eigenvalue $\lambda$, then $\mathbf{R}_{\Pi} \pmb{x}_{\lambda}$ is an eigenvector of $\mathbf{A}(G)$, corresponding to $\lambda$ also, where $\mathbf{R}_{\Pi}$ is the characteristic matrix of $\Pi$. Accordingly, we say that the eigenvector $\pmb{x}_{\lambda}$ of $\mathbf{A}(G / \Pi)$ {\it ``lifts''} to the eigenvector $\mathbf{R}_{\Pi} \pmb{x}_{\lambda}$ of $\mathbf{A}(G)$. Moreover all eigenvectors of $\mathbf{A}(G)$ could be divided into two classes: those that are constant on every cell of $\Pi$ and those that sum to zero on each cell of $\Pi$. As one can readily see, the first class consists of vectors lifted from eigenvectors of $\mathbf{A}(G / \Pi)$. In other words, if $\Pi = \{ C_1,\ldots,C_t \}$ is an equitable partition and  $x$ and $y$ are two vertices of $G$ belonging to the same cell of $\Pi$, then 
\begin{equation}\label{Def-EquitablePartition}
\langle \pmb{e}_x,\dproj{V_\lambda}(\pmb{R}_j) \rangle = \langle \pmb{e}_y,\dproj{V_\lambda}(\pmb{R}_j) \rangle, ~\forall~\lambda\in\mathrm{spec}~\mathbf{A}(G) \mbox{\it ~and } j\in [t], 
\end{equation}
where $\pmb{R}_j$ is the characteristic vector of $C_j$ and the vector $\dproj{V_\lambda}(\pmb{R}_j)$ is the orthogonal projection of $\pmb{R}_j$ onto the eigenspace $V_{\lambda}$. As we shall see below, the relation above is also sufficient for being equitable.

\begin{Lemma}\label{Lemma-EquitablePartProj}
Let $\Pi = \{ C_1,\ldots,C_t \}$ be a partition of $V(G)$. Then $\Pi$ is equitable if and only if for any two vertices $x$ and $y$ belonging to the same cell of $\Pi$, the relation (\ref{Def-EquitablePartition}) holds.
\end{Lemma}
\begin{proof}[\bf Proof]
We have discussed the necessity of our assertion, so let us show the sufficiency now. Obviously, the vectors $\pmb{R}_1,\ldots,\pmb{R}_t$ comprise an orthogonal basis of $U_\Pi$, which is the column space of $\mathbf{R}_{\Pi}$. To prove $U_\Pi$ is $\mathbf{A}(G)$-invariant, it suffices to show that $\mathbf{A}(G) \pmb{R}_k$ ($1\leq k \leq t$) can be written as a linear combination of $\pmb{R}_1,\ldots,\pmb{R}_t$.

In fact, 
\begin{align*}
\mathbf{A}(G) \pmb{R}_k 
& = \mathbf{A}(G) \left( \sum_{\lambda \hspace{0.2mm} \in \hspace{0.2mm} \mathrm{spec} \hspace{0.2mm} \mathbf{A}(G)} \dproj{V_{\lambda}}( \pmb{R}_k ) \right) \\
& = \sum_{\lambda} \mathbf{A}(G)\dproj{V_{\lambda}}( \pmb{R}_k ) \\
& = \sum_{\lambda} \lambda\cdot\dproj{V_{\lambda}}( \pmb{R}_k ).
\end{align*}
In accordance with our assumption, one can readily see that $\dproj{V_{\lambda}}( \pmb{R}_k )$ can be expressed as a linear combination of $\pmb{R}_1,\ldots,\pmb{R}_t$, so is $\mathbf{A}(G) \pmb{R}_k$.
\end{proof}

Clearly Lemma \ref{Lemma-EquitablePartProj} shows us that if $\Pi$ is an equitable partition and $C$ is a cell of $\Pi$, then the projection $\dproj{V_{\lambda}}( \pmb{R}_C )$ is in the subspace $\mathbf{R}_{\Pi} V_{\lambda}^{ G / \Pi }$, where $V_{\lambda}^{ G / \Pi }$ is the eigenspace of $\dAM{G / \Pi}$ corresponding to $\lambda$, and thus
\begin{align*}
\dproj{V_{\lambda}}( \pmb{R}_C ) 
& = \dproj{ \mathbf{R}_{\Pi} V_{\lambda}^{ G / \Pi } }( \pmb{R}_C )  \\
& = \sum_{c\in C} \dproj{ \mathbf{R}_{\Pi} V_{\lambda}^{ G / \Pi } }( \pmb{e}_c )  \\
& = |C| \cdot \dproj{ \mathbf{R}_{\Pi} V_{\lambda}^{ G / \Pi } }( \pmb{e}_c ).
\end{align*}
On the other hand, $\pmb{R}_C = \sum_{ \lambda \in \drm{spec}{ \dAM{G} } } \dproj{V_{\lambda}}( \pmb{R}_C )$. Therefore,
\begin{equation}\label{Equ-OrbitsAutG-Equitable}
\frac{1}{ |C| } \cdot \pmb{R}_C = 
\sum_{ \lambda \in \drm{spec}{ \dAM{G} } } \dproj{ \mathbf{R}_{\Pi} V_{\lambda}^{ G / \Pi } }( \pmb{e}_c ), ~~~~ \forall c \in C.
\end{equation}
This relation reveals that in order to determine the partition $\Pi_G^*$, we only need to work out those subspaces $\mathbf{R}_{\Pi_G^*} V_{\lambda}^{ G / \Pi_G^* }$ for each $\lambda \in \drm{spec}{ \dAM{G} }$.

Before showing how to obtain $\mathbf{R}_{\Pi_G^*} V_{\lambda}^{ G / \Pi_G^* }$ without knowing the partition  $\Pi_G^*$, we first introduce two kinds of subspaces $V_{\lambda} \langle u \rangle$ and $V_{1}^{\gamma}$ relevant to subgroups of $\drm{Aut}{G}$. For convenience, we use $\mathfrak{G}$ in what follows to denote the permutation group $\drm{Aut}{G}$.
\begin{equation}\label{Def-SubspacePointStabilizer}
V_{\lambda} \langle u \rangle :=
\{ \pmb{v} \in V_{\lambda} \mid \xi \hspace{0.5mm} \pmb{v} = \pmb{v}, ~ \forall \hspace{0.6mm} \xi\in\mathfrak{G}_u  \}, ~~ u\in [n].
\end{equation}
\begin{equation}\label{Def-SubspaceEigenvalue1}
V_{1}^{\gamma} := 
\{ \pmb{v} \in \R^n \mid \gamma \hspace{0.5mm} \pmb{v} = \pmb{v} \}, ~~ \gamma \in \mathfrak{G}.
\end{equation}
Apparently the partition $\Pi_u^*$ composed of orbits of $\mathfrak{G}_u$ is equitable and  $V_{\lambda} \langle u \rangle = \mathbf{R}_{\Pi_u^*} V_{\lambda}^{ G / \Pi_u^* }$. 

One can readily see that there are for any vertex $v$ of $G$ two possibilities: 
\begin{equation}\label{Equ-ActionGonV}
\mbox{either } \sigma v = v \mbox{ or } \sigma v \neq v, ~ \forall\hspace{0.6mm} \sigma \in \mathfrak{G}.
\end{equation}
It is interesting that there might be some subsets of $[n]$ possessing that relation (\ref{Equ-ActionGonV}). Let $B$ be a non-empty subset contained in some orbit $T$ of $\mathfrak{G}$, which is called a {\it block} for $\mathfrak{G}$ if 
$$\mbox{ either }\sigma B = B \mbox{ or } \sigma B \cap B = \emptyset, ~ \forall\hspace{0.6mm} \sigma \in \mathfrak{G}.
$$ 
Evidently, any element $t$ of $T$ and the orbit $T$ itself are blocks for $\mathfrak{G}$. If the group $\mathfrak{G}$ has only two such kinds of blocks in $T$ we say the action of $\mathfrak{G}$ on $T$ is {\it primitive}, otherwise {\it imprimitive.} On the other hand, the family of subsets $\{ \gamma B : \gamma \in \mathfrak{G} \}$ forms a partition of $T$, which is called the {\it system of blocks containing $B$} and denoted by $\mathscr{B}$. The action of $\mathfrak{G}$ on the system $\mathscr{B}$ is said to be {\it regular} if for any $\gamma \in \mathfrak{G}$, the stabilizer $\mathfrak{G}_B$ fixes $\gamma B$. 

Let $B_1,\ldots,B_m$ be a sequence of blocks for $\mathfrak{G}$ such that $B_1 \subsetneq B_2 \subsetneq \cdots \subsetneq B_m \subsetneq B_{m+1} = T$, $B_1$ is a minimal block and $B_i$ is maximal in $B_{i+1}$, {\it i.e.,} there is no block $K$ for $\mathfrak{G}$ so that $B_i \subsetneq K \subsetneq B_{i+1}$, $i=1,\ldots,m$. That kind of sequence is said to be a {\it block family} of $\mathfrak{G}$. Suppose $\mathscr{B}_i$ is the block system of $\mathfrak{G}$ containing $B_{i}$. We call those systems involved a {\it block system family of } $\mathfrak{G}$, which is denoted by $\mathscr{B}_1 \gneq \mathscr{B}_2 \gneq \cdots \gneq \mathscr{B}_m$. Suppose further that $\mathscr{B}_{i_1},\ldots,\mathscr{B}_{i_r}$ are those systems in the family such that the action of $\mathfrak{G}$ on $\mathscr{B}_{i_j}$ ($j = 1,\ldots,r$) is regular and $\gamma_{1},\gamma_{2},\ldots,\gamma_{r}$ are a group of permutations in $\mathfrak{G}$ such that $\gamma_{j} B_{i_j} \neq B_{i_j}$ and $\gamma_{j} B_{i_{j+1}} = B_{i_{j+1}}$.

\begin{Theorem}\label{Thm-Equation-OrbitsOfAutG}
$$
\mathbf{R}_{\Pi^*_G} V_{\lambda}^{ G / \Pi^*_G } 
= 
\left( \bigcap_{t\in T} \mathbf{R}_{\Pi_t^*} V_{\lambda}^{ G / \Pi_t^* } \right) 
\bigcap 
\left( \bigcap_{j=1}^r V_1^{ \gamma_j} \right)_{\textstyle .}
$$
\end{Theorem}

Accordingly, in order to determine the partition $\Pi^*_G$, we only need to have one partition $\Pi^*_t$ and a group of permutations $\gamma_1,\ldots,\gamma_{_{\footnotesize l(T)}}$ in $\mathfrak{G}$, where $t$ is an element of a non-trivial $\mathfrak{G}$ orbit $T$, which is composed of $t_1=t,t_2,\ldots,t_{_{\footnotesize l(T)}}$, and $\gamma_i\hspace{0.4mm} t = t_i$, $i=1,\ldots,l(T)$. Similarly, in order to determine $\Pi^*_t$, we only need to know one partition $\Pi^*_{t,u}$, which is composed of orbits of the stabilizer $\left( \mathfrak{G}_t \right)_u$, abbreviated to $\mathfrak{G}_{t,u}$, and a group of permutations $\delta_1,\ldots,\delta_{l(R)}$ in $\mathfrak{G}_t$ such that $u$ belongs to some non-trivial orbit $R$ of $\mathfrak{G}_t$, which consists of elements $u_1=u,u_2,\ldots,u_{l(R)}$, and $\delta_i\hspace{0.4mm} u = u_i$, $i=1,\ldots,l(R)$. Apparently, we can repeat this process until the partition consisting of orbits of the last  stabilizer, which fixes a sequence of vertices of $G$, is made up of trivial cells only, {\it i.e.,} the final partition is equal to $\{\{v\} : v \in [n] \}$.

We call a sequence of vertices $u_1,\ldots,u_s$ a fastening sequence of $\mathfrak{G}$ if $u_1$ belongs to a non-trivial orbit of $\mathfrak{G}$, $u_i$ belongs to some non-trivial orbit of $\mathfrak{G}_{u_1,\ldots,u_{i-1}}$ ($i=2,\ldots,s$) and $\mathfrak{G}_{u_1,\ldots,u_{s}} = \{ 1 \}$, where $\mathfrak{G}_{u_1,\ldots,u_{i-1}} = \left\{ \gamma \in \mathfrak{G} : \gamma \hspace{0.5mm} u_k = u_k, k = 1,\ldots,i-1 \right\}$. Let $x_1,x_2,\ldots,x_{s}$ and $y_1,y_2,\ldots,y_{s}$ be two fastening sequences of $\mathfrak{G}$. A moment's reflection would show that in order to determine a permutation in $\mathfrak{G}$ mapping $x_1$ to $y_1$, we only need to work out two group of partitions $\Pi^*_{x_1},\Pi^*_{x_1,x_2},\ldots,\Pi^*_{x_1,\ldots,x_{s}}$ and $\Pi^*_{y_1},\Pi^*_{y_1,y_2},\ldots,\Pi^*_{y_1,\ldots,y_{s}}$ and to know the corresponding relation between cells of partitions in each pair $( \Pi^*_{x_1,\ldots,x_k},\Pi^*_{y_1,\ldots,y_k} )$, $k=1,\ldots,s$. 

In the 3rd section, we will show how to work out those partitions and determine the corresponding relation between cells of partitions in each pair. For convenience, we use the term ``{\it information about $\mathfrak{G}$}'' to describe the information about the partition $\Pi_G^*$ and a series of partitions of $[n]$ associated with a fastening sequence of $\mathfrak{G}$.

\subsection{The Partition $\Pi[ \oplus V_{\lambda};v ]$ --- an Approximation to $\Pi^*_v$}

Obviously, all permutations of $[n]$ form a group under composition of maps, which is called the {\it symmetric group of degree} $n$ and denoted by $\drm{Sym}{[n]}$, or by $S_n$ for short. Each permutation $\sigma$ in $S_n$ can act on a vector $\pmb{u}=(u_1,\ldots,u_n)^t$ of $\R^n$ in a natural way: 
\begin{equation}\label{Def-PermutationOperator} \sigma \hspace{0.5mm} \pmb{u}=(u_{\sigma^{-1}1},u_{\sigma^{-1}2},\ldots,u_{\sigma^{-1} n})^t,
\end{equation}
 where $\R^n$ is the $n$-dimensional vector space over the real field $\R$. Accordingly, any permutation $\sigma$ in $S_n$ can be regarded, through the action on vectors, as a linear operator on $\R^n$. We call a (0,1)-square matrix a {\it permutation matrix} if in each row and column there is exactly one entry that is equal to 1. It is easy to check that the matrix $\mathbf{P}_{\sigma}$ of the operator $\sigma$ with respect to the standard basis $\pmb{e}_1,\ldots,\pmb{e}_n$ is a permutation matrix, where each $\pmb{e}_i$ $(i=1,\ldots,n)$ has exactly one non-trivial entry on $i$th coordinate that is equal to 1, and all other entries of $\pmb{e}_i$ are equal to 0.

Recall that the vertex set $V(G)$ is $[n]$, so a bijective map $\phi$ from $V(G)$ to itself is a permutation of $[n]$. It is easy to check that 
\begin{equation}\label{Equation00} 
\phi\mbox{ is an automorphism of }G\mbox{ if and only if }\mathbf{P}_{\phi}^{-1}\dAM{G}\mathbf{P}_{\phi}=\dAM{G},
\end{equation}
which presents an algebraic way of characterizing automorphisms of $G$. There is in virtue of eigenspaces of $\dAM{G}$ another way to characterize  automorphisms of $G$.

\begin{Lemma}\label{Lem-AutomorphismAndOperator} 
Let $G$ be a graph with the vertex set $[n]$ and let $\sigma$ be a permutation in $S_n$. Then $\sigma$ is an automorphism of $G$ if and only if every eigenspace of $\dAM{G}$ is $\sigma$-invariant. 
\end{Lemma}

Recall that the $n$-dimensional vector space $\R^n$ is endowed with the {\it  inner product} $\langle\cdot,\cdot\rangle$ such that $\langle\pmb{u},\pmb{v}\rangle=\pmb{v}^t\pmb{u}=\sum_{i=1}^n u_i\cdot v_i$ for any vectors $\pmb{u}=(u_1,\ldots,u_n)^t$ and $\pmb{v}=(v_1,\ldots,v_n)^t$ in $\R^n$. Since the matrix $\mathbf{A}(G)$ is symmetric, there is an orthonormal basis of $\R^n$ consisting of eigenvectors of $\dAM{G}$ according to the real spectral theorem (see \cite{Axler} for details).

\begin{proof}[\bf Proof]  
We begin with the necessity of the assertion. In accordance with the relation (\ref{Equation00}), $\sigma$ is an automorphism of $G$ if and only if $\mathbf{P}_{\sigma}^t\mathbf{A}\mathbf{P}_{\sigma}=\mathbf{A}$, so for any eigenvector $\pmb{v}$ of $\mathbf{A}$ corresponding to some eigenvalue $\lambda$, $$\mathbf{P}_{\sigma}^t\mathbf{A}\mathbf{P}_{\sigma}\pmb{v}=\mathbf{A}\pmb{v}=\lambda\pmb{v}.$$
Consequently, $\mathbf{A}\mathbf{P}_{\sigma}\pmb{v}=\lambda\mathbf{P}_{\sigma}\pmb{v}$, which means $\mathbf{P}_{\sigma}\pmb{v}$ is also an eigenvector of $\mathbf{A}$ corresponding to $\lambda$, and thus every eigenspace of $\mathbf{A}$ is $\mathbf{P}_{\sigma}$-invariant.

Conversely, let us select an orthonormal basis $\pmb{x}_1,\ldots,\pmb{x}_n$ of $\R^n$, consisting of eigenvectors of $\mathbf{A}$ such that $\mathbf{A}\pmb{x}_i = \lambda_i\pmb{x}_i$, $i=1,\ldots,n$. Since every eigenspace of $\mathbf{A}$ is $\mathbf{P}_{\sigma}$-invariant, for every $\pmb{x}_i$ we have
$$
\mathbf{A}\mathbf{P}_{\sigma}\pmb{x}_i
 = \lambda_i\mathbf{P}_{\sigma}\pmb{x}_i
 = \mathbf{P}_{\sigma} \lambda_i\pmb{x}_i
 =\mathbf{P}_{\sigma}\mathbf{A}\pmb{x}_i.
$$
Consequently, for an arbitrary vector $\pmb{v}=\sum_{i=1}^na_i\pmb{x}_i$ in $\R^n$,
$$
\mathbf{P}_{\sigma}\mathbf{A}\pmb{v}=\mathbf{P}_{\sigma}\mathbf{A}\sum_{i=1}^na_i\pmb{x}_i
=\sum_{i=1}^na_i\mathbf{P}_{\sigma}\mathbf{A}\pmb{x}_i=
\sum_{i=1}^na_i\mathbf{A}\mathbf{P}_{\sigma}\pmb{x}_i
=\mathbf{A}\mathbf{P}_{\sigma}\sum_{i=1}^na_i\pmb{x}_i
=\mathbf{A}\mathbf{P}_{\sigma}\pmb{v}.
$$
As a result, $\mathbf{P}_{\sigma}\mathbf{A}=\mathbf{A}\mathbf{P}_{\sigma}$, and thus the permutation $\sigma$ is an automorphism of $G$.
\end{proof}

In accordance with Lemma \ref{Lem-AutomorphismAndOperator}, we can describe automorphisms of $G$ and so the group $\drm{Aut}{G}$ in terms of eigenspaces of $\mathbf{A}(G)$. Let $U$ be a non-trivial subspace in $\R^n$. Set 
$$
\drm{Aut}{U} = \{ \sigma \in S_n : \sigma U = U \}.
$$
Then
\begin{equation}\label{Equ-AutG-EigSpace}
\drm{Aut}{G} = \bigcap_{\lambda \hspace{0.5mm} \in  \hspace{0.5mm} \drm{spec}{\dAM{G}}} \drm{Aut}{V_{\lambda}}.
\end{equation}
For convenience, we denote the right hand side of the equation above by $\drm{Aut}{\oplus V_{\lambda}}$. The relation (\ref{Equ-AutG-EigSpace}) shows us that each eigenspace uncovers some information useful in characterizing the 
 $\drm{Aut}{G}$ action on $[n]$. As we have seen, the family of partitions $\{ \Pi^*_v : v \in [n] \}$ is critical in determining the partition $\Pi_G^*$, so let us see how to gather information about the partition $\Pi^*_v(\lambda)$ of $[n]$, which is composed of orbits of $(\drm{Aut}{V_{\lambda}})_v$, where $\lambda \in \drm{spec}{\dAM{G}}$.
 
Recall that a linear operator $\mathcal{T}$ on $\R^n$ is said to be an {\it isometry} if $\| \mathcal{T} \pmb{v} \| = \| \pmb{v} \|$ for any vector $\pmb{v}$ in $\R^n$. It is easy to check that a permutation on $[n]$ is an isometry on $\R^n$.

\begin{Lemma}\label{ProjOperatorCommutative} 
Let $\mathcal{T}$ be an isometry on $\mathbb{R}^n $, and let $U$ be a subspace of  $\mathbb{R}^n $. Then the following statements are equivalent. 
\begin{enumerate}
    \item $U$ is $\mathcal{T}$-invariant.
    \item $\mathcal{T}\circ\dproj{U} =\dproj{U}\circ\mathcal{T}$, where $\dproj{U}$ is the orthogonal projection onto the subspace $U$. 
    \item There exists a basis $\pmb{b}_1,\ldots,\pmb{b}_n$ of $\R^n$ so that $\mathcal{T}\circ\dproj{U}(\pmb{b}_i) =\dproj{U}\circ\mathcal{T}(\pmb{b}_i)$, $i=1,\ldots,n$. 
\end{enumerate}
\end{Lemma}
\begin{proof}[\bf Proof]  We first verify that {\it i)}$\Rightarrow${\it ii)}. Let $\pmb{v}$ be a vector of $\R^n$. Then there exist uniquely $\pmb{u}\in U$ and $\pmb{u}'\in U^{\bot}$ so that $\pmb{v}=\pmb{u}+\pmb{u}'$. Consequently,   $\mathcal{T}\circ\dproj{U}(\pmb{v})=\mathcal{T}\circ\dproj{U}(\pmb{u}+\pmb{u}')=\mathcal{T}(\pmb{u})=\dproj{U}(\mathcal{T}(\pmb{u})+\mathcal{T}(\pmb{u}'))=\dproj{U}\circ\mathcal{T}(\pmb{v})$ since $\mathcal{T}$ is an isometry and $U$ is an $\mathcal{T}$-invariant subspace.

Clearly, the 2nd statement can imply the 3rd one. So now we turn to the last part and show that the 3rd statement implies the 1st one.

Let us first recall a fact that 
\begin{equation*} \mathcal{T}U=U \mbox{ if and only if }\mathcal{T}(\pmb{u}) = \dproj{U}\circ\mathcal{T}(\pmb{u}), ~\forall \pmb{u}\in U. \end{equation*}  
Since $\pmb{b}_1,\ldots,\pmb{b}_n $ is a basis of $\mathbb{R}^n$, for any vector $\pmb{u}\in U$, $\pmb{u}=\sum_{i=1}^n u_i\pmb{b}_i$ where $u_i\in\R$ and $i=1,\ldots,n$. In accordance with the 3rd statement, we have
\begin{align*} \dproj{U}\circ\mathcal{T}(\pmb{u}) 
&= \sum_{i=1}^n u_i\cdot\dproj{U}\circ\mathcal{T}(\pmb{b}_i) 
= \sum_{i=1}^n u_i\cdot\mathcal{T}\circ\dproj{U}(\pmb{b}_i) \\
&= \mathcal{T}\circ\dproj{U}\left(\sum_{i=1}^n u_i\pmb{b}_i\right)
= \mathcal{T}\left(\pmb{u}\right). \end{align*}
\end{proof} 
 
In accordance with Lemma \ref{ProjOperatorCommutative}, $(\drm{Aut}{V_{\lambda}})_v = \{ \sigma \in \drm{Aut}{V_{\lambda}} : \sigma \hspace{0.5mm} {\scriptstyle \circ} \hspace{0.5mm} \dproj{V_{\lambda}}(\pmb{e}_v) = \dproj{V_{\lambda}}(\pmb{e}_v) \}$. Hence, one can easily obtain a partition of $[n]$ relevant to $\Pi^*_v(\lambda)$, which is induced by coordinates of the vector $\dproj{V_{\lambda}}(\pmb{e}_v)$, {\it i.e.,} two vertices belong to one cell of the partition if the coordinates corresponding to them are the same. As a matter of fact, we can work out a refined partition more close to $\Pi^*_v(\lambda)$ in virtue of a geometric feature of $\drm{Aut}{V_{\lambda}}$ --- region.    

Let $X$ be a subspace of $\R^n$. As we shall see below, a region of $X$ can  be defined in two ways --- outside or inside. Let us first define region outside. Suppose $\pmb{b}_1,\ldots,\pmb{b}_m$ is a group of vectors in $\R^n$ such that $\drm{span}{\{ \pmb{b}_1,\ldots,\pmb{b}_m \}} = \R^n$ and $\| \pmb{b}_i \| > 0$ ($i=1,\ldots,m$) where $\drm{span}{\{ \pmb{b}_i : i \in [m] \}}$ stands for the subspace spanned by $\pmb{b}_1,\ldots,\pmb{b}_m$. Clearly, for each member $\pmb{b}_i$ in the group, there is one unique subspace of dimension $n-1$, which is the orthogonal complement, denoted by $\pmb{b}_i^{ \perp }$, of the vector $\pmb{b}_i$ and called the {\it divider of $\R^n$ associated with $\pmb{b}_i$}. Then the whole space $\R^n$ is divided into 3 parts by $\pmb{b}_i^{ \perp }$:
\begin{enumerate}

\item those vectors each of which has a positive inner product with $\pmb{b}_i$, so we denote the part by $\pmb{b}_i^+$;

\item those vectors each of which has a negative inner product with $\pmb{b}_i$, so we denote the part by $\pmb{b}_i^-$;

\item those vectors each of which is orthogonal to $\pmb{b}_i$, so we denote the part by $\pmb{b}_i^{ \perp }$.

\end{enumerate}
By using all dividers $\pmb{b}_i^{\perp},\ldots,\pmb{b}_m^{\perp}$ of $\R^n$, the whole space can be divided into many parts of two classes: those each of which is comprised of vectors not orthogonal to any vector in the group $\{ \pmb{b}_i : i \in [m] \}$, and those each of which is contained in some divider. In order to investigate those parts contained in a divider $\pmb{b}_i^{ \perp }$, we focus on dividers of the subspace $\pmb{b}_i^{ \perp }$ associated with vectors $\left\{ \dproj{ \pmb{b}_i^{ \perp } }( \pmb{b}_k ) : k \in [m]\setminus\{i\} \right\}$. In this way, we divide the whole space $\R^n$ into parts such that any two of them have only trivial intersection $\{ \pmb{0} \}$, each part resulted is called a {\it region} of $\R^n$ with respect to $\{ \pmb{b}_i : i \in [m] \}$.

Because $X$ is a subspace of $\mathbb{R}^n$, it is naturally divided into a number of parts by those dividers $\{ \pmb{b}_i^{\perp} : i \in [m] \}$, each of which is called a {\itshape region} of $X$. More precisely, a region of $X$ is obtained by intersecting $X$ with some region of $\R^n$. A moment's reflection shows that any region of $X$ is convex.

Now let us try to carve up $X$ inside, that shows another way of defining region. First, we figure out those orthogonal projections of the group $\pmb{b}_1,\ldots,\pmb{b}_m$ onto $X$, which are denoted by $\pmb{x}_1,\ldots,\pmb{x}_m$. Because $\drm{span}{\{ \pmb{b}_i : i \in [m] \}} = \R^n$, $\drm{span}{\{ \pmb{x}_i : i \in [m] \}} = X$.  It is clear that each vector $\pmb{x}_i$ in the group such that $\| \pmb{x}_i \| \neq 0$ possesses uniquely one orthogonal complement in $X$, which is denoted by $\pmb{x}_i^{\perp}$ and called the {\itshape divider of $X$ associated with} $\pmb{x}_i$. Then those dividers $\pmb{x}_1^{\perp},\ldots,\pmb{x}_m^{\perp}$ carve up $X$ into a number of parts, and again each part is called a region of $X$. One can readily see that those two ways of defining region are equivalent.

Although a region contains lots of vectors, we can always use one vector to indicate the region. Let us present several notions relevant step by step. We first consider a region $R$ which is not contained in any divider $\pmb{b}_k^{\perp}$ ($k = 1,\ldots,m$). A non-trivial vector $\pmb{s}_i$ of some divider $\pmb{b}_i^{\perp}$ is said to be a {\it straightforward projection} of $R$ if there exists $\pmb{r} \in R$ so that $\pmb{s}_i = \dproj{ \pmb{b}_i^{\perp} }( \pmb{r} )$ and $\theta ( \pmb{r} - \pmb{s}_i ) + \pmb{s}_i \in R$, $\forall \theta \in (0,1)$. We call a divider $\pmb{b}_i^{\perp}$ of $\R^n$ a {\it separator} of $R$ if the subspace 
$$
\drm{span}{\{ \pmb{s}_i \in \pmb{b}_i^{\perp} : \pmb{s}_i \mbox{ is a straightforward projection of } R \}} = \pmb{b}_i^{\perp}. 
$$
The {\itshape incidence set} of $R$, which is denoted by $\mathcal{I}_R$, is defined as follows:
$$
\mathcal{I}_R = \{ i \in [m] : \pmb{b}_i^{\perp} \mbox{ is a separator of } R \}.
$$
Furthermore, we define a sign function on the group $\{ \pmb{b}_i : i \in [m] \}$ related to $R$:
$$
\signdu{\pmb{b}_i}{R} = \left\{ 
\begin{matrix}
1  & \mbox{ if } \pmb{b}_i^{\perp} \mbox{ is a separator of } R & \hspace{-2mm} \mbox{ and } R \subseteq \pmb{b}_i^{+}; \\
-1 & \mbox{ if } \pmb{b}_i^{\perp} \mbox{ is a separator of } R & \hspace{-2mm} \mbox{ and } R \subseteq \pmb{b}_i^{-}; \\
0  & \hspace{2mm} \mbox{ otherwise. } & ~
\end{matrix}
\right.
$$
We are now ready to introduce the {\it indicator of $R$}, which is  
$$
\mbox{the vector } \sum_{ x\in\mathcal{I}_R } \frac{\signdu{\pmb{b}_x}{R}}{\|\pmb{b}_x\|} \cdot \pmb{b}_x \mbox{ and denoted by } \pmb{i}_R.
$$ 
It is easy to see that the indicator $\pmb{i}_R$ is contained in $R$.

Note that we assumed that the region $R$ considered is not contained in any divider involved, but it could be the case that there are some of dividers, $\pmb{b}^{\perp}_{ k_1 },\ldots,\pmb{b}^{\perp}_{ k_q }$, say, such that $R \subseteq \cap_{i\in [q]} \pmb{b}^{\perp}_{ k_i }$. Accordingly we should focus on the division of the subspace $X = \cap_{i\in [q]} \pmb{b}^{\perp}_{ k_i }$ with respect dividers 
$\big\{ \pmb{p}_j^{\perp} : j \in [m] \setminus \{ k_1,\ldots,k_q \} \big\},$
where $\pmb{p}_j = \dproj{ X }( \pmb{b}_j )$. Then we could define those four notions relevant in a slightly different way. More precisely,  a non-trivial vector $\pmb{s}_j$ of some divider $\pmb{p}_j^{\perp}$ is said to be a  straightforward projection of $R$ if there exists $\pmb{r} \in R$ so that $\pmb{s}_j = \dproj{ \pmb{p}_j^{\perp} }( \pmb{r} )$ and $\theta ( \pmb{r} - \pmb{s}_j ) + \pmb{s}_j \in R$, $\forall \theta \in (0,1)$. We call a divider $\pmb{p}_j^{\perp}$ of $X$ a separator of $R$ if the subspace 
$$
\drm{span}{\{ \pmb{s}_j \in \pmb{p}_j^{\perp} : \pmb{s}_j \mbox{ is a straightforward projection of } R \}} = \pmb{p}_j^{\perp}.
$$
The incidence set $\mathcal{I}_R$ of $R$ is defined as 
$\{ j \in [m] \setminus \{ k_1,\ldots,k_q \} : \pmb{p}_j^{\perp} \mbox{ is a separator of } R \}$, and the  sign function related to $R$ is defined as follows:
$$
\signdu{\pmb{p}_j}{R} = \left\{ 
\begin{matrix}
1  & \mbox{ if } \pmb{p}_j^{\perp} \mbox{ is a separator of } R & \hspace{-2mm} \mbox{ and }  R \subseteq \pmb{p}_j^{+}; \\
-1 & \mbox{ if } \pmb{p}_j^{\perp} \mbox{ is a separator of } R & \hspace{-2mm} \mbox{ and }  R \subseteq \pmb{p}_j^{-}; \\
0  & \hspace{2mm} \mbox{ otherwise. } & ~
\end{matrix}
\right.
$$
Finally the indicator of $R$ is the vector 
$\pmb{i}_R = \sum_{ x\in\mathcal{I}_R } \big( \signdu{\pmb{p}_x}{R}/ \| \pmb{p}_x \| \big) \cdot \pmb{p}_x$. It is easy to see that for any region $R$ of $\R^n$ with respect to $\{\pmb{b}_i : i \in [m] \}$, the key to identifying $R$ is to determine the incidence set $\mathcal{I}_R$.

In the case that $\dim X = 1$, there are essentially two regions in $X$. Suppose $\pmb{x}$ is a vector in $X$ of length 1. Then the region $R$ containing $\pmb{x}$ degenerates into the set $\{ r\cdot\pmb{x} : r \in \R^+ \}$ and another region is $\{ r\cdot( -\pmb{x} ) : r \in \R^+ \}$, so we can use $\pmb{x}$ and $-\pmb{x}$ to indicate those two regions, and thus we do not need separators or the incidence set of $R$ to distinguish it from another region. 

It is the division of $V_{\lambda}$ ($\lambda\in\drm{spec}{\dAM{G}}$) carved by the orthogonal projections of the standard basis (OPSB) $\left\{ \dproj{ V_{\lambda} }(\pmb{e}_v) : v \in [n] \right\}$ or by some of them that we are particularly interested in, because one region of $V_{\lambda}$ with respect to the OPSB is an elementary unit illustrating the action of $\drm{Aut}{V_{\lambda}}$ on $[n]$. 

Evidently for any member $\dproj{ V_{\lambda} }(\pmb{e}_v)$ in the OPSB, there is a region of $V_{\lambda}$ containing the projection. A moment's reflection would show that the subgroup $( \drm{Aut}{V_{\lambda}} )_v$ does not move the region containing $\dproj{ V_{\lambda} }(\pmb{e}_v)$, so the incidence set of the region must be an union of some of orbits of $( \drm{Aut}{V_{\lambda}} )_v$. Consequently, by carving up $V_{\lambda}$ layer by layer with regions containing $\dproj{ V_{\lambda} }(\pmb{e}_v)$, we can obtain a partition of $[n]$ each cell of which is composed of the incidence set of the region relevant.

Let us see how to determine the incidence set of a region. Suppose $X$ is a subspace of $\R^n$ with dimension larger than 2. Clearly, $\drm{span}{\{ \dproj{X}( \pmb{e}_v ) : v \in V(G) \}} = X$. Suppose $\pmb{x}$ is a vector in $X$ such that if $\pmb{e}_v \perp \pmb{x}$ then $\pmb{e}_v \perp X$  ($v \in V(G)$) and $R$ is a region of $X$ with respect to $\{ \dproj{X}( \pmb{e}_v ) : v \in V(G)$ and $\dproj{X}( \pmb{e}_v ) \neq\pmb{0} \}$, which contains $\pmb{x}$. It is not difficult to see that a vertex $v$ of $G$ belongs to the incidence set $\mathcal{I}_R$ if and only if $\exists \hspace{0.6mm} \pmb{q}_v \in \pmb{p}_v^{\perp}$, where $\pmb{p}_v = \dproj{ X }( \pmb{e}_v ) \neq\pmb{0}$, {\it s.t.,} 
\begin{equation}\label{Equ-IncidenceSetRegion}
\dsgn{x} - \dsgn{q}_v = \left( \drm{sgn}{ \langle \pmb{x} \rangle_v } \right) \cdot\pmb{e}_v,
\end{equation}
{\it i.e.,}
$$
\langle \dsgn{x} - \dsgn{q}_v \rangle_i =
\left\{ 
\begin{matrix}
\drm{sgn}{ \langle \pmb{x} \rangle_v } & \mbox{if } i =    v, \\ 
0                                      & \mbox{if } i \neq v,
\end{matrix}
\right.
$$
where $\dsgn{x}$ is the sign vector associated with the vector $\pmb{x} = (x_1,\ldots,x_n)^t$, which is defined as $(\drm{sgn}{x_1},\ldots,$ $\drm{sgn}{x_n})^t$, and $\langle \pmb{x} \rangle_v = x_v$. The key to seeing the relation (\ref{Equ-IncidenceSetRegion}) is to note that every region is convex.

\begin{Lemma}\label{Lem-IncidenceSetRegion}
Let $X$ is subspace of $\R^n$ of dimension larger than 2 and let $\pmb{x}$ be a vector of $X$ not orthogonal to any non-trivial projection $\pmb{p}_i = \dproj{ X }( \pmb{e}_i )$, $i \in [n]$. Suppose $R$ is a region of $X$ with respect to $\{ \pmb{p}_i : i \in [n]$ and $\pmb{p}_i \neq\pmb{0} \}$, which contains $\pmb{x}$. Then $v$ belongs to $\mathcal{I}_R$ if and only if  
$\pmb{sgn}\hspace{0.6mm} \dproj{X}(\pmb{s}_v) = \pmb{s}_v$, where $\pmb{s}_v = \dsgn{x} - \drm{sgn}{ \langle \pmb{x} \rangle_v } \cdot\pmb{e}_v$.
\end{Lemma}
\begin{proof}[\bf Proof]
We first present a simple observation. If $\pmb{u}$ is a vector of $X$, then 
\begin{equation*}
\pmb{sgn}\hspace{0.6mm} \dproj{X}( \dsgn{u} ) = \dsgn{u}.
\end{equation*}
Note that $\dsgn{u}$ is actually the indicator of one region $R^0_{\pmb{u}}$ of $\R^n$ carved up by dividers associated with the standard basis. Then $\pmb{u} \in X$ implies that $X \cap R^0_{\pmb{u}} \supsetneq \{ \pmb{0} \}$, and thus $\dproj{X}( \dsgn{u} ) \in R^0_{\pmb{u}}$. Hence $\pmb{sgn}\hspace{0.6mm} \dproj{X}( \dsgn{u} ) = \dsgn{u}$, for $\dsgn{z}' = \dsgn{z}''$ $\forall \pmb{z}',\pmb{z}'' \in R^0$, where $R^0$ is a region of $\R^n$.

Suppose $v \in \mathcal{I}_R$. Then $\exists \hspace{0.6mm} \pmb{q}_v \in \pmb{p}_v^{\perp}$ {\it s.t.,} 
$\dsgn{x} - \dsgn{q}_v = \drm{sgn}{ \langle \pmb{x} \rangle_v } \cdot\pmb{e}_v$, which implies that
$\dsgn{x} - \drm{sgn}{ \langle \pmb{x} \rangle_v } \cdot\pmb{e}_v = \dsgn{q}_v$, {\it i.e.,} 
$\pmb{s}_v = \dsgn{q}_v$. Therefore
$$
\pmb{sgn}\hspace{0.6mm} \dproj{X}(\pmb{s}_v) = 
\pmb{sgn}\hspace{0.6mm} \dproj{X}(\dsgn{q}_v) =
\dsgn{q}_v = \pmb{s}_v.
$$

\vspace{2mm}
On the other hand, because $\pmb{s}_v = \pmb{sgn}\hspace{0.6mm} \dproj{X}(\pmb{s}_v)$ and $\pmb{s}_v = \dsgn{x} - \drm{sgn}{ \langle \pmb{x} \rangle_v } \cdot\pmb{e}_v$, it is sufficient to show that $\dproj{X}(\pmb{s}_v) \in \pmb{p}_v^{\perp}$, which then implies that $\dproj{X}(\pmb{s}_v)$ is the vector $\pmb{q}_v$ we want in the relation (\ref{Equ-IncidenceSetRegion}).

Note that 
$\pmb{sgn}\hspace{0.6mm} \dproj{X}(\pmb{s}_v) = \dsgn{x} - \drm{sgn}{ \langle \pmb{x} \rangle_v } \cdot\pmb{e}_v$ 
$\Rightarrow$ $\langle \dproj{X}(\pmb{s}_v) \rangle_v = 0$ 
$\Rightarrow$ $\dproj{X}(\pmb{s}_v) \in \pmb{p}_v^{\perp}$.
\end{proof}

It is easy to see that in the case that the vector $\pmb{x}$ we select is contained in some dividers 
$\pmb{p}_{k_1}^{\perp},\ldots,\pmb{p}_{k_q}^{\perp}$ of $X$, the region $R$ containing $\pmb{x}$ must be in the subspace $\cap_{j=1}^q \pmb{p}_{k_j}^{\perp}$. Then we can employ Lemma \ref{Lem-IncidenceSetRegion} for $\cap_{j=1}^q \pmb{p}_{k_j}^{\perp}$ to find out the incidence set of $R$.  

As pointed above, we have a partition $\Pi[ V_{\lambda};v ]$ of $[n]$ built by grouping vertices of $G$ according to regions in $V_{\lambda}$ each of which contains the vector $\dproj{ V_{\lambda} }(\pmb{e}_v)$, {\it i.e.,} each cell of $\Pi[ V_{\lambda};v ]$ is composed of the members in the incidence set of the region that contains $\dproj{ V_{\lambda} }(\pmb{e}_v)$. Then each cell of $\Pi[ V_{\lambda};v ]$ is invariant under the action of $(\drm{Aut}{G})_v$. There are other relations enjoyed by vertices belonging to the same orbit of $(\drm{Aut}{G})_v$, which enables us refine the partition $\Pi[ V_{\lambda};v ]$. 

Let $\pmb{x} = (x_1,\ldots,x_n)^t$ be a vector of $\R^n$. We call the multiset $\{ x_1,\ldots,x_n \}$ the {\it type} of $\pmb{x}$, which is denoted by $\{ \pmb{x} \}$, and two vectors $\pmb{x}$ and $\pmb{y}$ are said to be {\it in the same type} if two multisets $\{ \pmb{x} \}$ and $\{ \pmb{y} \}$ are the same. Apparently if two vertices $x$ and $y$ are in the same orbit of $(\drm{Aut}{G})_v$ then for any eigenvalue $\lambda$ of $\dAM{G}$,  
$$
\{ \dproj{ V_{\lambda} }( \pmb{e}_x ) \} = \{ \dproj{ V_{\lambda} }( \pmb{e}_y ) \}
$$ 
and 
$$
\langle \dproj{ V_{\lambda} }( \pmb{e}_x ),\dproj{ V_{\lambda} }( \pmb{e}_v ) \rangle = \langle \dproj{ V_{\lambda} }( \pmb{e}_y ),\dproj{ V_{\lambda} }( \pmb{e}_v ) \rangle,
$$ 
so we can use these two relations to refine each cell of $\Pi[ V_{\lambda};v ]$ and then get a better approximation to $\Pi^*_v( \lambda )$. 

As well-known, $\dAM{G}$ possesses at least 3 eigenspaces except one special case that $G$ is isomorphic to $K_n$, the complete graph of order $n$. Hence we need to integrate the information represented by partitions $\{ \Pi[ V_{\lambda}; v] : \lambda \in \drm{spec}{\dAM{G}} \}$ into one equitable partition $\Pi[ \oplus_{\lambda} V_{\lambda}; v]$, which is a better approximation to $\Pi^*_v$. We present the detail of how to integrate those partitions in the 1st part of the 3rd section. 

On the other hand, by conducting the same operation for each vertex of $G$, we can obtain a family of partitions $\{ \Pi[ \oplus V_{\lambda}; v] : v \in [n] \}$. Again we should integrate those partitions into an equitable partition $\bar{\Pi}[ \oplus V_{\lambda} ]$ so that we have an approximation to $\Pi^*_G$ at this stage. As we will see in the 3rd section, we can use in most cases the cells of $\bar{\Pi}[ \oplus V_{\lambda} ]$ to split eigenspaces of $\dAM{G}$ so that each subspace singled out is invariant under the action of $\drm{Aut}{G}$. In the case that $\bar{\Pi}[ \oplus V_{\lambda} ] = \{ [n] \}$ and $\Pi[\oplus V_{\lambda} ; v]$ possesses a big cell $C_m^v$ such that $| C_m^v | > n/2$, there is a close relation among those subspaces spanned by cells of $\Pi[ \oplus V_{\lambda}; v]$. 

To be precise, let $C_1^v=\{v\},C_2^v,\ldots,C_m^v$ be the cells of $\Pi[\oplus V_{\lambda} ; v]$ such that $|C_2^v| \leq \cdots \leq |C_m^v|$ and $m\geq 3$. Set $Y_{\lambda, v} = V_{\lambda} \ominus \mathbf{R}_{\Pi[\oplus V_{\lambda} ; v] } V_{\lambda}^{G/\Pi[\oplus V_{\lambda} ; v]}$, {\it i.e.,} $Y_{\lambda, v}$ is the orthogonal complement of $\mathbf{R}_{\Pi[\oplus V_{\lambda} ; v] } V_{\lambda}^{G/\Pi[\oplus V_{\lambda} ; v]}$ in $V_{\lambda}$, where $\mathbf{R}_{\Pi[\oplus V_{\lambda} ; v] }$ stands for the characteristic matrix of the equitable partition $\Pi[\oplus V_{\lambda} ; v]$ and $V_{\lambda}^{G/\Pi[\oplus V_{\lambda} ; v]}$ is the eigenspace of $\dAM{G/\Pi[\oplus V_{\lambda} ; v]}$ corresponding to $\lambda$, and $X_{\lambda,v,m-1} = \drm{span}{\{ Y_{\lambda,v} : \cup_{i=2}^{m-1} C_i^v \}}$, which is spanned by vectors $\left\{ \dproj{Y_{\lambda,v}}( \pmb{e}_u ) : u \in \cup_{i=2}^{m-1} C_i^v \right\}$.

\begin{Lemma}\label{Lem-SeparatingBigCell} 
Suppose $\bar{\Pi}[ \oplus V_{\lambda} ]$ contains only one cell $[n]$. If $|C_m^v| > n/2$ then one of following two cases occurs. 
\begin{enumerate}

\item[i)] The subspace $\duspan{\oplus_{\lambda \hspace{0.5mm} \in \hspace{0.5mm} \drm{spce}{\dAM{G}}}X_{\lambda,v,m-1}}{C_m^v}$ is non-trivial.

\item[ii)] For any vertex $x$ of $[n] \setminus C_m^v$, $C_m^x = C_m^v$ where $C_m^x$ denotes the biggest cell of $\Pi[\oplus V_{\lambda} ; x]$.

\end{enumerate}
\end{Lemma}

As we shall see in the 2nd part of 3rd section, this lemma shows us how to assemble those subspaces spanned by cells of $\Pi[\oplus V_{\lambda} ; v]$ and accordingly how to work out $\Pi^*_v$. As a matter of fact, there are two kinds of combinatorial constructions useful in assembling those subspaces spanned by cells of $\Pi[\oplus V_{\lambda} ; v]$, which will be presented in the next two sections. In brief, we devise a deterministic algorithm by means of those properties, which solves Graph Isomorphism Problem for any graph of order $n$ in time $n^{ O( \log n ) }$ that is equal to $2^{ O\left( \log^2 n \right) }$.

\section{Blocks for $\drm{Aut}{G}$}

In the section 1.1, we have seen how to reduce the problem of determining $\Pi^*_G$ to that of determining a series of partitions of $[n]$ each of which consists of orbits of a stabilizer fixing a sequence of vertices. The key to achieving that is Theorem 3, so let us first prove the assertion. We begin with a classical result characterizing the relation between blocks and their stabilizers, which explains the reason why blocks are vitally important in finding out a generating set of $\mathfrak{G}$. 

\begin{Lemma}[Dixon and Mortimer \cite{DM}] \label{LemBlock&Stabilizers} 
Let $\mathfrak{G}$ be a permutation group acting on $[n]$, let $\mathcal{B}$ be the set of all blocks $B$ for $\mathfrak{G}$ with $b\in B\subseteq T $, where $T$ is an orbit of $\mathfrak{G}$, and let $\mathcal{S}$ be the set of all subgroups $\mathfrak{H}$ of $\mathfrak{G}$ with $\dfix{G}{b} \leq \mathfrak{H}$. Then there is a bijection $\Psi$ of $\mathcal{B}$ onto $\mathcal{S}$ defined by $\Psi(B) = \dfix{G}{B}$, and furthermore the mapping $\Psi$ is order-preserving in the sense that if $B_1$ and $B_2$ are two blocks in $\mathcal{B}$ then $B_1 \subseteq B_2$ if and only if $\Psi(B_1) \leq \Psi(B_2)$. 
\end{Lemma}

According to the relation above, one can easily see that stabilizers of blocks for $\mathfrak{G}$ play a significant role in generating the group.

\begin{Lemma}\label{KeyThmMaximalSubgroup}  
Let $\mathfrak{G}$ be a permutation group acting on $[n]$ and let $B$ be a block for $\mathfrak{G}$ which is contained in some orbit $T$ of $\mathfrak{G}$. Then $B$ is a maximal block if and only if $\dfix{G}{B}$ is a maximal subgroup of $\mathfrak{G}$.   
\end{Lemma}

Apparently, the lemma above implies that the action of $\mathfrak{G}$ on its orbit $T$ is primitive if and only if each stabilizer $\mathfrak{G}_t$ is a maximal subgroup of $\mathfrak{G}$, where $t$ is one member of $T$. Moreover, if $T = \{t_1=t,t_2,\ldots,t_{_{s}}\}$ is a non-trivial orbit of $\mathfrak{G}$, {\it i.e.,} $s\geq 2$, then $\mathfrak{G} = \langle \mathfrak{G}_t,\gamma_2,\ldots,\gamma_{s} \rangle$, where $\gamma_i \in \mathfrak{G}$ and $\gamma_i t = t_i$, $i=2,\ldots,s$. Similarly, in order to generate the stabilizer $\mathfrak{G}_t$, we first choose one of its non-trivial orbit, and then find the stabilizer $\mathfrak{G}_{t,u}$ of some element $u$ in the orbit and permutations in $\mathfrak{G}_t$ mapping $u$ to the rest of elements in the orbit. Clearly, this reduction can be proceeded repeatedly until the stabilizer resulted contains the identity only, and therefore we need at most $(n-1) + (n-2) + \cdots + 2 + 1 = n(n-1)/2$ permutations to generate $\mathfrak{G}$. 

Now let us prove the relation that
$$
\mathbf{R}_{\Pi^*_G} V_{\lambda}^{ G / \Pi^*_G } 
= 
\left( \bigcap_{t\in T} \mathbf{R}_{\Pi_t^*} V_{\lambda}^{ G / \Pi_t^* } \right) 
\bigcap 
\left( \bigcap_{j=1}^r V_1^{ \gamma_j} \right).
$$

\begin{proof}[\bf Proof to Theorem \ref{Thm-Equation-OrbitsOfAutG}]
As we have pointed out in the section 1.1, $\mathbf{R}_{\Pi_t^*} V_{\lambda}^{ G / \Pi_t^* } = V_{\lambda} \langle t \rangle$. Moreover it is easy to see that 
$$
\mathbf{R}_{\Pi^*_G} V_{\lambda}^{ G / \Pi^*_G } \subseteq 
\left( \bigcap_{t\in T} V_{\lambda} \langle t \rangle \right) 
\bigcap 
\left( \bigcap_{j=1}^r V_1^{ \gamma_j} \right).
$$

As to the opposite direction, let us take a vector $\pmb{x}$ from $\left( \cap_{t\in T} V_{\lambda} \langle t \rangle \right) \bigcap \left( \cap_{j=1}^r V_1^{ \gamma_j} \right).$ Note that $\pmb{x} \in \cap_{t\in T} V_{\lambda} \langle t \rangle$ $\Rightarrow$ $\mathfrak{G}_{\sigma t} \pmb{x} = \pmb{x}$, $\forall \sigma\in \mathfrak{G}$, and $\pmb{x} \in \cap_{j=1}^r V_1^{ \gamma_j}$ $\Rightarrow$ $\gamma_j \hspace{0.5mm} \pmb{x} = \pmb{x}$, $\forall j\in [r]$. As a result, 
$$
\langle \mathfrak{G}_{t_1},\cdots,\mathfrak{G}_{t_k},\gamma_1,\cdots,\gamma_r \rangle \pmb{x} = \pmb{x}. 
$$
On the other hand, it is plain to see that  $\mathfrak{G} = \langle \mathfrak{G}_{t_1},\cdots,\mathfrak{G}_{t_k},\gamma_1,\cdots,\gamma_r \rangle$, so $\mathfrak{G} \pmb{x} = \pmb{x}$ and thus $\pmb{x} \in \mathbf{R}_{\Pi^*_G} V_{\lambda}^{ G / \Pi^*_G }$. 
\end{proof}

Recall that our ultimate goal is to decide whether or not two given graphs $G$ and $H$ are isomorphic and in the case of being isomorphic to output one isomorphism from $G$ to $H$. It is easy to see if we have the information about $\drm{Aut}{G}$ and $\drm{Aut}{H}$, {\it i.e.,} the information about partitions $\Pi^*_G,\Pi^*_{u_1},\Pi^*_{u_1,u_2},\cdots,\Pi^*_{u_1,\ldots,u_s}$ and $\Pi^*_H,\Pi^*_{v_1},\Pi^*_{v_1,v_2},\cdots,\Pi^*_{v_1,\ldots,v_s}$, where $u_1,\ldots,u_s$ and $v_1,\ldots,v_s$ are two fastening sequences of $\drm{Aut}{G}$ and $\drm{Aut}{H}$ respectively, and the corresponding relations between cells of partitions in each pair $(\Pi^*_G,\Pi^*_H)$, $(\Pi^*_{u_1},\Pi^*_{v_1})$, $(\Pi^*_{u_1,u_2},\Pi^*_{v_1,v_2})$ and so forth, then we can efficiently achieve our goal. In the next section we shall present the algorithm $\mathscr{A}$ that enables us to reveal the information by means of $\oplus V_{\lambda}^G$ and  $\oplus V_{\lambda}^H$.

\vspace{2mm}
As one might expect, if $\mathfrak{G}$ acts on $T$ imprimitively, the structure of $\mathfrak{G}$ action on  $T$ is more colorful, which is illustrated by blocks for $\mathfrak{G}$. Moreover, it turns out that minimal blocks for $\mathfrak{G}$ are crucial for splitting eigenspaces of $\dAM{G}$ and for assembling subspaces spanned by cells of $\Pi[\oplus V_{\lambda} ; v]$. 

Now let us see how to build blocks for $\mathfrak{G}$ in virtue of partitions composed of orbits of stabilizers each of which fixes exactly one vertex of $G$. Pick two vertices $v'$ and $v''$ from $V(G)$. We can construct a bipartite graph $\dbipart{\Pi^*}{v}$ by means of two partitions $\Pi^*_{v'}$ and $\Pi^*_{v''}$: the vertex set consists of cells of $\Pi^*_{v'}$ and of $\Pi^*_{v''}$ and two vertices in the graph are adjacent if the intersection of two cells relevant is not empty. Evidently, the graph $\dbipart{\Pi^*}{v}$ is bipartite and two parts of the vertex set are made respectively up of orbits of $\mathfrak{G}_{v'}$ and of $\mathfrak{G}_{v''}$. Note that one component of $\dbipart{\Pi^*}{v}$ naturally corresponds to a subset of $[n]$, so we can use a component to indicate the subset relevant. 

There are essentially two kinds of blocks for a permutation group, and a non-trivial component $C[v']$ in $\dbipart{\Pi^*}{v}$ containing the vertex $v'$ shows us one of them.  

\begin{Lemma}\label{LemFindBlocks-1}
Let $C[x]$ be a component in $\dbipart{\Pi^*}{v}$ containing the element $x$. Then $C[x] = \langle \mathfrak{G}_{v'}, \mathfrak{G}_{v''} \rangle x$, and moreover the component $C[v']$ in $\dbipart{\Pi^*}{v}$ is a block for $\mathfrak{G}$.
\end{Lemma}
\begin{proof}[\bf Proof]
Suppose $y$ is in the subset $\langle \mathfrak{G}_{v'}, \mathfrak{G}_{v''} \rangle x$. Then $\exists$ $\sigma_1,\ldots,\sigma_m \in \mathfrak{G}_{v'}$ and $\gamma_1,\ldots,\gamma_m\in\mathfrak{G}_{v''}$ {\it s.t.,} $y = (\Pi_{i} \hspace{0.6mm} \sigma_i\gamma_i) x$. According to the definition to $\dbipart{\Pi^*}{v}$, it is easy to see that $y \in C[x]$, and thus $\langle \mathfrak{G}_{v'}, \mathfrak{G}_{v''} \rangle x \subseteq C[x]$.

By using the same argument, one can readily see that $C[x] \subseteq \langle \mathfrak{G}_{v'}, \mathfrak{G}_{v''} \rangle x$.

In accordance with the first claim, $C[v'] = \langle \mathfrak{G}_{v'}, \mathfrak{G}_{v''} \rangle v'$. To show $C[v']$ is a block, it is sufficient to prove that if $\sigma$ is a permutation in $\mathfrak{G}$ such that $\sigma \big(\langle \mathfrak{G}_{v'}, \mathfrak{G}_{v''} \rangle v'\big) \cap \big(\langle \mathfrak{G}_{v'}, \mathfrak{G}_{v''} \rangle v'\big) \neq \emptyset$, then $\sigma \big(\langle \mathfrak{G}_{v'}, \mathfrak{G}_{v''} \rangle v'\big) = \langle \mathfrak{G}_{v'}, \mathfrak{G}_{v''} \rangle v'$.

Suppose there are $\xi,\zeta \in \langle \mathfrak{G}_{v'}, \mathfrak{G}_{v''} \rangle$, {\it s.t.,} $\sigma\xi \hspace{0.6mm} v' = \zeta v'$. Then $\zeta^{-1} \sigma \xi \hspace{0.6mm} v' = v'$, so $\zeta^{-1} \sigma \xi \in \mathfrak{G}_{v'}$. Thus $\sigma \in \zeta\mathfrak{G}_{v'}\xi^{-1} \subseteq \langle \mathfrak{G}_{v'}, \mathfrak{G}_{v''} \rangle$.
\end{proof}

Although the component $C[v']$ in $\dbipart{\Pi^*}{v}$ must be a block for $\mathfrak{G}$, it is possible that $C[v']$ contains only one vertex in $V$. For instance, suppose $\mathfrak{G}$ is the automorphism group of a cube (see the diagram below), then $C[1]$ in $\dubipart{\Pi^*}{1}{8}$ contains only one vertex 1 and $C[8]$ only 8.

\begin{center}\setlength{\unitlength}{0.5bp}
\hspace{-6mm}\includegraphics[width=3.4cm]{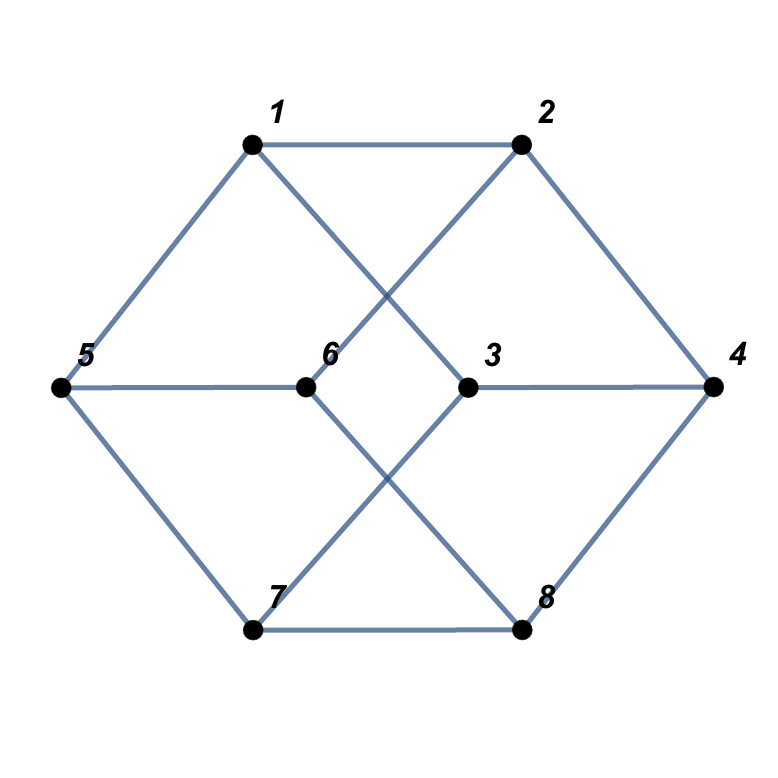}
\end{center}

In order to deal with that case, we introduce a binary relation among vertices in $T$: $x \sim y$ if $\Pi_x^*= \Pi_y^*$. Obviously, it is an equivalence relation on $T$, so it could induce a partition of $T$, which is denoted by $\widetilde{\Pi}[T]$.

\begin{Lemma}\label{LemFindBlocks-3}
All cells of $\widetilde{\Pi}[T]$ constitute a block system of $\mathfrak{G}$.
\end{Lemma}
\begin{proof}[\bf Proof]
Suppose $C_s$ is a cell of $\widetilde{\Pi}[T]$ containing the vertex $s$. We pick arbitrarily one vertex $y$ in $T \setminus C_s$. Let $\sigma$ be a permutation in $\mathfrak{G}$ such that $\sigma s = y$. Then $\sigma C_s \neq C_s$. To show $C_s$ is a block for $\mathfrak{G}$, it is sufficient to prove that $\sigma C_s \cap C_s = \emptyset$.

Note that $y\notin C_s$, which implies $\Pi^*_s \neq \Pi^*_y$. Consequently, the cell containing $s$ in $\Pi_y^*$ cannot be singleton, otherwise $\mathfrak{G}_y s = s$ $\Rightarrow$ $\mathfrak{G}_y \leq \mathfrak{G}_s$. Then $\mathfrak{G}_y = \mathfrak{G}_s $ and thus $\Pi_y^*=\Pi_s^*$, which contradicts the assumption that $y\notin C_s$. As a result, any member in $C_s$ cannot be singleton in $\Pi_y^*$. On the other hand, $\mathfrak{G}_y (\sigma x) = \left( \sigma \mathfrak{G}_s \sigma^{-1} \right) (\sigma x) = \sigma x$ for any $x \in C_s$, {\it i.e.,} $\sigma x$ is a singleton in $\Pi_y^*$. Therefore $C_s \cap \sigma C_s = \emptyset$.

It is clear that for any $\gamma \in \mathfrak{G}$, $\gamma C_s$ also belongs to $\widetilde{\Pi}[T]$ and $T = \cup_{ \gamma \in \mathfrak{G} } \gamma C_s$, so $\widetilde{\Pi}[T]$ is a block system.
\end{proof}

\begin{Theorem}\label{ThmNeceAndSuffForPrimitiveness}
Let $\mathfrak{G}$ be a permutation group of $[n]$ and let $T$ be an orbit of $\mathfrak{G}$. Then the $\mathfrak{G}$ action on $T$ is primitive if and only if one of two cases below occurs 
\begin{enumerate}
\item[{\rm i)}] $\dbipart{\Pi^*}{t}$ is connected, $\forall\hspace{0.6mm} t',t'' \in T$;

\item[{\rm ii)}] $\dbipart{\Pi^*}{t}$ is a perfect matching consisting of $|T|$ edges, $\forall\hspace{0.6mm} t',t'' \in T$, and $|T|$ is a prime number. In fact, $\mathfrak{G}$ is a circulant group of prime order in this case.
\end{enumerate}
\end{Theorem}
\begin{proof}[\bf Proof]
Let us begin with the sufficiency of our assertion. In the case i), if there exists a non-trivial block $B \subseteq T$ for $\mathfrak{G}$, then the bipartite graph $\dbipart{\Pi^*}{b}$ cannot be connected for any vertices $b'$ and $b''$ in $B$. In fact, the component $C[b']$ in $\dbipart{\Pi^*}{b}$, due to Lemma \ref{LemFindBlocks-1}, consisting of vertices in $\left\langle \mathfrak{G}_{b'},\mathfrak{G}_{b''} \right\rangle b'$, is contained in $\mathfrak{G}_B b' = B$. This is in contradiction with the assumption that $\dbipart{\Pi^*}{b}$ is connected.

Obviously, the action of $\mathfrak{G}$ on $T$ is primitive in the case ii).

\vspace{2mm} As to the necessity, one first note that there are only two possibilities for each stabilizer $\mathfrak{G}_t$: $\mathfrak{G}_t \supsetneq \{ 1 \}$ or $\mathfrak{G}_t = \{ 1 \}$. Because $\mathfrak{G}$ is primitive, the subgroup $\mathfrak{G}_t$ is maximal due to Lemma \ref{KeyThmMaximalSubgroup}. Hence for any permutation $\xi \in \mathfrak{G} \setminus \{1\}$, $\langle \xi \rangle = \langle \xi,\mathfrak{G}_t \rangle = \mathfrak{G}$ in the second case, which implies that $\mathfrak{G}$ is a circulant group of prime order.

According to Lemma \ref{LemFindBlocks-3}, $\forall\hspace{0.5mm} t',t'' \in T$, $\mathfrak{G}_{t'} \neq \mathfrak{G}_{t''}$, provided that $\mathfrak{G}_t \supsetneq \{ 1 \}$. On the other hand, the primitiveness of $\mathfrak{G}$ implies that $\langle \mathfrak{G}_{t'},\mathfrak{G}_{t''} \rangle = \mathfrak{G}$. By means of Lemma \ref{LemFindBlocks-1}, $C[t'] = \langle \mathfrak{G}_{t'}, \mathfrak{G}_{t''} \rangle t' = \mathfrak{G} t' = T$, so the graph $\dbipart{\Pi^*}{t}$ is connected.
\end{proof}

Suppose an $\mathfrak{G}$-orbit $T$ is composed of $s$ vertices $t_1,\ldots,t_s$. We can use those partitions associated with members of $T$ to construct a multipartite graph $\dumultipart{\Pi^*}{t_1}{t_s}$ with $s$ parts, which is similar to the bipartite graph $\dubipart{\Pi^*}{t_i}{t_j}$. The vertex set of the graph is the set of cells in $\cup_{i=1}^s \Pi^*_{t_i}$ and two vertices are adjacent if the two cells relevant have a non-empty intersection. Obviously, each component of $\dumultipart{\Pi^*}{t_1}{t_s}$ corresponds to a subset of $[n]$, so we can regard a component of the graph as a subset of $[n]$. Furthermore, it is not difficult to see any two members of $[n]$ belonging to distinct orbits of $\mathfrak{G}$ cannot be contained in the same component of $\dumultipart{\Pi^*}{t_1}{t_s}$. Thus we always focus one orbit of $\mathfrak{G}$ in characterizing structures of $\dumultipart{\Pi^*}{t_1}{t_s}$.

\begin{Lemma}\label{Lem-BlocksInMultipartConstruction}
Any component of the graph $\dumultipart{\Pi^*}{t_1}{t_s}$, which is contained in $T$, is a block for $\mathfrak{G}$, and the partition of $T$ induced by components of $\dumultipart{\Pi^*}{t_1}{t_s}$ is a block system of $\mathfrak{G}$.
\end{Lemma}
\begin{proof}[\bf Proof]
First of all, one can use the arguments in proving Lemma \ref{LemFindBlocks-1} to prove the first assertion. To be precise, it is easy to see that the component $C[t]$ of the multipartite graph containing the vertex $t \in T$ is the same as the subset $\dusubgroup{\mathfrak{G}}{t_1}{t_s} t$, and $\dusubgroup{\mathfrak{G}}{t_1}{t_s} t$ is a block for $\mathfrak{G}$. Consequently, $\sigma \dusubgroup{\mathfrak{G}}{t_1}{t_s} t$ is also a block for $\mathfrak{G}$ for any permutation $\sigma$ in $\mathfrak{G}$.

Moreover one can readily see that $\sigma \hspace{0.4mm} C[t]$ is contained in a component of the graph. Clearly, the vertex $\sigma t$ belongs to $\sigma \hspace{0.4mm} C[t]$. Hence 
$$
\sigma \dusubgroup{\mathfrak{G}}{t_1}{t_s} t \subseteq \dusubgroup{\mathfrak{G}}{t_1}{t_s} \sigma t,
$$
 and thus 
$\dusubgroup{\mathfrak{G}}{t_1}{t_s} t \subseteq \sigma^{-1} \dusubgroup{\mathfrak{G}}{t_1}{t_s} \sigma t.$
Note that $T = \{ \sigma t : \sigma \in \mathfrak{G} \}$, so for any $\gamma \in \mathfrak{G}$ we have 
$$
\gamma \dusubgroup{\mathfrak{G}}{t_1}{t_s} (\sigma t) \subseteq 
\dusubgroup{\mathfrak{G}}{t_1}{t_s} \gamma (\sigma t),
$$
and thus 
$\sigma^{-1}\dusubgroup{\mathfrak{G}}{t_1}{t_s}\sigma t \subseteq \dusubgroup{\mathfrak{G}}{t_1}{t_s} t.$
As a result, 
\begin{equation}\label{Equ-ComponentsBlocks}
\sigma^{-1}\dusubgroup{\mathfrak{G}}{t_1}{t_s}\sigma t = \dusubgroup{\mathfrak{G}}{t_1}{t_s} t,
\end{equation}
and therefore 
$\dusubgroup{\mathfrak{G}}{t_1}{t_s}\sigma t = \sigma \dusubgroup{\mathfrak{G}}{t_1}{t_s} t$, which means for any component $C$ of $\dumultipart{\Pi^*}{t_1}{t_s}$, one can find a permutation $\sigma$ in $\mathfrak{G}$ so that $C = \sigma \hspace{0.5mm} C[t]$. Hence the set of components of the graph restricted on $T$ forms one block system of $\mathfrak{G}$. 
\end{proof}

\begin{Theorem}\label{Thm-OrbitsOfAutG}
The graph $\dumultipart{\Pi^*}{t_1}{t_s}$ restricted on $T$ is disconnected if and only if $T$ possesses a block system $\mathscr{B}$ of $\mathfrak{G}$ such that the action of $\mathfrak{G}$ on $\mathscr{B}$ is regular.
\end{Theorem}
\begin{proof}[\bf Proof]
First of all, it is easy to see that the sufficiency of the assertion holds according to the definition to the graph $\dumultipart{\Pi^*}{t_1}{t_s}$.

Let $C$ be a subset of $T$ corresponding to some component of $\dumultipart{\Pi^*}{t_1}{t_s}$. Then $\{ \sigma C : \sigma \in \mathfrak{G} \}$ is a block system of $\mathfrak{G}$ due to Lemma \ref{Lem-BlocksInMultipartConstruction}. Thus to hold the desire, we only need to show the action of $\mathfrak{G}$ on $\{ \sigma C : \sigma \in \mathfrak{G} \}$ is regular, which is equivalent to that 
$\sigma^{-1} \mathfrak{G}_C \sigma \hspace{0.6mm} C = C$, $\forall \sigma\in\mathfrak{G}$. 

Suppose $t$ belongs to $C$. Then $C = \dusubgroup{\mathfrak{G}}{t_1}{t_s} t$, and thus 
$\dusubgroup{\mathfrak{G}}{t_1}{t_s}$ is the stabilizer of $C$. In accordance with the relation (\ref{Equ-ComponentsBlocks}), we have 
\begin{align*}
\sigma^{-1} \mathfrak{G}_C \sigma \hspace{0.6mm} C
& = \sigma^{-1} \mathfrak{G}_C \sigma \left( \dusubgroup{\mathfrak{G}}{t_1}{t_s} t \right) \\
& = \sigma^{-1} \mathfrak{G}_C \sigma \left( \sigma^{-1} \dusubgroup{\mathfrak{G}}{t_1}{t_s} \sigma t \right) \\
& = \sigma^{-1}\dusubgroup{\mathfrak{G}}{t_1}{t_s}\sigma t \\
& = \dusubgroup{\mathfrak{G}}{t_1}{t_s} t \\
&= C.
\end{align*}
\end{proof}

According to the result above, the orbit $T$ of $\mathfrak{G}$ is contained in $\dumultipart{\Pi^*}{t_1}{t_s}$ as one component unless there exists a block system $\mathscr{B}$ in $T$ on which the action of $\mathfrak{G}$ is regular.

\section{The Algorithm}

In 1982, L. Babai, D.Yu. Grigoryev and D.M. Mount presented two polynomial algorithms in the article \cite{BaGrMu}, each of which solves Graph Isomorphism Problem for graphs with bounded eigenvalue multiplicity.\footnote{In order to obtain the decomposition $\oplus V_{\lambda} = \R^n$, one needs to calculate eigenvalues and eigenspaces of $\dAM{G}$ first, the complexity of which (within a relative error bound $2^{-b}$) is bounded by $O(n^3 + (n \log^2 n) \log b)$ (see \cite{PanChenZheng} for details). } Naturally in the case that some of eigenspaces of $\dAM{G}$ are of dimension tending to infinity as $n \rightarrow \infty$, we should split those large eigenspaces into subspaces with dimension as small as possible. As shown in the section 1.2, the cells of each partition $\Pi[ V_{\lambda};v ]$ ($\lambda \in \drm{spec}{ \dAM{G} }$) can be used to split the eigenspace $V_{\lambda}$. 

In the 1st part of this section, we will show how to integrate partitions $\{ \Pi[ V_{\lambda};v ] : \lambda \in \drm{spec}{ \dAM{G} } \}$ into one partition $\Pi[ \oplus_{\lambda} V_{\lambda};v ]$ that is more effective in splitting eigenspaces of $\dAM{G}$. Moreover, we can actually integrate information contained in partitions $\{ \Pi[ \oplus V_{\lambda};v ] : v \in [n] \}$ so that two partitions $\bar{\Pi}[ \oplus V_{\lambda} ]$ and $\Pi[ \oplus V_{\lambda};\mathtt{B} ]$ resulted could reveal some global information about the structure of $\mathfrak{G}$ action on $[n]$. By means of that we assemble in the 2nd part those subspaces singled out for uncovering symmetries in $G$. In brief, by inputting the decomposition $\oplus V_{\lambda}$ of $\R^n$, our algorithm $\mathscr{A}$ outputs the information about $\mathfrak{G}$, {\it i.e,.} the partition $\Pi_G^*$ and a series of partitions of $[n]$ associated with a fastening sequence of $\mathfrak{G}$.

\subsection{Splitting Eigenspaces of $\dAM{G}$}

\noindent $\blacktriangle$ $ \Pi[ \oplus V_{\lambda};v ]$ --- an approximation to $\Pi^*_v$

\begin{enumerate}

\item Let $\pmb{x} = (x_1,\ldots,x_n)^t$ be a vector of $\R^n$. Recall that the type of the vector $\pmb{x}$ is the multiset $\{ x_1,\ldots,x_n \}$, which is denoted by $\{ \pmb{x} \}$. Apparently  if two vertices $x$ and $y$ are in the same orbit of $\mathfrak{G}_v$ then for any eigenvalue $\lambda$ of $\dAM{G}$,  
\begin{equation}\label{Equ-TypeFeature}
\{ \dproj{ V_{\lambda} }( \pmb{e}_x ) \} = \{ \dproj{ V_{\lambda} }( \pmb{e}_y ) \}
\end{equation} 
and 
\begin{equation}\label{Equ-InnerProductFeature}
\langle \dproj{ V_{\lambda} }( \pmb{e}_x ),\dproj{ V_{\lambda} }( \pmb{e}_v ) \rangle = \langle \dproj{ V_{\lambda} }( \pmb{e}_y ),\dproj{ V_{\lambda} }( \pmb{e}_v ) \rangle.
\end{equation} 
As we have seen in the introduction, there is another geometric tool also useful in determining the partition $\Pi^*_v$ --- region, so we employ all of them to work out an approximation to $\Pi^*_v$. Obviously there are two cases relevant to be dealt with.

\begin{enumerate}

\item In the case that there are some of vectors in the OPSB onto $V_{\lambda}$ that are orthogonal to $\dproj{ V_{\lambda} }( \pmb{e}_v )$, set $\mathcal{I}_0 = \{ x \in [n] : \dproj{ V_{\lambda} }( \pmb{e}_x ) \perp \dproj{ V_{\lambda} }( \pmb{e}_v ) \}$. Next, we examine types of those projections corresponding to vertices in $\mathcal{I}_0$, and then we group members of $\mathcal{I}_0$ so that two vertices $x$ and $y$ belong to the same cell if $\{ \dproj{ V_{\lambda} }( \pmb{e}_x ) \} = \{ \dproj{ V_{\lambda} }( \pmb{e}_y ) \}$. Each cell of the partition of $\mathcal{I}_0$ resulted is said to be a {\it thin cell with reference to $V_{\lambda}$}. Evidently, each cell resulted is $\mathfrak{G}_v$-invariant.

\item Let $\duspan{V_{\lambda}}{ [n] \setminus \mathcal{I}_0 }$ be the subspace spanned by vectors $\{ \dproj{ V_{\lambda} }( \pmb{e}_w ) : w \in [n] \setminus \mathcal{I}_0 \}$. Clearly, $\duspan{V_{\lambda}}{ [n] \setminus \mathcal{I}_0 }$ can be divided into regions with respect to 
$\{ \dproj{ V_{\lambda} }( \pmb{e}_w ) : w \in [n] \setminus \mathcal{I}_0 \}$. We now partition the subset $[n] \setminus \mathcal{I}_0$ by means of the region of $\duspan{V_{\lambda}}{ [n] \setminus \mathcal{I}_0 }$ which contains the vector $\dproj{ V_{\lambda} }( \pmb{e}_v )$:
 \begin{enumerate}
 
 \item Find out the incidence set of the region $R$ containing $\dproj{ V_{\lambda} }( \pmb{e}_v )$ by means of Lemma \ref{Lem-IncidenceSetRegion}.  
 
 \item Group vertices in $\mathcal{I}_R$ according to their types and angles relevant, {\it i.e.,} two vertices $x$ and $y$ belong to the same group if they enjoy the relations (\ref{Equ-TypeFeature}) and  (\ref{Equ-InnerProductFeature}).
  
 \end{enumerate}
 
Again, it is obvious that each cell resulted is $\mathfrak{G}_v$-invariant. Delete the subset $\mathcal{I}_R$ from $[n] \setminus \mathcal{I}_0$ and partition the rest of vertices by means of the region of $\duspan{V_{\lambda}}{ [n] \setminus ( \mathcal{I}_0 \cup \mathcal{I}_R ) }$, which is carved up by $\{ \dproj{ V_{\lambda} }( \pmb{e}_w )^{\perp} : w \in [n] \setminus ( \mathcal{I}_0 \cup \mathcal{I}_R ) \}$ and contains the vector $\dproj{ V_{\lambda} }( \pmb{e}_v )$. Repeat the procedure above so that we finally obtain a partition of $[n]\setminus \mathcal{I}_0$. 
\end{enumerate}

The partition of $[n]$ obtained in above way is denoted by $\Pi[ V_{\lambda};v ]$.

\vspace{2mm}
Let $V_{\lambda,[n]\setminus \mathcal{I}_0}$ denote $\duspan{V_{\lambda}}{ [n] \setminus \mathcal{I}_0 }$. It is plain to verify that $\duspan{V_{\lambda,[n]\setminus \mathcal{I}_0}}{\mathcal{I}_{R}} = V_{\lambda,[n]\setminus \mathcal{I}_0}$, where $R$ is the region of $V_{\lambda,[n]\setminus \mathcal{I}_0}$ containing $\dproj{ V_{\lambda} }( \pmb{e}_v )$. This fact is quite useful in splitting big cells as we shall see in the next part.

\item Note that a partition $\Pi[ V_{\lambda};v ]$ is related to the eigenspace $V_{\lambda}$, so we can use all those partitions to obtain a global one $\Pi[ \oplus V_{\lambda};v ] := \cap_{\lambda} \Pi[ V_{\lambda};v ]$. Let $\Pi_1$ and $\Pi_2$ be two partitions of $[n]$. Then 
$$
\Pi_1 \cap \Pi_2 = \{ C_{1i} \cap C_{2j} : C_{1i} \in \Pi_1 \mbox{ and } C_{2j} \in \Pi_2 \}.
$$

\item Let $C$ be a cell of $\Pi[ \oplus V_{\lambda};v ]$ which is not a singleton, and set $V_{\lambda,C} = \duspan{V_{\lambda}}{C}$ that is  the subspace spanned by $\{ \dproj{ V_{\lambda} }( \pmb{e}_x ) : x \in C \}$. Recall that $V_{\lambda} \langle v \rangle = \{ \pmb{u} \in V_{\lambda} : \xi \hspace{0.5mm} \pmb{u} = \pmb{u}, ~ \forall \xi\in\mathfrak{G}_v  \}$. A moment's reflection would show that if $C$ is an orbit of $\mathfrak{G}_v$ then 
\begin{equation}\label{Equ-OrbitPointStabilizerAutG}
\sum_{ x \in C} \dproj{ V_{\lambda,C} \ominus V_{\lambda} \langle v \rangle }( \pmb{e}_x ) = \pmb{0},
\end{equation}
where $V_{\lambda,C} \ominus V_{\lambda} \langle v \rangle$ stands for the orthogonal complement of $V_{\lambda,C} \cap V_{\lambda} \langle v \rangle$ in $V_{\lambda,C}$. Accordingly, we can give a further check on $\Pi[ \oplus V_{\lambda};v ]$. Note also that $\Pi_v^*$ consisting of orbits of $\mathfrak{G}_v$ is an equitable partition, so we first refine $\Pi[ \oplus V_{\lambda};v ]$ by virtue of Lemma \ref{Lemma-EquitablePartProj} so that the partition resulted is equitable, which is denoted still by $\Pi[ \oplus V_{\lambda};v ]$. Next we refine each cell of $\Pi[ \oplus V_{\lambda};v ]$ further with a relation similar to (\ref{Equ-OrbitPointStabilizerAutG}).

\begin{enumerate}

\item If $C$ is thin when embedded in the subspace $V_{\lambda,C}$, {\it i.e.,} 
$\displaystyle \dproj{ V_{\lambda,C} }( \pmb{e}_v ) \in \bigcap_{x\in C} \dproj{ V_{\lambda,C} }( \pmb{e}_x )^{\perp}$, then it is said to be {\it balanced} if the sum vector 
$$\sum_{ x \in C} 
\dproj{ 
V_{\lambda,C} \ominus V_{\lambda,\Pi_v} 
}( \pmb{e}_x ) = \pmb{0},$$
where $V_{\lambda,\Pi_v}$ stands for the subspace $\mathbf{R}_{\Pi[ \oplus V_{\lambda};v ]} V_{\lambda}^{G/\Pi[ \oplus V_{\lambda};v ]}$ and $V_{\lambda,C} \ominus V_{\lambda,\Pi_v}$ denotes the orthogonal complement of $V_{\lambda,C} \cap V_{\lambda,\Pi_v}$ in $V_{\lambda,C}$.

\item In the case that $C$ is not thin when embedded in $V_{\lambda,C}$, it is said to be {\it balanced} if for any two vertices $u$ and $w$ belonging to $C$, 
$$\left\langle \dproj{ V_{\lambda,C} }( \pmb{e}_u ),\pmb{i}_{R_{\lambda}[C]} \right\rangle 
= \left\langle \dproj{ V_{\lambda,C} }( \pmb{e}_w ),\pmb{i}_{R_{\lambda}[C]} \right\rangle$$
and the sum vector 
$$
\sum_{ x \in C} 
\dproj{ 
V_{\lambda,C} \ominus 
\big( 
V_{\lambda,\Pi_v} \oplus \drm{span}{\{ \pmb{i}_{R_{\lambda}[C]} \}}
\big)
}( \pmb{e}_x )  = \pmb{0},
$$ 
where $R_{\lambda}[C]$ is the region of $V_{\lambda,C}$ with respect to $\{ \dproj{ V_{\lambda,C} }( \pmb{e}_x ) : x \in C \}$ such that the incidence set $\mathcal{I}_{ R_{\lambda}[C] }$ is $C$ and $\pmb{i}_{R_{\lambda}[C]}$ is the indicator of $R_{\lambda}[C]$.
\end{enumerate}

In the case that a thin cell $C$, when embedded in $V_{\lambda,C}$, is not balanced, we refine $C$ further through a series of regions relevant to 
$V_{\lambda,C} \ominus V_{\lambda,\Pi_v}$ with respect to $\{ \dproj{ V_{\lambda,C} \ominus V_{\lambda,\Pi_v} }( \pmb{e}_x ) : x \in C \}$, each of which contains the sum vector above. The process of doing so is the same as we group vertices of $G$ through a series of regions relevant to $V_{\lambda}$ containing $\dproj{ V_{\lambda} }( \pmb{e}_v )$, {\it i.e.,} the process of working out $\Pi[ V_{\lambda};v ]$.

In general case, if a cell $C$ is not balanced when embedded in $V_{\lambda,C}$, then we first refine $C$ according to inner products $\{ \langle \dproj{ V_{\lambda,C} }( \pmb{e}_x ),\pmb{i}_{R_{\lambda}[C]} \rangle : x \in C \}$ and then to the sum vector involved through the process that is the same as what we did in dealing with a thin cell.

\end{enumerate}

Apparently, it is possible that after having carried out the operation iii), some of cells of the resulted partition $\Pi[ \oplus V_{\lambda};v ]$ violate the relations (\ref{Equ-TypeFeature}) and (\ref{Equ-InnerProductFeature}), where $\dproj{ V_{\lambda} }( \pmb{e}_v )$ is replaced with the sum vector relevant, or $\Pi[ \oplus V_{\lambda};v ]$ is not equitable now. Then we go back and carry out the operations i), ii) and iii) again. Repeat this procedure so that the resulted partition cannot be refined further through those three operations, and then we call the partition output a {\it balanced partition} of $[n]$ and denoted it still by $\Pi[ \oplus V_{\lambda};v ]$.

\vspace{2mm}
Let $S$ be a subset of $[n]$. Set $V_{\lambda,S} = \duspan{V_{\lambda}}{S}$, where $\lambda\in\drm{spec}{\dAM{G}}$ . We say $S$ forms a {\it complete configuration} if $\forall \lambda\in\drm{spec}{ \dAM{G} }$ and $\forall s \in S$,  
$$
\langle \dproj{ V_{\lambda,S} }( \pmb{e}_s ),\dproj{ V_{\lambda,S} }( \pmb{e}_x ) \rangle =
\langle \dproj{ V_{\lambda,S} }( \pmb{e}_s ),\dproj{ V_{\lambda,S} }( \pmb{e}_y ) \rangle,
$$
for any two members $x$ and $y$ in $S\setminus\{s\}$. One can readily see that if $S$ is a complete configuration then the action of $\drm{Aut}{ \ReSt{\oplus_{\lambda} V_{\lambda,S}}{S}}$ on $S$ is the same as the action of $\drm{Sym}{ S }$ on $S$, where $\drm{Aut}{ \ReSt{\oplus_{\lambda} V_{\lambda,S}}{S}}$ stands for the permutation group of $S$ that preserves each $V_{\lambda,S}$ invariant. For instance, if $[n]$ itself forms a complete configuration then every partition $\Pi[ \oplus V_{\lambda};v ]$ ($v\in [n]$) possesses only two cells $\{v\}$ and $[n]\setminus\{v\}$, and thus $\drm{Aut}{G} \cong \drm{Sym}{[n]}$, {\it i.e.,} $G$ is the complete graph of order $n$.

\vspace{3mm}
\noindent {$\blacktriangle$ $\bar{\Pi}[ \oplus V_{\lambda} ]$} --- whether or not belonging to the same orbit of $\mathfrak{G}$ 

\vspace{1mm}
As one can easily see, we actually use $\Pi[ \oplus V_{\lambda};v ]$ at this stage to approximate $\Pi^*_v$, so we can use the family $\left\{ \Pi[ \oplus V_{\lambda};v ] : v \in [n] \right\}$ to build a partition of $[n]$ close to $\Pi_G^*$ that is composed of orbits of $\mathfrak{G}$.

Clearly, if two vertices $u$ and $v$ are in the same orbit of $\mathfrak{G}$ then there is an automorphism $\sigma$ such that $\sigma \Pi_u^* = \Pi_v^*$, {\it i.e.,} there is a bijection between cells of $\Pi_u^*$ and of $\Pi_v^*$. On the other hand, it is not difficult to verify that the way of splitting $[n]$ and working out $\Pi[ \oplus V_{\lambda};u ]$ and $\Pi[ \oplus V_{\lambda};v ]$ induces a corresponding relation between cells of two partitions, which is denoted by $\phi_{uv}$. Accordingly we define a binary relation among vertices of $G$: $u \leftrightarrow v$ if $\forall\hspace{0.4mm} C_i,C_j \in \Pi[ \oplus V_{\lambda};u ]$, and $\forall \hspace{0.4mm} \lambda \in \drm{spec}{ \dAM{G} }$, 
\begin{equation}\label{Equ-OrbitsApproximation}
\ReSt{\{ \dproj{V_{\lambda}}( \pmb{R}_{C_i} ) \}}{C_j} = \ReSt{\{ \dproj{V_{\lambda}}( \pmb{R}_{\phi_{uv} C_i} ) \}}{\phi_{uv} C_j},
\end{equation}
where $\ReSt{\{ \dproj{V_{\lambda}}( \pmb{R}_{C_i} ) \}}{C_j}$ stands for the subset of $\{ \dproj{V_{\lambda}}( \pmb{R}_{C_i} ) \}$ consisting of coordinates of the vector $\dproj{V_{\lambda}}( \pmb{R}_{C_i} )$ corresponding to the subset $C_j$.
Clearly, the relation `$\leftrightarrow$' is an equivalence one, so we have a partition of $[n]$, which is denoted by $\Pi[ \oplus V_{\lambda} ]$. One can readily see that each cell of $\Pi[ \oplus V_{\lambda} ]$ is a union of some of orbits of $\mathfrak{G}$.

Note that the partition $\Pi_G^*$ is equitable, so we refine the partition $\Pi[ \oplus V_{\lambda} ]$ by means of Lemma \ref{Lemma-EquitablePartProj} so that it is equitable, and the partition resulted is denoted by $\bar{\Pi}[ \oplus V_{\lambda} ]$. An equitable partition is said to be {\it uniform} if the relation (\ref{Equ-OrbitsApproximation}) holds for any two vertices belonging to the same cell of the partition. Hence $\bar{\Pi}[ \oplus V_{\lambda} ]$ is an uniform partition. Again it is easy to check that each cell of $\bar{\Pi}[ \oplus V_{\lambda} ]$ is an union of some of orbits of $\mathfrak{G}$.

\vspace{2mm}
\noindent {$\blacktriangle$ $\Pi[ \oplus V_{\lambda};\mathtt{B} ]$} --- whether or not belonging to the same minimal block for $\mathfrak{G}$

\vspace{1mm}
Recall that Theorem \ref{ThmNeceAndSuffForPrimitiveness} provides an efficient way of finding out minimal blocks for $\mathfrak{G}$ by means of orbits of stabilizers each of which fixes exactly one vertex of $G$, so we could employ this tool to refine the partition $\bar{\Pi}[ \oplus V_{\lambda} ]$.

\begin{enumerate}

\item Select arbitrarily one vertex $x$ from some cell $S$ of $\bar{\Pi}[ \oplus V_{\lambda} ]$, and verify whether or not for any vertex $y \in S \setminus \{ x \}$, the bipartite graph $\big[ \Pi[ \oplus V_{\lambda};x ],\Pi[ \oplus V_{\lambda};y ] \big]$, when restricted to $S$, is connected or comprised of a perfect matching. If it is not the case for some vertex $y$ in $S$, we can refine $S$ in the following two ways:
\begin{enumerate}

\item[1)] In the case that $\Pi[ \oplus V_{\lambda};x ] = \Pi[ \oplus V_{\lambda};y ]$, we can use the subset $\mathtt{B}$ that is composed of singletons of $\Pi[ \oplus V_{\lambda};x ]$ contained in $S$ to split $S$. 

\item[2)] In the case that the vertex $x$ belongs to some non-trivial component of the bipartite graph $\big[ \Pi[ \oplus V_{\lambda};x ],\Pi[ \oplus V_{\lambda};y ] \big]$, we examine all such components for every $y\in S$ and use one of those  components $\mathtt{B}$ of minimum order to split $S$.

\end{enumerate}

It is obvious that if both of two cases occur, we should select the subset $\mathtt{B}$ of minimum order to split $S$.

\item Notice that $\mathtt{B}$ is used at this stage to approximate a minimal block for $\mathfrak{G}$, so we can use some feature enjoyed by minimal blocks to give a further check on $\mathtt{B}$. Set
$$
\ReSt{\Pi[ \oplus V_{\lambda};x ]}{\mathtt{B}} = \{ C \in \Pi[ \oplus V_{\lambda};x ] : C \subseteq \mathtt{B} \}.
$$ 
Recall that $\Pi[ \oplus V_{\lambda};x ]$ is an approximation to $\Pi^*_x$, so according to Lemma \ref{LemFindBlocks-1} and \ref{LemFindBlocks-3}, there are two cases:

\begin{enumerate}

\item[(a)] The partition $\ReSt{\Pi[ \oplus V_{\lambda};x ]}{\mathtt{B}}$ is composed of a number of singletons, {\it i.e.,} each cell of $\ReSt{\Pi[ \oplus V_{\lambda};x ]}{\mathtt{B}}$ contains exactly one member. Note that we know the corresponding relation between cells of $\ReSt{\Pi[ \oplus V_{\lambda};x ]}{\mathtt{B}}$ and of $\ReSt{\Pi[ \oplus V_{\lambda};y ]}{\mathtt{B}}$, $\forall y \in \mathtt{B} \setminus \{ x \}$, due to the process of outputting those two partitions $\Pi[ \oplus V_{\lambda};x ]$ and $\Pi[ \oplus V_{\lambda};y ]$. Thus we can easily figure out a group of permutations of $\mathtt{B}$, which is denoted by $\mathfrak{P}$. It is plain to see that each permutation in $\mathfrak{P}$ can actually be regarded as an operator on the subspace $\oplus_{\lambda} \hspace{0.5mm} \duspan{V_{\lambda}}{\mathtt{B}}$, where $\duspan{V_{\lambda}}{\mathtt{B}}$ is spanned by vectors $\{ \dproj{ V_{\lambda} }(\pmb{e}_u) : u \in \mathtt{B} \}$, for it naturally acts on vectors $\left\{ \dproj{ V_{\lambda} }(\pmb{e}_u) : u \in \mathtt{B} \right\}$ in the way defined as (\ref{Def-PermutationOperator}). Hence by checking whether or not each subspace $\duspan{V_{\lambda}}{\mathtt{B}}$ ($\lambda\in\drm{spec}{\dAM{G}}$) is invariant under the action of $\mathfrak{P}$, we can easily determine the group $\drm{Aut}{\ReSt{\oplus_{\lambda}\hspace{0.5mm} \duspan{V_{\lambda}}{\mathtt{B}}}{\mathtt{B}}}$ and the structure of its action, where $\drm{Aut}{\ReSt{\oplus_{\lambda}\hspace{0.5mm} \duspan{V_{\lambda}}{\mathtt{B}}}{\mathtt{B}}}$ stands for the permutation group of $\mathtt{B}$ that preserves every subspace 
$\duspan{V_{\lambda}}{\mathtt{B}}$ invariant, $\lambda\in\drm{spec}{\dAM{G}}$. 

On the other hand, if $B$ is a minimal block for $\mathfrak{G}$, then the action of $\mathfrak{G}_{B}$ on $B$ is primitive. Accordingly, if the action of $\drm{Aut}{\ReSt{\oplus_{\lambda}\hspace{0.5mm} \duspan{V_{\lambda}}{\mathtt{B}}}{\mathtt{B}}}$ on $\mathtt{B}$ is not primitive, we select one of minimal blocks for $\drm{Aut}{\ReSt{\oplus_{\lambda}\hspace{0.5mm} \duspan{V_{\lambda}}{\mathtt{B}}}{\mathtt{B}}}$ and denote it by $\mathtt{B}$.

\item[(b)] For any vertex $y$ in $\mathtt{B} \setminus\{ x \}$, the bipartite graph $\big[ \Pi[ \oplus V_{\lambda};x ],\Pi[ \oplus V_{\lambda};y ] \big]$, when restricted on $\mathtt{B}$, is connected. In order to 
decide whether or not $\mathtt{B}$ is a good approximation in this case, we construct a directed graph $\mathrm{PBG}(\mathtt{B})$ and check if it enjoys a simple feature. 

First of all, let us present one fundamental property that should be enjoyed by the graph we shall construct. It is clear that if $x$ and $y$ belong to the same orbit of $\mathfrak{G}$, there is a corresponding relation between cells of $\Pi_x^*$ and of $\Pi_y^*$. In fact, suppose $T_x$ is an orbit of $\mathfrak{G}_x$ and $\sigma$ is a permutation in $\mathfrak{G}$ so that $\sigma x = y$. Then $\sigma T_x$ belongs to $\Pi_y^*$, and if $\gamma x = y$ ($\gamma \in \mathfrak{G}$) then $\gamma T_x = \sigma T_x$. Obviously, that map from $\Pi_x^*$ to $\Pi_y^*$ is a one to one correspondence which we use to construct a direct graph associated with a minimal block for $\mathfrak{G}$.

Let $K$ be a minimal block for $\mathfrak{G}$ and $b$ a member of $K$. Apparently $K = \{ \sigma b : \sigma \in \mathfrak{G}_K \}$. Let $T_b$ be an orbit of $\mathfrak{G}_b$ which is contained in $K$. The {\it block graph} $\mathrm{BG}(K)$ with the pair $(K,\{ \sigma T_b : \sigma \in \mathfrak{G}_K \})$ possesses the vertex set $K$, and there is an arc from $\alpha b$ to $\beta b$, {\it i.e.,} $\alpha b \rightarrow \beta b$, if $\beta b$ is in $\alpha\hspace{0.4mm} T_b$, where $\alpha$ and $\beta$ belong to $\mathfrak{G}_K$. Suppose $w \in T_b$ such that $\beta b = \alpha w$. Note that for any permutation $\gamma \in \mathfrak{G}_K$, 
$$
\beta \hspace{0.4mm} b = \alpha \hspace{0.4mm} w \Leftrightarrow \gamma (\beta b) = \gamma(\alpha \hspace{0.4mm} w) 
\Leftrightarrow \gamma \beta b \in \gamma\alpha\hspace{0.4mm} T_b, 
\mbox{ so } \gamma\alpha\hspace{0.4mm} b \rightarrow \gamma\beta b 
\mbox{ by definition. }
$$
Hence $\gamma$ is an automorphism  of the direct graph $\mathrm{BG}(K)$. Consequently if $\sigma \in \mathfrak{G}_K$ {\it s.t.,} $\sigma b \in T_b$ then $b \rightarrow \sigma b \rightarrow \sigma^2 b \rightarrow \sigma^3 b \rightarrow \cdots \rightarrow b$, which implies that there is a strong component in $\mathrm{BG}(K)$.
Moreover, it is easy to check that any strong component of $\mathrm{BG}(K)$ is a block for $\mathfrak{G}$. In fact, suppose $P$ is a strong component of the graph and $\gamma$ a permutation in $\mathfrak{G}_K$. Then $\gamma P$ is also a strong component, and thus $\gamma P \cap P \neq \emptyset$ $\Rightarrow$ $\gamma P = P$. As a result $\mathrm{BG}(K)$ is strong connected since $K$ is a minimal block for $\mathfrak{G}$.

We are now ready to build the direct graph $\mathrm{PBG}(\mathtt{B})$, called {\it pseudo-block graph}, which is similar to $\mathrm{BG}(K)$. Suppose $E(x)$ is a cell of $\ReSt{\Pi[ \oplus V_{\lambda};x ]}{\mathtt{B}}$. The graph $\mathrm{PBG}(\mathtt{B})$ has the vertex set $\mathtt{B}$, and its arc set is determined by $\{ \phi_{xy} E(x) : y \in \mathtt{B} \}$, where $\phi_{xy}$ stands for the corresponding relation between cells of two partitions $\Pi[ \oplus V_{\lambda};x ]$ and $\Pi[ \oplus V_{\lambda};y ]$ induced by the procedure of outputting those two partitions. More precisely there is an arc from $u$ to $v$, {\it i.e.,} $u \rightarrow v$, if $v$ is in $\phi_{xu}\hspace{0.4mm} E(x)$.

Note that $\mathrm{PBG}(\mathtt{B})$ can be constructed with any cell $E(x)$ of $\ReSt{\Pi[ \oplus V_{\lambda};x ]}{\mathtt{B}}$, which is not equal to $\{x\}$, so we select one of them with minimum order to build the graph. Since $\Pi[ \oplus V_{\lambda};x ]$ is an approximation to $\Pi^*_x$, the direct graph $\mathrm{PBG}(\mathtt{B})$ would be strong connected. If it is not the case, we split $\mathtt{B}$ into pieces corresponding to strong connected components of $\mathrm{PBG}(\mathtt{B})$ and select one of components of minimal order as $\mathtt{B}$.

\vspace{2mm}
Recall that  if a subset $S$ of $[n]$ forms a complete configuration then the action of the group $\drm{Aut}{ \ReSt{\oplus_{\lambda} V_{\lambda,S}}{S}}$ on $S$ is the same as the action of $\drm{Sym}{ S }$ on $S$. Hence, it is not necessary to construct $\mathrm{PBG}(\mathtt{B})$ for revealing the structure of 
$\drm{Aut}{ \ReSt{ \oplus_{\lambda} \duspan{V_{\lambda}}{\mathtt{B}} }{\mathtt{B}} }$ action if $\mathtt{B}$ forms a complete configuration, and therefore
\begin{equation*}
\mbox{ the cell } E(x) \mbox{ we select to build } \mathrm{PBG}(\mathtt{B}) \mbox{ must be of order less than }|\mathtt{B}|/2.
\end{equation*}
 
\end{enumerate}

\item Clearly if $B$ is a non-trivial block for $\mathfrak{G}$, then the partition $\Pi^*_{B}$ consisting of orbits of $\mathfrak{G}_{B}$ is equitable, for the stabilizer $\mathfrak{G}_{B}$ of $B$ is a subgroup of $\mathfrak{G}$. Set $Y_{\lambda,\Pi_G^*} = V_{\lambda} \ominus \mathbf{R}_{ \Pi^*_G } V_{\lambda}^{ G / \Pi^*_G }$, {\it i.e.,} $Y_{\lambda,\Pi_G^*}$ is the orthogonal complement of $\mathbf{R}_{ \Pi^*_G } V_{\lambda}^{ G / \Pi^*_G }$ in $V_{\lambda}$ ($\lambda \in \drm{spec}{\dAM{G}}$). Since $\Pi^*_{B}$ is a proper refinement of $\Pi^*_G$ that is also an equitable partition, the characteristic vector $\pmb{R}_{B}$ of $B$ has a non-trivial projection onto the subspace $Y_{\lambda,\Pi_G^*}$ for some $\lambda \in \drm{spec}{ \dAM{G} }$, {\it i.e.,}
$$
\dproj{ Y_{\lambda,\Pi_G^*} }( \pmb{R}_{B} ) \neq \pmb{0}.
$$
We can use this feature and process of outputting the partition $\Pi[ \oplus V_{\lambda};v ]$ to refine $\bar{\Pi}[ \oplus V_{\lambda} ]$ properly, provided that $\mathtt{B} \subsetneq S$ where $S$ is the cell of $\bar{\Pi}[ \oplus V_{\lambda} ]$ we select at the first step.

Recall that $\bar{\Pi}[ \oplus V_{\lambda} ]$ is an equitable partition, so we first list those eigenspaces $V_{\lambda}$ of $\dAM{G}$ such that 
$$
\pmb{p}_{\lambda,\mathtt{B}} :=
 \dproj{ Y_{\lambda,\bar{\Pi}[ \oplus V_{\lambda} ]} }
( \pmb{R}_{\mathtt{B}} ) \neq \pmb{0},
$$ where $Y_{\lambda,\bar{\Pi}[ \oplus V_{\lambda} ]} = V_{\lambda} \ominus \mathbf{R}_{ \bar{\Pi}[ \oplus V_{\lambda} ] } V_{\lambda}^{ G / \bar{\Pi}[ \oplus V_{\lambda} ] }$. For each such subspace $Y_{\lambda,\bar{\Pi}[ \oplus V_{\lambda} ]}$, one can use those three tests for obtaining $\Pi[ V_{\lambda};v ]$ to construct a partition $\Pi[ V_{\lambda};\mathtt{B} ]$ with the vector $\dproj{ V_{\lambda} }( \pmb{e}_v )$ replaced by $\pmb{p}_{\lambda,\mathtt{B}}$. Then we use those three operations for working out $ \Pi[ \oplus V_{\lambda};v ]$ to build a balanced partition that is denoted by $\Pi[ \oplus V_{\lambda};\mathtt{B} ]$.

\end{enumerate}

In summary, 
\begin{equation}\label{Computation-I}
\oplus_{\lambda} V_{\lambda} \longrightarrow \left\{ \Pi[ \oplus V_{\lambda};v ] : v \in V(G) \right\} \longrightarrow \bar{\Pi}[ \oplus V_{\lambda} ] ~ \& ~ \Pi[ \oplus V_{\lambda};\mathtt{B} ].
\end{equation}

\vspace{2mm}
Now let us see how to use those partitions we have erected to decompose eigenspaces of $\dAM{G}$. We first decompose each $V_{\lambda}$ ($\lambda \in \drm{spec}{\dAM{G}}$) by the uniform partition $\bar{\Pi}[ \oplus V_{\lambda} ]$:
\begin{equation*}
V_{\lambda} = 
\mathbf{R}_{\bar{\Pi}[ \oplus V_{\lambda} ]} V_{\lambda}^{G/\bar{\Pi}[ \oplus V_{\lambda} ]} \oplus
Y_{\lambda,\bar{\Pi}[ \oplus V_{\lambda} ]}.
\end{equation*}
Apparently, $\mathbf{R}_{\bar{\Pi}[ \oplus V_{\lambda} ]} V_{\lambda}^{G/\bar{\Pi}[ \oplus V_{\lambda} ]}$ is an $\mathfrak{G}$-invariant subspace, so is $Y_{\lambda,\bar{\Pi}[ \oplus V_{\lambda} ]}$. As a result, in order to uncover the structure of the $\mathfrak{G}$ action on $V_{\lambda}$, we need to decompose the subspace $Y_{\lambda,\bar{\Pi}[ \oplus V_{\lambda} ]}$ further. Suppose $S_1,\ldots,S_t$ are cells of $\bar{\Pi}[ \oplus V_{\lambda} ]$ such that $|S_1| \leq \cdots \leq |S_t|$. Let $X_{\lambda,S_i}$ ($i = 1,\ldots,t$) be the subspace $\duspan{Y_{\lambda,\bar{\Pi}[ \oplus V_{\lambda} ]}}{S_i}$ that is spanned by vectors $\{ \dproj{Y_{\lambda,\bar{\Pi}[ \oplus V_{\lambda} ]}}(\pmb{e}_{x_i}) : x_i \in S_i \}$. Obviously, each cell $S_i$ is invariant under the action of $\mathfrak{G}$, so is the subspace $X_{\lambda,S_i}$ according to Lemma \ref{ProjOperatorCommutative}. Consequently, $\mathfrak{G} = \cap_{i=1}^t \hspace{0.5mm} \drm{Aut}{\oplus_{\lambda} X_{\lambda,S_i}}$, where $\drm{Aut}{\oplus_{\lambda} X_{\lambda,S_i}}$ stands for the permutation group of $[n]$ such that each subspace $X_{\lambda,S_i}$ ($\lambda \in \drm{spec}{\dAM{G}}$) is invariant under the action of those permutations contained in the group. In accordance with our definition of the subspace $X_{\lambda,S_i}$, $\drm{Aut}{\oplus_{\lambda} X_{\lambda,S_i}}$ is determined by its action on $S_i$. Hence, we can deal with subspaces $\oplus_{\lambda} X_{\lambda,S_1},\cdots,\oplus_{\lambda} X_{\lambda,S_t}$ one by one. In fact, there is a simple relation among those subspaces that can simplify our work.

\begin{Lemma}\label{Lem-SeparatingPseudoOrbit} 
Let $\Pi$ be an equitable partition and $C_1$ and $C_2$ two cells of $\Pi$ none of that is a singleton. Suppose $Y_{\lambda,\Pi} = V_{\lambda} \ominus \mathbf{R}_{\Pi} V_{\lambda}^{G/\Pi}$, {\it i.e.,} $Y_{\lambda,\Pi}$ is the orthogonal complement of $\mathbf{R}_{\Pi} V_{\lambda}^{G/\Pi}$ in $V_{\lambda}$. Then for any two vertices $u_2,v_2$ belonging to  $C_2$,
$\dproj{ \duspan{Y_{\lambda,\Pi}}{C_1}  }(\pmb{e}_{u_2}) = \dproj{ \duspan{Y_{\lambda,\Pi}}{C_1}  }(\pmb{e}_{v_2})$ if and only if $\drm{span}{ \{ Y_{\lambda,\Pi} : C_1 \} } \perp \drm{span}{ \{ Y_{\lambda,\Pi} : C_2 \} }$, where the subspace $\drm{span}{ \{ Y_{\lambda,\Pi} : C_i \} }$ ($i=1,2$) is spanned by vectors $\{ \dproj{ Y_{\lambda,\Pi} }( \pmb{e}_{x_i} ) : x_i \in C_i \}$.
\end{Lemma}
\begin{proof}[\bf Proof]
First of all, one should note that because $\Pi$ is equitable, our assumption that $|C_i| \geq 2$ ($i=1,2$) implies that $\drm{span}{ \{ Y_{\lambda,\Pi} : C_i \} }$ is not trivial for some eigenvalue $\lambda$, {\it i.e.,} $\drm{span}{ \{ Y_{\lambda,\Pi} : C_i \} } \neq \pmb{0}$.

It is easy to see the sufficiency is true, since if $\drm{span}{ \{ Y_{\lambda,\Pi} : C_1 \} } \perp \drm{span}{ \{ Y_{\lambda,\Pi} : C_2 \} }$ then $\dproj{ \duspan{Y_{\lambda,\Pi}}{C_1}  }(\pmb{e}_{y_2}) = \pmb{0}$, $\forall y_2 \in C_2$.

As to the necessity, the key fact is that for any vertex $x_1 \in C_1$, $\langle \dproj{Y_{\lambda,\Pi}}(\pmb{e}_{x_1}),\pmb{R}_{C_2} \rangle = 0$, for $\dproj{Y_{\lambda,\Pi}}(\pmb{e}_{x_1}) \perp \mathbf{R}_{\Pi} V_{\lambda}^{G/\Pi}$. Notice that our assumption is equivalent to that for any two members $u_2$ and $v_2$ of $C_2$,
$$
\langle \dproj{Y_{\lambda,\Pi}}(\pmb{e}_{x_1}),\pmb{e}_{u_2} \rangle = 
\langle \dproj{Y_{\lambda,\Pi}}(\pmb{e}_{x_1}),\pmb{e}_{v_2} \rangle,
$$ so we have 
$$
\langle \dproj{Y_{\lambda,\Pi}}(\pmb{e}_{x_1}),|C_2| \cdot \pmb{e}_{u_2} \rangle =
\langle \dproj{Y_{\lambda,\Pi}}(\pmb{e}_{x_1}), \sum_{z_2 \in C_2} \pmb{e}_{z_2} \rangle =
\langle \dproj{Y_{\lambda,\Pi}}(\pmb{e}_{x_1}), \pmb{R}_{C_2} \rangle = 0.
$$
Consequently, $\langle \dproj{Y_{\lambda,\Pi}}(\pmb{e}_{x_1}),\pmb{e}_{y_2} \rangle = 0$, $\forall y_2 \in C_2$.
\end{proof}

In accordance with Lemma \ref{Lem-SeparatingPseudoOrbit}, if there is an eigenvalue $\lambda$ of $\dAM{G}$ so that $X_{\lambda,S_i}$ is not orthogonal to  $X_{\lambda,S_j}$ $(i < j)$, $S_j$ must be split into as least two parts due to projections $\{ \dproj{X_{\lambda,S_i}}( \pmb{e}_{x_j} ) : x_j \in S_j \}$, so we can first work out the group $\drm{Aut}{\oplus_{\lambda} X_{\lambda,S_i}}$ and then use the information to find symmetries represented in $\oplus_{\lambda} X_{\lambda,S_j}$. The detail of that process will be presented in the next part. As a matter of fact, we may also employ $S_j$ to split $X_{\lambda,S_i}$ in this case, for the subspace $\duspan{X_{\lambda,S_i}}{S_j}$ is $\mathfrak{G}$-invariant according to Lemma \ref{ProjOperatorCommutative}. Hence if $\duspan{X_{\lambda,S_i}}{S_j} \subsetneq X_{\lambda,S_i}$, then we can decompose $X_{\lambda,S_i}$ into two subspaces $\duspan{X_{\lambda,S_i}}{S_j}$ and its orthogonal complement in $X_{\lambda,S_i}$, which makes our work of determining $\drm{Aut}{\oplus_{\lambda} X_{\lambda,S_i}}$ more efficiently.

In the case that $\oplus_{\lambda} X_{\lambda,S_i} \perp \oplus_{\lambda} X_{\lambda,S_j}$, we can first deal with those two subspaces separately and then incorporate the information about $\drm{Aut}{\oplus_{\lambda} X_{\lambda,S_i}}$ and $\drm{Aut}{\oplus_{\lambda} X_{\lambda,S_j}}$ to obtain $\drm{Aut}{\oplus_{\lambda} \big( X_{\lambda,S_i} \oplus X_{\lambda,S_j} \big)}$. More precisely, suppose $S_{i_1},\ldots,S_{i_l}$ are cells of the partition $\bar{\Pi}[ \oplus V_{\lambda} ]$ such that $i_k$ ($k = 1,\ldots,l$) is the minimum integer in $\{1,\ldots,t\}$ {\it s.t.,} $X_{\lambda,S_{i_k}} \perp \big( \oplus^{k-1}_{j=0} X_{\lambda,S_{i_j}} \big)$, where  $X_{\lambda,S_{i_0}} = X_{\lambda,S_{1}}$. Let us make a further assumption that if there is a cell $S_j$ in $\bar{\Pi}[ \oplus V_{\lambda} ] \setminus \{ S_{i_0},\ldots,S_{i_l} \}$ such that $X_{\lambda,S_j}$ is not orthogonal to some subspace $X_{\lambda,S_{i_k}}$ ($0\leq k \leq l$) then $\duspan{X_{\lambda,S_{i_k}}}{S_j} = X_{\lambda,S_{i_k}}$, otherwise we can decompose the subspace $X_{\lambda,S_{i_k}}$ in the way explained in the last paragraph. As a result, 
\begin{equation}\label{SubspaceDecomposition-I}
Y_{\lambda,\bar{\Pi}[ \oplus V_{\lambda} ]} = 
\oplus_{k=0}^l X_{\lambda,S_{i_k}} \oplus Z_{\lambda,S_1},
\end{equation}
where the subspace $Z_{\lambda,S_1}$ is the orthogonal complement of $\oplus_{k=0}^l X_{\lambda,S_{i_k}}$ in $Y_{\lambda,\bar{\Pi}[ \oplus V_{\lambda} ]}$. It is plain to see that we can decompose $Z_{\lambda,S_1}$ by means of subspaces not orthogonal to some of subspaces $X_{\lambda,S_{i_0}},\ldots,X_{\lambda,S_{i_l}}$ in a way similar to that of decomposing $Y_{\lambda,\bar{\Pi}[ \oplus V_{\lambda} ]}$. By repeating this process, we can ultimately obtain an orthogonal decomposition for $Y_{\lambda,\bar{\Pi}[ \oplus V_{\lambda} ]}$.

As to those subspaces contained in the first part, there are two possibilities: 
\begin{enumerate}

\item[(A)] $\forall S_{p'}\in\bar{\Pi}[ \oplus V_{\lambda} ]$, $\duspan{\oplus_{\lambda} X_{\lambda,S_{i_p}}}{S_{p'}} \neq \pmb{0} \Rightarrow X_{\lambda,S_{p'}} \perp \oplus_{\lambda} X_{\lambda,S_{i_q}}$, or \\ $\forall S_{q'}\in\bar{\Pi}[ \oplus V_{\lambda} ]$, $\duspan{\oplus_{\lambda} X_{\lambda,S_{i_q}}}{S_{q'}} \neq \pmb{0} \Rightarrow X_{\lambda,S_{q'}} \perp \oplus_{\lambda} X_{\lambda,S_{i_p}}$, \\ where $p,q \in [l]$ and $p \neq q$.

\vspace{2mm}
In this case $\drm{Aut}{\oplus_{\lambda} X_{\lambda,S_{i_p}}}$ has no impact on $\drm{Aut}{\oplus_{\lambda} X_{\lambda,S_{i_q}}}$ and vice verse, so we can deal with those two subspaces $\oplus_{\lambda} X_{\lambda,S_{i_p}}$ and $\oplus_{\lambda} X_{\lambda,S_{i_q}}$ separately.

\item[(B)] $\exists \hspace{0.6mm} S_j \in\bar{\Pi}[ \oplus V_{\lambda} ]$, $\duspan{\oplus_{\lambda} X_{\lambda,S_{i_p}}}{S_{j}} \neq \pmb{0}$ and $\duspan{\oplus_{\lambda} X_{\lambda,S_{i_q}}}{S_{j}} \neq \pmb{0}$, where $p,q \in [l]$ and $p \neq q$.

\vspace{2mm}
In order to determine the group $\drm{Aut}{\oplus_{\lambda} \big( X_{\lambda,S_{i_p}} \oplus X_{\lambda,S_{i_q}} \big)}$ in this case, we need to compare the effect of $\drm{Aut}{\oplus_{\lambda} X_{\lambda,S_{i_p}} }$ action on $S_j$ with that of $\drm{Aut}{\oplus_{\lambda} X_{\lambda,S_{i_q}} }$ action on $S_j$. As we have seen in the 2nd section, to do so we only need to compare a series of partitions consisting of orbits of stabilizers, each of which fixes a sequence of members of $S_j$, so that can be down efficiently.

\end{enumerate}

Now let us see how to cope with the subspace $X_{\lambda,S_1}$. Due to our discussion above, we assume that $\duspan{X_{\lambda,S_1}}{S_j} = X_{\lambda,S_1}$ or $\pmb{0}$, $\forall j > 1$, so we cannot decompose  $X_{\lambda,S_1}$ further by means of $S_j$. Recall that the equitable partition $\Pi[ \oplus V_{\lambda} ; \mathtt{B} ]$ ($x\in S_1$) is built for refining $\bar{\Pi}[ \oplus V_{\lambda} ]$ in the case that $\mathtt{B} \subsetneq S_1$, so we can use the partition to decompose $X_{\lambda,S_1}$:
\begin{equation*}
X_{\lambda,S_1} =
\left( X_{\lambda,S_1} \cap 
\mathbf{R}_{ \Pi[ \oplus V_{\lambda} ; \mathtt{B} ] } V_{\lambda}^{G/\Pi[ \oplus V_{\lambda} ; \mathtt{B} ]} \right) 
\oplus Y_{\lambda,S_1,\Pi_{\mathtt{B}}},
\end{equation*}
where $Y_{\lambda,S_1,\Pi_{\mathtt{B}}}$ is the orthogonal complement of the first subspace in $X_{\lambda,S_1}$. It is easy to see that $X_{\lambda,S_1} \cap \mathbf{R}_{ \Pi[ \oplus V_{\lambda} ; \mathtt{B} ] } V_{\lambda}^{G/\Pi[ \oplus V_{\lambda} ; \mathtt{B} ]} $ is an $\mathfrak{G}_{\mathtt{B}}$-invariant subspace, so is $Y_{\lambda,S_1,\Pi_{\mathtt{B}}}$. Consequently, in order to uncover the structure of the $\mathfrak{G}_{\mathtt{B}}$ action on $X_{\lambda,S_1}$, we need to decompose the subspace $Y_{\lambda,S_1,\Pi_{\mathtt{B}}}$ further. 

Since $\mathtt{B}$ may represent a block for $\mathfrak{G}$, there could be a block system of $\mathfrak{G}$  containing $\mathtt{B}$ as one member. More precisely, one can obtain, by carrying out first two operations of outputting $\Pi[ \oplus V_{\lambda} ; \mathtt{B} ]$ on the rest of members of $S_1$, not only one subset $\mathtt{B}$ but a group of subsets $\mathtt{B}_1 = \mathtt{B},\mathtt{B}_2,\ldots,\mathtt{B}_q$ of $S_1$. Furthermore, there is a partition $\stackrel{\scriptscriptstyle \Box }{\Pi}[ \oplus V_{\lambda} ; \mathtt{B} ]$ of $\{ \mathtt{B}_i : i = 1,\ldots,q \}$ induced by $\ReSt{ \Pi[ \oplus V_{\lambda} ; \mathtt{B} ] }{S_1}$ that is the partition of $S_1$ consisting of cells $\Pi[ \oplus V_{\lambda} ; \mathtt{B} ]$ each of which is contained in $S_1$. Let $\mathtt{L}_1=\mathtt{B}_1,\mathtt{L}_2,\ldots,\mathtt{L}_c$ be cells of $\stackrel{\scriptscriptstyle \Box }{\Pi}[ \oplus V_{\lambda} ; \mathtt{B} ]$ such that $|\mathtt{L}_1| \leq \cdots \leq |\mathtt{L}_c|$. Then we can use those cells to split the subspace $Y_{\lambda,S_1,\Pi_{\mathtt{B}}}$. 

Let $X_{\lambda,S_1,\mathtt{L}_i}$ ($i = 1,\ldots,c$) denote the subspace $\duspan{Y_{\lambda,S_1,\Pi_{\mathtt{B}}}}{\mathtt{L}_i}$. Clearly each cell $\mathtt{L}_i$ is invariant under the action of $\mathfrak{G}_{\mathtt{B}}$, so is the subspace $\oplus_{\lambda} X_{\lambda,S_1,\mathtt{L}_i}$ according to Lemma \ref{ProjOperatorCommutative}, $i=1,\ldots,c$. On the other hand, each $\mathtt{L}_i$ may contain some of cells of $\ReSt{ \Pi[ \oplus V_{\lambda} ; \mathtt{B} ] }{S_1}$, so we may split $X_{\lambda,S_1,\mathtt{L}_i}$ further by means of those cells relevant. Note that $\Pi[ \oplus V_{\lambda} ; \mathtt{B} ]$ is an equitable partition, so the lemma \ref{Lem-SeparatingPseudoOrbit} works well for cells of $\Pi[ \oplus V_{\lambda} ; \mathtt{B} ]$. As a result, we can finally decompose the subspace $Y_{\lambda,S_1,\Pi_{\mathtt{B}}}$ in a way similar to (\ref{SubspaceDecomposition-I}).

\vspace{2mm}
We now turn to the subspace $X_{\lambda,S_1,\mathtt{B}}$. Let $z$ be a vertex of $\mathtt{B}$. Then $X_{\lambda,S_1,\mathtt{B}}$ can be decomposed into a number of smaller subspaces by means of the equitable partition $\Pi[ \oplus V_{\lambda} ; z ]$:
\begin{equation*}
X_{\lambda,S_1,\mathtt{B}} =
\left( X_{\lambda,S_1,\mathtt{B}} \cap 
\mathbf{R}_{ \Pi[ \oplus V_{\lambda} ; z ] } V_{\lambda}^{G/\Pi[ \oplus V_{\lambda} ; z ]} \right) 
\oplus Y_{\lambda,S_1,\mathtt{B}}.
\end{equation*}
Again we need to split the subspace $Y_{\lambda,S_1,\mathtt{B}}$ further for revealing the structure of the action of $\mathfrak{G}_z$.

Let $\ReSt{ \Pi[\oplus V_{\lambda} ; z] }{\mathtt{B}}$ be the partition of $\mathtt{B}$ consisting of cells $\Pi[\oplus V_{\lambda} ; z]$ each of which is contained in $\mathtt{B}$. Suppose $\ReSt{ \Pi[\oplus V_{\lambda} ; z] }{\mathtt{B}} = \big\{ C_1^{z}=\{z\},C_2^{z},\ldots,C_l^{z} \big\}$ and $|C_2^{z}| \leq \cdots \leq |C_l^{z}|$. Let $X_{\lambda,S_1,\mathtt{B},C^{z}_i}$ denote the subspace $\duspan{Y_{\lambda,S_1,\mathtt{B}}}{C^{z}_i}$, $i=2,\ldots,l$. Then we can employ operations for obtaining $\Pi[ \oplus V_{\lambda} ; v ]$ to work out a balanced partition $\Pi[ \oplus_{\lambda,i} \hspace{0.5mm} X_{\lambda,S_1,\mathtt{B},C^{z}_i} ; y ]$ for every $y\in \mathtt{B}$. One may notice the difference between $\oplus_{\lambda} V_{\lambda}$ and $\oplus_{\lambda,i} \hspace{0.5mm} X_{\lambda,S_1,\mathtt{B},C^{z}_i}$: any two eigenspaces of $\dAM{G}$ are orthogonal but it could be the case that $\exists \hspace{0.6mm} X_{\lambda,S_1,\mathtt{B},C^{z}_i}$ and $X_{\lambda,S_1,\mathtt{B},C^{z}_j}$ ($i<j$) {\it s.t.,} $X_{\lambda,S_1,\mathtt{B},C^{z}_i} \cap X_{\lambda,S_1,\mathtt{B},C^{z}_j} \supsetneq \pmb{0}$. So the 2nd sum is not even a direct one. A moment's reflection would show, however, that we can decompose $Y_{\lambda,S_1,\mathtt{B}}$ by means of $X_{\lambda,S_1,\mathtt{B},C^{z}_2},\cdots,X_{\lambda,S_1,\mathtt{B},C^{z}_l}$ in a way that is the same as what we did for $Y_{\lambda,\bar{\Pi}[ \oplus V_{\lambda} ]}$ in (\ref{SubspaceDecomposition-I}), and obtain an orthogonal decomposition for $Y_{\lambda,S_1,\mathtt{B}}$. For simplicity, we do not introduce a new symbol here for that decomposition.

After that, we can work out an uniform partition $\bar{\Pi}[ \oplus_{\lambda,i} \hspace{0.5mm} X_{\lambda,S_1,\mathtt{B},C^{z}_i} ]$ for $\mathtt{B}$ in a way similar to building $\bar{\Pi}[\oplus V_{\lambda}]$ from $\{ \Pi[\oplus V_{\lambda} ; ] : v \in [n] \}$. Moreover, we can also work out a balanced partition $\Pi[ \oplus_{\lambda,i} \hspace{0.5mm} X_{\lambda,S_1,\mathtt{B},C^{z}_i} ; \mathtt{B}_y ]$ as an approximation to some block for $\drm{Aut}{\oplus_{\lambda,i} \hspace{0.5mm} X_{\lambda,S_1,\mathtt{B},C^{z}_i}}$.

In brief, we can obtain the following by means of the operations outputting (\ref{Computation-I}):
\begin{align*}
\oplus_{\lambda,i} \hspace{0.5mm} X_{\lambda,S_1,\mathtt{B},C^{z}_i} & \rightarrow 
\left\{ \Pi[ \oplus_{\lambda,i} \hspace{0.5mm} X_{\lambda,S_1,\mathtt{B},C^{z}_i};x_i ] : x_i \in C_i^{z} \right\} ~~ (i=1,\ldots,l) \\
& \rightarrow \bar{\Pi}[ \oplus_{\lambda,i} \hspace{0.5mm} X_{\lambda,S_1,\mathtt{B},C^{z}_i} ] ~ \& ~ 
\Pi[ \oplus_{\lambda,i} \hspace{0.5mm} X_{\lambda,S_1,\mathtt{B},C^{z}_i} ; \mathtt{B}_y ].
\end{align*}
Apparently, $\bar{\Pi}[ \oplus_{\lambda,i} \hspace{0.5mm} X_{\lambda,S_1,\mathtt{B},C^{z}_i} ]$ is a refinement of $\ReSt{\Pi[\oplus V_{\lambda} ; z]}{\mathtt{B}}$. If the latter is refined properly by the first one, then we can decompose those subspaces $X_{\lambda,S_1,\mathtt{B},C^{z}_2},\cdots,X_{\lambda,S_1,\mathtt{B},C^{z}_l}$ further for some of eigenvalues of $\dAM{G}$.

\subsection{Assembling Subspaces }

\vspace{2.5mm}
Recall that $S_1,\ldots,S_t$ are cells of $\bar{\Pi}[ \oplus V_{\lambda} ]$ such that $|S_1|\leq\cdots\leq |S_t|$, so if $t \geq 2$ then $|S_1| \leq n/2$, and thus each subspace $X_{\lambda,S_1}$ ($\lambda\in\drm{spec}{\dAM{G}}$) is of dimension not more than $n/2$. Accordingly we can first determine the group $\drm{Aut}{\oplus_{\lambda} X_{\lambda,S_1}}$ and then use the information to deal with the rest of cells. As a result, $\bar{\Pi}[ \oplus V_{\lambda} ] = \{ [n] \}$ in the worst case.

Notice that $\mathtt{B}$ passes the first two tests in the process of building the partition $\Pi[ \oplus V_{\lambda} ; \mathtt{B} ]$, each of which is a necessary condition for being a minimal block for $\mathfrak{G}$, so if $\mathtt{B} \subsetneq [n]$ then $|\mathtt{B}| \leq n/2$. Hence again we can first determine the group $\drm{Aut}{\oplus_{\lambda} X_{\lambda,S_1,\mathtt{B}}}$ and then use the information to deal with other cells $\mathtt{S}_2,\ldots,\mathtt{S}_c$ of $\Pi[ \oplus V_{\lambda} ; \mathtt{B} ]$.  
$$\mbox{Therefore in the worst case, }\mathtt{B} = [n].$$

Despite the fact that our effort to split eigenspaces of $\dAM{G}$ by using cells of $\Pi[ \oplus V_{\lambda} ; \mathtt{B} ]$ fails in the case that $\mathtt{B} = [n]$, we know according to the 2nd test for building $\mathtt{B}$ that each partition $\Pi[ \oplus V_{\lambda};v ]$ ($v\in [n]$) contains exactly one singleton --- $\{ v \}$, and the direct graph $\mathrm{PBG}(\mathtt{B}) = (\mathtt{B},\{ \phi_{xy} E(x) : y \in \mathtt{B} \})$ is strong connected. It is the 2nd relation that offers us a powerful apparatus for dealing with the case that there is a big cell $C^v_m$ in $\Pi[ \oplus V_{\lambda};v ]$ such that $|C^v_m| > n / 2$.

As a matter of fact, there is in the case $\bar{\Pi}[ \oplus V_{\lambda} ] = \{ [n] \}$ another important property we can use to deal with the big cell of $\Pi[ \oplus V_{\lambda};v ]$, which has been stated in the section 1.2. Recall that $C_1^v=\{v\},C_2^v,\ldots,C_m^v$ are the cells of the partition $\Pi[\oplus V_{\lambda} ; v]$ such that $|C_2^v| \leq \cdots \leq |C_m^v|$ and $m\geq 3$. Accordingly we can single out two subspaces of $V_{\lambda}$: 
$$ Y_{\lambda, v} = V_{\lambda} \ominus \mathbf{R}_{\Pi[\oplus V_{\lambda} ; v] } V_{\lambda}^{G/\Pi[\oplus V_{\lambda} ; v]} \mbox{ and } X_{\lambda,v,m-1} = \drm{span}{\{ Y_{\lambda,v} : \cup_{i=2}^{m-1} C_i^v \}},
$$
where $\lambda \in \drm{spec}{\dAM{G}}$.

\vspace{2.5mm}
\noindent {\bf Lemma 7.} {\it Suppose $\bar{\Pi}[ \oplus V_{\lambda} ]$ contains only one cell $[n]$. If $|C_m^v| > n/2$ then one of following two cases occurs. 
\begin{enumerate}

\item[i)] The subspace $\duspan{\oplus_{\lambda \hspace{0.5mm} \in \hspace{0.5mm} \drm{spce}{\dAM{G}}}X_{\lambda,v,m-1}}{C_m^v}$ is non-trivial.

\item[ii)] For any vertex $x$ of $[n] \setminus C_m^v$, $C_m^x = C_m^v$ where $C_m^x$ denotes the biggest cell of $\Pi[\oplus V_{\lambda} ; x]$.

\end{enumerate}}\vspace{2.5mm}

Obviously the vertex $v$ is contained in a singleton $\{v\}$ as a cell of $\Pi[\oplus V_{\lambda} ; v]$, so if $\Pi[\oplus V_{\lambda} ; v]$ possesses only two cells then $\Pi[\oplus V_{\lambda} ; v] = \{ \{v\}, [n]\setminus \{v\} \}$. As a result, if $\bar{\Pi}[ \oplus V_{\lambda} ]$ contains only one cell and $\Pi[\oplus V_{\lambda} ; v]$ contains only two cells then the graph $G$ is actually isomorphic to $K_n$, the complete graph of order $n$. On the other hand, it is easy to verify that if $\bar{\Pi}[ \oplus V_{\lambda} ]$ contains only one cell then $G$ is a regular graph. 

\begin{proof}[\bf Proof]
We assume that $\drm{span}{\{\oplus_{\lambda} \hspace{0.5mm} X_{\lambda,v,m-1} : C_m^v \}} = \pmb{0}$. Let $Z_{\lambda,v,m}$ denote the subspace $\duspan{Y_{\lambda,v}}{ C_m^v }$. Then $X_{\lambda,v,m-1} \perp Z_{\lambda,v,m}$. Note that 
$$
V_{\lambda} = 
V_{\lambda,\Pi_v} \oplus
X_{\lambda,v,m-1} \oplus Z_{\lambda,v,m},
$$
where $V_{\lambda,\Pi_v}$ stands for the subspace $\mathbf{R}_{\Pi[\oplus V_{\lambda} ; v] } V_{\lambda}^{G/\Pi[\oplus V_{\lambda} ; v]}$, so those three subspaces are orthogonal to each other. Consequently, for any $w\in C_m^v$,
\begin{equation}\label{Relation-A}
\dproj{ V_{\lambda} }( \pmb{e}_w ) = 
\dproj{ V_{\lambda,\Pi_v} }( \pmb{e}_w ) + \dproj{ Z_{\lambda,v,m} }( \pmb{e}_w ),
\end{equation} 
and thus 
$\dproj{ V_{\lambda,\Pi_v} \oplus X_{\lambda,v,m-1} }( \pmb{e}_{w} ) = \frac{1}{|C_m^v|} \cdot \dproj{ V_{\lambda} }( \pmb{R}_{C_m^v} )$, where $\pmb{R}_{C_m^v}$ is the characteristic vector of the subset $C_m^v$.

Let $\pmb{p}_{\lambda,u}$ denote the projection $\dproj{ V_{\lambda} }( \pmb{e}_u )$, $u \in [n]$. It is easy to check that for any vertex $x \in [n]\setminus C_m^v$,
\begin{equation}\label{Property-A}
\left\langle \pmb{p}_{\lambda,x},\pmb{p}_{\lambda,w'}  \right\rangle = 
\left\langle \pmb{p}_{\lambda,x},\pmb{p}_{\lambda,w''} \right\rangle, ~~ \forall w',w''\in C_m^v.
\end{equation}


\vspace{2mm}
let $\pmb{p}_{\lambda,C_m^v}$ denote $\dproj{ V_{\lambda} }( \pmb{R}_{C_m^v} )$. We first consider a simple case.

\begin{case}\label{Case-BB}
$\drm{span}{\{\pmb{p}_{\lambda,v}\}} = \drm{span}{\{\pmb{p}_{\lambda,C_m^v}\}}$.
\end{case}

Suppose $V_{\lambda}$ is an eigenspace of $\dAM{G}$ such that $\pmb{p}_{\lambda,v} \neq \pmb{p}_{\lambda,x}$, where $x$ is taken from $[n]\setminus\big( \{v\}\cup C_m^v \big)$. Then for any $w\in C_m^v$, 
$\left\langle \pmb{p}_{\lambda,v},\pmb{p}_{\lambda,w}  \right\rangle \neq \left\langle \pmb{p}_{\lambda,x},\pmb{p}_{\lambda,w}  \right\rangle$. Combining the relation (\ref{Property-A}) with the condition that $|C_m^v| > n/2$, one can readily see that there is no such a big cell $C_m^x$ of size $|C_m^v|$ in $\Pi[ \oplus V_{\lambda}; x ]$ so that 
for any $z\in C_m^x$, 
$\left\langle \pmb{p}_{\lambda,x},\pmb{p}_{\lambda,z}  \right\rangle = \left\langle \pmb{p}_{\lambda,v},\pmb{p}_{\lambda,w}  \right\rangle$, 
which contradicts the assumption that the uniform partition $\bar{\Pi}[ \oplus V_{\lambda} ]$ contains only one cell $[n]$. As a result, $\drm{span}{\{ X_{\lambda,v,m-1} : C_m^v \}} \neq \pmb{0}$ in this case. 

\begin{case}\label{Case-AA}
$\drm{span}{\{\pmb{p}_{\lambda,v}\}} \neq \drm{span}{\{\pmb{p}_{\lambda,C_m^v}\}}$.
\end{case}

\begin{Claim}
Let $w$ be a vertex of $C_m^v$ and $x$ a vertex of $[n]\setminus\big( \{v\} \cup C_m^v \big)$. Then
\begin{equation}\label{Property-B}
\left\langle \pmb{p}_{\lambda,v},\pmb{p}_{\lambda,w}  \right\rangle = \left\langle \pmb{p}_{\lambda,x},\pmb{p}_{\lambda,w} \right\rangle.
\end{equation}
Moreover, if $\lambda$ is not the biggest eigenvalue of $\dAM{G}$ then $\left\langle \pmb{p}_{\lambda,v},\pmb{p}_{\lambda,w}  \right\rangle < 0$.
\end{Claim}

Let $C^v_H$ be the subset of $[n]$ such that $\forall q \in C^v_H$ and $w \in C^v_m$,
$\left\langle \pmb{p}_{\lambda,v},\pmb{p}_{\lambda,q} \right\rangle = 
\left\langle \pmb{p}_{\lambda,v},\pmb{p}_{\lambda,w} \right\rangle$. Consequently, $C^v_m \subseteq C^v_H$ and thus $|C^v_H |>n/2$. In accordance with the condition that $\bar{\Pi}[ \oplus V_{\lambda} ]$ contains only one cell $[n]$, there is also a big subset $C_H^x$ associated with $\pmb{p}_{\lambda,x}$ of size $|C_H^v|$ such that 
$\forall q \in C^v_H$ and $r \in C^x_H$,
$$
\left\langle \pmb{p}_{\lambda,v},\pmb{p}_{\lambda,q} \right\rangle = 
\left\langle \pmb{p}_{\lambda,x},\pmb{p}_{\lambda,r} \right\rangle.
$$
Then $C_m^v$ must be contained in $C_H^x$ due to the requirement that $|C_m^v| > n/2$ and relation (\ref{Property-A}). Hence the equation (\ref{Property-B}) follows. 
Therefore, $\left\langle \pmb{p}_{\lambda,v},\pmb{p}_{\lambda,C_m^v} \right\rangle = 
\left\langle \pmb{p}_{\lambda,x},\pmb{p}_{\lambda,C_m^v} \right\rangle = 
\left\langle \pmb{p}_{\lambda,v},|C_m^v| \cdot \pmb{p}_{\lambda,w} \right\rangle$.

Since $\bar{\Pi}[ \oplus V_{\lambda} ]$ is an equitable partition and contains only one cell, if $\lambda$ is not the biggest eigenvalue of $\dAM{G}$ then 
$$\sum_{i=1}^m \dproj{V_{\lambda}}( \pmb{R}_{C_i^v} ) = \dproj{V_{\lambda}}( \pmb{1} ) = \pmb{0},
$$
where $\pmb{1} = (1,1,\ldots,1)$. Consequently, $\left\langle \pmb{p}_{\lambda,v},\pmb{p}_{\lambda,w}  \right\rangle < 0$.

\vspace{3mm}
\noindent{\bf Case 2.1.} $V_{\lambda}$ is an eigenspace of $\dAM{G}$ such that $\pmb{p}_{\lambda,x} = \pmb{p}_{\lambda,v}$ for any $x \in [n]\setminus C_m^v$. 

Since $\bar{\Pi}[ \oplus V_{\lambda} ]$ possesses only one cell and $C_m^v$ is a cell of $\Pi[\oplus V_{\lambda} ; v]$, there exists a cell $C^x$ of $\Pi[ V_{\lambda} ; x ]$ so that $C_m^v \subseteq C^x$.

\vspace{2mm}
\noindent{\bf Case 2.2.} $V_{\lambda}$ is an eigenspace of $\dAM{G}$ such that $\pmb{p}_{\lambda,y} \neq \pmb{p}_{\lambda,v}$ for some $y \in [n] \setminus \big( \{v\} \cup C_m^v \big)$, and $\langle \pmb{p}_{\lambda,v},\pmb{p}_{\lambda,w} \rangle = 0$ for any $w \in C_m^v$.

According to the relation (\ref{Property-B}), $C_m^v$ is contained in the thin cell of $\Pi[ V_{\lambda} ; x ]$ for any $x\in [n] \setminus C_m^v$.

\vspace{2mm}
\noindent{\bf Case 2.3.} $V_{\lambda}$ is an eigenspace of $\dAM{G}$ such that $\pmb{p}_{\lambda,y} \neq \pmb{p}_{\lambda,v}$ for some $y \in [n] \setminus \big( \{v\} \cup C_m^v \big)$, and $\langle \pmb{p}_{\lambda,v},\pmb{p}_{\lambda,w} \rangle \neq 0$ for any $w \in C_m^v$.
 
Obviously $\dproj{V_{\lambda}}( \pmb{1} )$ must be $\pmb{0}$ in this case. Let $x$ be a vertex in $[n]\setminus C_m^v$. Set 
$$
A^{\lambda}_x = \big\{ u \in [n] : \langle \pmb{p}_{\lambda,u},\pmb{p}_{\lambda,x} \rangle \neq 0 \big\}
\mbox{ and }
V_{\lambda,A^{\lambda}_x} = \drm{span}{\left\{ \pmb{p}_{\lambda,u} : u \in A^{\lambda}_x \right\}}.
$$ 
Apparently $C^v_m \subseteq A^{\lambda}_x$ and thus the vector $\dproj{ V_{\lambda} }( \pmb{R}_{C_m^v} )$, which is equal to $\sum_{w\in C^v_m} \pmb{p}_{\lambda,w}$, belongs to $V_{\lambda,A^{\lambda}_x}$. As a result, $V_{\lambda,A^{\lambda}_x} \supsetneq \drm{span}{\{ \pmb{p}_{\lambda,r} : r \in A^{\lambda}_x \setminus C_m^v \}}$.

In fact, if it is not the case, {\it i.e.,} $V_{\lambda,A^{\lambda}_x} = \drm{span}{\{ \pmb{p}_{\lambda,r} : r \in A^{\lambda}_x \setminus C_m^v \}}$, then $V_{\lambda,A^{\lambda}_x} \subseteq V_{\lambda,\Pi_v} \oplus X_{\lambda,v,m-1}$, for the latter subspace is spanned by vectors $\{ \pmb{p}_{\lambda,y} : y \in [n]\setminus C^v_m \}$. 

On the other hand, since $\bar{\Pi}[ \oplus V_{\lambda} ]$ possesses only one cell, any two members $\pmb{p}_{\lambda,s}$ and $\pmb{p}_{\lambda,t}$ ($s,t\in [n]$) of the OPSB onto $V_{\lambda}$ would be in the same type, {\it i.e.,} $\{ \pmb{p}_{\lambda,s} \} = \{ \pmb{p}_{\lambda,t} \}$. Thus $\|\pmb{p}_{\lambda,s}\| = \|\pmb{p}_{\lambda,t}\|$. Note that for any $w\in C_m^v$, $\dproj{ V_{\lambda,\Pi_v} \oplus X_{\lambda,v,m-1} }( \pmb{e}_{w} ) = \frac{1}{|C_m^v|} \cdot \dproj{ V_{\lambda} }( \pmb{R}_{C_m^v} )$ and $|C_m^v| > n/2$, so $\dproj{V_{\lambda,A^{\lambda}_x}}( \pmb{1} ) \neq \pmb{0}$, which is a contradiction. Therefore, $\drm{span}{\{ \pmb{p}_{\lambda,r} : r \in A^{\lambda}_x \setminus C_m^v \}} \subsetneq V_{\lambda,A^{\lambda}_x}$.

Suppose $R_x$ is the first region of $V_{\lambda,A^{\lambda}_x}$ obtained in outputting $\Pi[  V_{\lambda} ; x ]$, which contains $\pmb{p}_{\lambda,x}$ and is carved up by dividers $\{ \dproj{ V_{\lambda,A^{\lambda}_x} }( \pmb{e}_u )^{\perp} : u \in A^{\lambda}_x \}$. One can readily check that $\duspan{V_{\lambda,A^{\lambda}_x}}{\mathcal{I}_{R_x}} = V_{\lambda,A^{\lambda}_x}$ where $\mathcal{I}_{R_x}$ is the incidence set of $R_x$. Consequently, $\mathcal{I}_{R_x} \cap C_m^v \neq \emptyset$. Clearly $C_m^v \subseteq \mathcal{I}_{R_v}$, since $C_m^v \in \Pi[ \oplus V_{\lambda} ; v ]$. As we have seen, for any $w\in C_m^v$ and $x\in [n]\setminus C_m^v$, 
$$
\left\langle \pmb{p}_{\lambda,v},\pmb{p}_{\lambda,w}  \right\rangle = \left\langle \pmb{p}_{\lambda,x},\pmb{p}_{\lambda,w} \right\rangle \mbox{ and }
\dproj{ V_{\lambda,\Pi_v} \oplus X_{\lambda,v,m-1} }( \pmb{e}_{w} ) = \frac{1}{|C_m^v|} \cdot \dproj{ V_{\lambda} }( \pmb{R}_{C_m^v} ),
$$ 
so $C_m^v \subseteq \mathcal{I}_{R_x}$. 

\vspace{3mm}
It is routine to check that those three operations used in building $\Pi[ \oplus V_{\lambda} ; x ]$ cannot split the subset $C_m^v$, so $C_m^v$ is a cell of $\Pi[ \oplus V_{\lambda} ; x ]$. As a result, for any $x \in [n]\setminus C_m^v$, $\Pi[ \oplus V_{\lambda} ; x ]$ contains $C_m^v$ as a cell. \end{proof}

As one can readily see, Lemma \ref{Lem-SeparatingBigCell} holds in more general case. Let $\bar{\Pi}$ be an uniform partition of $[n]$ and $S$ some non-singleton cell of $\bar{\Pi}$ we need to split. Since $\bar{\Pi}$ is an equitable partition, the eigenspace $V_{\lambda}$ can be decomposed as $\mathbf{R}_{\bar{\Pi}} V_{\lambda}^{G/\bar{\Pi}} \oplus Y_{\lambda,\bar{\Pi}}$, where those two subspaces involved are orthogonal to one another. Set $X_{\lambda,\bar{\Pi},S} = \duspan{ Y_{\lambda,\bar{\Pi}} }{S}$. Let $x$ be a vertex in $S$ and let $\Pi[ \oplus X_{\lambda,\bar{\Pi},S} ; x ]$ be the balanced partition obtained by those three operations that are used to output $\Pi[ \oplus V_{\lambda} ; v ]$ but now carried out on $\oplus_{\lambda \in\drm{spec}{\dAM{G}}} \hspace{0.5mm} X_{\lambda,\bar{\Pi},S}$. 

Let $\Pi_x$ denote the partition $\Pi[ \oplus X_{\lambda,\bar{\Pi},S} ; x ]$ and $\ReSt{\Pi_x}{S}$ the family of subsets consisting of those cells of $\Pi_x$ each of which is contained in $S$. Suppose $C_1^x=\{x\},C_2^x,\ldots,C_m^x$ are cells of $\ReSt{\Pi_x}{S}$ such that 
$|C_2^x| \leq \cdots \leq |C_m^x|$ and $m\geq 3$. Because $\Pi_x$ is an equitable partition,  
$$
X_{\lambda,\bar{\Pi},S} = 
\left( X_{\lambda,\bar{\Pi},S} \cap \mathbf{R}_{\Pi_x} V_{\lambda}^{G/\Pi_x} \right)
\oplus Y_{\lambda,x},
$$ 
where $Y_{\lambda,x}$ is the orthogonal complement of the first subspace in $X_{\lambda,\bar{\Pi},S}$. Again we use $X_{\lambda,x,m-1}$ to denote the subspace $\drm{span}{\{ Y_{\lambda,x} : \cup_{i=2}^{m-1} C_i^x \}}$, where $\lambda \in \drm{spec}{\dAM{G}}$.

\begin{Lemma}\label{Lem-SeparatingBigCell-GeneralForm} 
If $|C_m^x| > |S|/2$ then one of following two cases occurs. 
\begin{enumerate}

\item[i)] The subspace $\duspan{\oplus_{\lambda \hspace{0.5mm} \in \hspace{0.5mm} \drm{spce}{\dAM{G}}}X_{\lambda,x,m-1}}{C_m^x}$ is non-trivial.

\item[ii)] For any vertex $y$ of $S \setminus C_m^x$, $C_m^y = C_m^x$ where $C_m^y$ denotes the biggest cell of $\ReSt{\Pi_y}{S}$.

\end{enumerate}
\end{Lemma}

It is routine to verify that the assertion above can be proved by the argument used in proving Lemma \ref{Lem-SeparatingBigCell}.

\vspace{1.5mm}
Since $\Pi[ \oplus V_{\lambda} ; x ]$ is an equitable partition, each eigenspace $V_{\lambda}$ can be decomposed into $V_{\lambda,\Pi_x} \oplus Y_{\lambda,x}$, where $V_{\lambda,\Pi_x}$ stands for the subspace $\mathbf{R}_{\Pi[\oplus V_{\lambda} ; x] } V_{\lambda}^{G/\Pi[\oplus V_{\lambda} ; x]}$ and $Y_{\lambda,x}$ is the orthogonal complement of $V_{\lambda,\Pi_x}$ in $V_{\lambda}$. Our aim here is to assemble in the case that $\mathtt{B} = [n]$ those subspaces we have singled out for revealing symmetries represented in $Y_{\lambda,x}$. Recall that $C^x_1 = \{ x \},C^x_2,\ldots,C^x_m$ are cells of $\Pi[ \oplus V_{\lambda} ; x ]$ such that $m\geq 3$ and $|C^x_2| \leq \cdots \leq |C^x_m|$. Clearly, there are two cases:
$$
|C^x_m| \leq n/2 ~ \mbox{ or } ~ |C^x_m| > n/2.
$$
We first consider the 2nd case and then use the machinery developed for that case to deal with the 1st one. 

There are due to Lemma \ref{Lem-SeparatingBigCell} two possibilities.

\begin{itemize}

\item[1)] $\forall\hspace{0.6mm} y \in [n]\setminus C^x_m$, $C^y_m = C^x_m$.

First of all, we use the relation above to define a binary relation among vertices of $G$: two vertices $u$ and $v$ are said to be related if $C_m^u = C_m^v$. Evidently, it is an equivalence relation, so there is a partition $P_B$ of $[n]$ associated with the relation. Let $E_1^B,\ldots,E_q^B$ be cells of $P_B$. Clearly if $u_i \in E_i^B$ then $E_i^B = \cup_{j=1}^{m-1} C_j^{u_i}$, and thus $|E_i^B| = n - |C_m^{u_i}|$ ($i=1,\ldots,q$), which is less than $n/2$.

According to Lemma \ref{Lem-SeparatingBigCell}, each eigenspace $V_{\lambda}$ has an orthogonal decomposition $V_{\lambda,\Pi_x} \oplus X_{\lambda,x,m-1} \oplus Z_{\lambda,x,m}$, where $X_{\lambda,x,m-1} = \drm{span}{\{ Y_{\lambda,x} : \cup_{i=2}^{m-1} C_i^x \}}$ and $Z_{\lambda,x,m} = \drm{span}{\{ Y_{\lambda,x} : C_m^x \}}$. Consequently, $\dproj{V_{\lambda}}(\pmb{R}_{C_m^{u_i}}) = \dproj{V_{\lambda,i}}(\pmb{R}_{C_m^{u_i}})$, where $u_i \in E_i^B$ and $V_{\lambda,i} = \duspan{ V_{\lambda} }{ E_i^B }$. In accordance with relations (\ref{Relation-A}) and (\ref{Property-B}), coordinates of the vector $\dproj{V_{\lambda}}(\pmb{R}_{C_m^{u_i}})$ only take two values:
$$
\left\langle \dproj{V_{\lambda}}(\pmb{R}_{C_m^{u_i}}),\dproj{V_{\lambda}}(\pmb{e}_{w}) \right\rangle \mbox{ if } w \in C_m^{u_i},
$$
or
$$
\left\langle \dproj{V_{\lambda}}(\pmb{R}_{C_m^{u_i}}),\dproj{V_{\lambda}}(\pmb{e}_{z}) \right\rangle \mbox{ if } z \in [n] \setminus C_m^{u_i}.
$$ 
As a result, the subspace $\drm{span}{\{ \dproj{V_{\lambda}}(\pmb{R}_{C_m^{u_i}}) : i \in [q] \}}$ is of dimension $q-1$, and therefore the group $\mathrm{Aut}\hspace{0.5mm} \drm{span}{\{ \dproj{V_{\lambda}}(\pmb{R}_{C_m^{u_i}}) : i \in [q] \}}$ is isomorphic to the product group $\Pi_{k=1}^{q} \drm{Sym}{\big[ \hspace{0.4mm} |E_1^B| \hspace{0.4mm} \big]}$, where $\big[ \hspace{0.4mm} |E_1^B| \hspace{0.4mm} \big] = \left\{ 1,2,\ldots,|E_1^B|\right\}$.

Set $Y_{\lambda,i} = V_{\lambda,i} \ominus \drm{span}{\{ \dproj{V_{\lambda}}(\pmb{R}_{C_m^{u_i}}) \}}$, {\it i.e.,} $Y_{\lambda,i}$ is the orthogonal complement of the subspace spanned by $\dproj{V_{\lambda}}(\pmb{R}_{C_m^{u_i}})$ in $V_{\lambda,i}$. Because $E_j^B \subseteq C_m^{u_i}$ if $i\neq j$, $Y_{\lambda,i} \perp Y_{\lambda,j}$, so the eigenspace $V_{\lambda}$ can be decomposed as follows:
$$
\left( \oplus_{i=1}^q Y_{\lambda,i} \right) \oplus
 \drm{span}{\{ \dproj{V_{\lambda}}(\pmb{R}_{C_m^{u_i}}) : i = 1,\ldots,q \}}.
$$
Accordingly, in order to determine whether or not $\dproj{V_{\lambda}}(\pmb{e}_{u_i})$ and $\dproj{V_{\lambda}}(\pmb{e}_{v_i})$ are symmetric in $V_{\lambda}$, where $u_i$ and $v_i$ belong to $E_i^B$, we only need to determine whether or not there is a permutation $\gamma$ of $E_i^B$ so that $\gamma \hspace{0.4mm} Y_{\lambda,i} = Y_{\lambda,i}$ and $\gamma \hspace{0.4mm} \dproj{V_{\lambda}}(\pmb{e}_{u_i}) = \dproj{V_{\lambda}}(\pmb{e}_{v_i})$, while in order to determine whether or not $\dproj{V_{\lambda}}(\pmb{e}_{u_i})$ and $\dproj{V_{\lambda}}(\pmb{e}_{u_j})$ ($i \neq j$) are symmetric in $V_{\lambda}$, where $u_i\in E_i^B$ and $u_j\in E_j^B$, we only need to determine whether there is a permutation $\gamma$ of $E_i^B \cup E_j^B$ so that $\gamma \hspace{0.4mm} V_{\lambda} = V_{\lambda}$, $\gamma \hspace{0.4mm} Y_{\lambda,i} = Y_{\lambda,j}$ and $\gamma \hspace{0.4mm} \dproj{V_{\lambda}}(\pmb{e}_{u_i}) = \dproj{V_{\lambda}}(\pmb{e}_{u_j})$. 

As a result, we need only to focus on relations among members in $E_i^B$ ($i=1,\ldots,q$) in order to work out the information about the group $\drm{Aut}{\oplus_{\lambda} Y_{\lambda,i}}$, {\it i.e.,} the information about the partition of $E_i^B$ consisting orbits of $\drm{Aut}{\oplus_{\lambda} Y_{\lambda,i}}$ action on $E_i^B$ and a series of partitions of $E_i^B$ associated with a fastening sequence of the group.\footnote{ Note that $\oplus_{\lambda} Y_{\lambda,i} \perp \oplus_{\lambda} Y_{\lambda,j}$ if $i \neq j$, {\it i.e.,} $\duspan{\oplus_{\lambda} Y_{\lambda,i}}{E_j^B} = \pmb{0}$, so the information about $\drm{Aut}{\oplus_{\lambda} Y_{\lambda,i}}$ could be fully described with partitions of $E_i^B$. } Moreover, after having obtained those partitions of $E_i^B$ relevant to $\drm{Aut}{\oplus_{\lambda} Y_{\lambda,i}}$ and its stabilizers, one can by running the algorithm on $\oplus_{\lambda} Y_{\lambda,j}$ ($j\neq i$) easily determine the corresponding relations between cells of those partitions of $E_i^B$ and of $E_j^B$. Finally we can obtain in a reductive way the information about $\drm{Aut}{\oplus V_{\lambda}}$.

\item[2)] $\duspan{\oplus_{\lambda} \hspace{0.5mm}X_{\lambda,x,m-1}}{C^x_m}\neq\pmb{0}$.

Recall that $X_{\lambda,x,m-1} = \duspan{Y_{\lambda,x}}{\cup_{i=2}^{m-1} C_i^x}$, $\lambda\in\drm{spec}{\dAM{G}}$, so we can single out one more subspace $Z_{\lambda,x,m}$ of $Y_{\lambda,x}$ which is the orthogonal complement of $X_{\lambda,x,m-1}$. Consequently we have for each eigenspace an orthogonal decomposition 
$$ 
V_{\lambda} = V_{\lambda,\Pi_x} \oplus
X_{\lambda,x,m-1} \oplus Z_{\lambda,x,m}.
$$
Accordingly $\dproj{ X_{\lambda,x,m-1} }( \pmb{R}_{C_m^x} ) = \pmb{0}$, $\forall\hspace{0.5mm} \lambda \in \drm{spec}{\dAM{G}}$. Because $\drm{span}{\{ \oplus_{\lambda} \hspace{0.5mm} X_{\lambda,x,m-1} : C_m^x \}} \neq \pmb{0}$, there exist a group of vectors $\pmb{s}_1,\ldots,\pmb{s}_q$ in $\oplus_{\lambda} \hspace{0.5mm} X_{\lambda,x,m-1}$ such that $q\geq 2$ and $\forall \hspace{0.5mm} i \in [q]$, $\exists \hspace{0.5mm} w \in  C_m^x$ {\it s.t.,} $\dproj{\oplus_{\lambda}  X_{\lambda,x,m-1}}(\pmb{e}_w) = \pmb{s}_i$.  

Set $\drm{proj}{^{-1}} \pmb{s}_i = \{ w \in C_m^x : \dproj{\oplus_{\lambda} X_{\lambda,x,m-1}}(\pmb{e}_w) = \pmb{s}_i \}$. Since each subspace $X_{\lambda,x,m-1}$ is $\mathfrak{G}_x$-invariant, if any one of three cases below occurs then $C_m^x$ cannot be an orbit of $\mathfrak{G}_x$:
\begin{enumerate}

\item $\cup_{i=1}^q \hspace{0.5mm} \drm{proj}{^{-1}} \pmb{s}_i \subsetneq C^x_m$;

\item $\exists \hspace{0.5mm} i,j \in [q]$ {\it s.t.,} $|\drm{proj}{^{-1}} \pmb{s}_i| \neq |\drm{proj}{^{-1}} \pmb{s}_j|$;

\item $\exists \hspace{0.5mm} i,j \in [q]$ {\it s.t.,} $\pmb{s}_i$ and $\pmb{s}_j$ do not belong to the same orbit of $\drm{Aut}{\oplus_{\lambda} X_{\lambda,x,m-1}}$.

\end{enumerate}
Hence, if $C_m^x$ is an orbit of $\mathfrak{G}_x$, it must be split into at least two equal parts by grouping its members according to projections, {\it i.e.,} $\drm{proj}{^{-1}} \pmb{s}_1,\ldots,\drm{proj}{^{-1}} \pmb{s}_q$, so the order of each part is less than $|C_m^x|/2<n/2$. In what follows, we assume none of three cases listed above occurs. Again, there are two possibilities. 
 
 \begin{itemize}
 
 \item[2.1)] $\forall\hspace{0.6mm} w',w'' \in C^x_m$, $\dproj{\oplus_{\lambda} X_{\lambda,x,m-1}}(\pmb{e}_{w'}) \neq \dproj{\oplus_{\lambda}  X_{\lambda,x,m-1}}(\pmb{e}_{w''})$.
 
We begin with defining a digraph $\mathrm{DPBG}(\mathtt{B}) = \left( \mathtt{B},\left\{ \mathtt{B} \setminus (C_m^y\cup\{y\}) : y \in \mathtt{B} \right\}\right)$ that is a denser version of the digraph $\mathrm{PBG}(\mathtt{B})$ which we constructed as the 2nd test for $\mathtt{B}$ being a minimal block for $\mathfrak{G}$ or not. More precisely $\mathtt{B}$ is the vertex set of $\mathrm{DPBG}(\mathtt{B})$ and there is an arc from $u'$ to $u''$, {\it i.e.,} $u' \rightarrow u''$, if $u'' \in \mathtt{B} \setminus \big( C_m^{u'} \cup \{u'\} \big)$. Obviously, if we construct $\mathrm{PBG}(\mathtt{B})$ by virtue of a family of small cells $\{ \phi_{xy} C_i^x : y \in \mathtt{B} \}$  such that $|C_i^x| < |C_m^x|$ then $\mathrm{PBG}(\mathtt{B})$ is contained in $\mathrm{DPBG}(\mathtt{B})$ as a subgraph, where $\phi_{xy}$ stands for the corresponding relation between cells of two partitions $\Pi[ \oplus V_{\lambda};x ]$ and $\Pi[ \oplus V_{\lambda};y ]$ induced by the procedure of outputting those two partitions. Because $\mathrm{PBG}(\mathtt{B})$ is strong connected, so is $\mathrm{DPBG}(\mathtt{B})$.

Let $u$ be a member of $\mathtt{B}$ and set $N_1^+(u) = \mathtt{B} \setminus (C_m^u\cup\{u\})$, which is the set of out-neighbors of the vertex $u$ in $\mathrm{DPBG}(\mathtt{B})$. Clearly $|N_1^+(u)| < |\mathtt{B}|/2 \leq n/2$. Moreover, we can define the set of out-neighbors of $u$ at the $k$-th level in an inductive way:
$$
N_k^+(u) = \{ t \in \mathtt{B} \setminus \left( \cup_{i=1}^{k-1} N_i^+(u) \right) : \exists \hspace{0.6mm} s \in N_{k-1}^+(u) ~ s.t., ~ t \in N_1^+(s) \}, ~~ k = 2,\ldots,d,
$$
where $d$ denotes the longest distance from $x$ to other vertices in $\mathrm{DPBG}(\mathtt{B})$. 

Now let us see how to determine $\Pi_x^*$, the partition of $[n]$ composed of the orbits of $\mathfrak{G}_x$, by virtue of the distance between $x$ and the rest of vertices. Since each subspace $X_{\lambda,x,m-1}$ ($\lambda\in\drm{spec}{\dAM{G}}$) is spanned by $\{ \dproj{X_{\lambda,x,m-1}}(\pmb{e}_u) : u \in N_1^+(x) \}$, the dimension of $X_{\lambda,x,m-1}$ is less than $|\mathtt{B}|/2 \leq n/2$, so we can determine in a reductive way the orbits of $\drm{Aut}{\oplus_{\lambda} X_{\lambda,x,m-1}}$ and a series of partitions of $[n]$ associated with a fastening sequence of the group.

Let $t$ be an out-neighbor of $x$ and set $Z_{\lambda,x,t}^{(1)} = \duspan{Z_{\lambda,x,m}}{N_1^+(t)}$. We determine the partition of $[n]$ composed of the orbits of $\drm{Aut}{\oplus_{\lambda} Z_{\lambda,x,t}^{(1)}}$ and a series of partitions of $[n]$ associated with a fastening sequence of the group. Next we conduct a test for consistency of actions of $\drm{Aut}{\oplus_{\lambda} Z_{\lambda,x,t}^{(1)}}$ and of $\big(\drm{Aut}{\oplus_{\lambda} X_{\lambda,x,m-1}}\big)_t$ for every $t$ in $N_1^+(x)$. 

To be precise, we need to determine the partition of $[n]$, which is composed of the orbits of the group $\big(\drm{Aut}{\oplus_{\lambda} Z_{\lambda,x,t}^{(1)}}\big)\cap\big(\drm{Aut}{\oplus_{\lambda} X_{\lambda,x,m-1}}\big)_t$, and a series of partitions of $[n]$ associated with a fastening sequence of the group. Note that each member of $N_1^+(t)$ has a representative in $\oplus_{\lambda} X_{\lambda,x,m-1}$, so this could be done efficiently. The group resulted is denoted by $\drm{Aut}{\oplus_{\lambda,t} \left( X_{\lambda,x,m-1}} \oplus Z_{\lambda,x,t}^{(1)} \right)$.

Let $r$ be a vertex in $N_2^+(x)$ and let $Z_{\lambda,x,r}^{(2)}$ denote the orthogonal complement of the subspace $\duspan{Z_{\lambda,x,m}}{N_1^+(r)} \cap \big( \oplus_{t \in N_1^+(x)} Z_{\lambda,x,t}^{(1)} \big)$ in $\duspan{Z_{\lambda,x,m}}{N_1^+(r)}$. We determine the partition of $[n]$ consisting of the orbits of $\drm{Aut}{\oplus_{\lambda} Z_{\lambda,x,r}^{(2)}}$ and a series of partitions of $[n]$ associated with a fastening sequence of the group. Again we need to conduct a test for consistency of actions of $\drm{Aut}{\oplus_{\lambda} Z_{\lambda,x,r}^{(2)}}$ and of $\big(\drm{Aut}{\oplus_{\lambda} X_{\lambda,x,m-1}}\big)_{t,r}$ for every $r$ in $N_2^+(x)$. The group resulted is denoted by $\drm{Aut}{\oplus_{\lambda,t,r} \left( X_{\lambda,x,m-1}} \oplus Z_{\lambda,x,t}^{(1)} \oplus Z_{\lambda,x,r}^{(2)} \right)$.

One can readily see that by repeating the process above for each $u \in N_k^+(x)$ ($k=2,3,\ldots,d$), we can finally obtain the information about $\mathfrak{G}_x$.

 \item[2.2)] $\exists\hspace{0.6mm} \pmb{s}_1,\ldots,\pmb{s}_q \in \oplus_{\lambda} X_{\lambda,x,m-1}$ such that $2\leq q \leq |C_m^x|/2$ and $\forall\hspace{0.6mm} w \in C^x_m$, $\exists\hspace{0.6mm} i \in [q]$, 
 $$\dproj{\oplus_{\lambda} X_{\lambda,x,m-1}}(\pmb{e}_{w}) = \pmb{s}_i.$$
 
Clearly, we need the information about $\drm{Aut}{\oplus_{\lambda} X_{\lambda,x,m-1}}$ in dealing with the subspace $\oplus_{\lambda} Z_{\lambda,x,m}$, so we first determine that reductively. Furthermore, one can readily see that if $C^x_m$ is one of orbits of $\mathfrak{G}_x$ then those subsets $\drm{proj}{^{-1}} \pmb{s}_1,\ldots,\drm{proj}{^{-1}} \pmb{s}_q$ comprise a block system of $\mathfrak{G}_x$. Accordingly, in order to obtain the information about $\drm{Aut}{\oplus_{\lambda} Z_{\lambda,x,m}}$, we need to work out the information about $\drm{Aut}{\oplus_{\lambda} Z_{\lambda,x,\pmb{s}_i}}$, where $Z_{\lambda,x,\pmb{s}_i} = \duspan{Z_{\lambda,x,m}}{ \drm{proj}{^{-1}} \pmb{s}_i }$ and $i=1,\ldots,q$.

On the other hand, if the action of $\drm{Aut}{\oplus_{\lambda} X_{\lambda,x,m-1}}$ on $\pmb{s}_1,\ldots,\pmb{s}_q$ is not the same as the action of $\drm{Sym}{[q]}$ on $[q]$, then there are at least 3 orbits of $\big(\drm{Aut}{\oplus_{\lambda} X_{\lambda,x,m-1}}\big)_{\pmb{s}_i}$, $\forall i\in [q]$, so for some eigenvalue $\lambda$ we can split the subspace $Z_{\lambda,x,m}$ into smaller subspaces in a way like what we did on $V_{\lambda}$ with the cells $S_1,\ldots,S_t$ of $\bar{\Pi}[ \oplus V_{\lambda} ]$. Hence, we assume in what follows that the action of $\drm{Aut}{\oplus_{\lambda} X_{\lambda,x,m-1}}$ on $\{ \pmb{s}_1,\ldots,\pmb{s}_q \}$ is transitive and for each $i \in [q]$ $\big(\drm{Aut}{\oplus_{\lambda} X_{\lambda,x,m-1}}\big)_{\pmb{s}_i}$ possesses only two orbits $\{ \pmb{s}_i \}$ and $\{ \pmb{s}_1,\ldots,\pmb{s}_q \}\setminus\{ \pmb{s}_i \}$. As a result there are only two cases. 

  \begin{itemize}
  
   \item[2.2A)] $\forall \hspace{0.6mm} i,j \in [q]$, if $i\neq j$ then $Z_{\lambda,x,\pmb{s}_i} \perp Z_{\lambda,x,\pmb{s}_j}$. 
   
In this case, each subspace $Z_{\lambda,x,m}$ ($\lambda\in\drm{spec}{\dAM{G}}$) could be decomposed as an orthogonally direct sum $\oplus_{k = 1}^q Z_{\lambda,x,\pmb{s}_k}$, so we can employ the machinery developed for dealing with the case 1) to work out the information about the group $\drm{Aut}{\oplus_{\lambda} Z_{\lambda,x,m}}$, {\it i.e.,} the information about the partition of $C_m^x$ consisting of orbits of $\drm{Aut}{\oplus_{\lambda} Z_{\lambda,x,m}}$ and a series of partitions of $C_m^x$ associated with a fastening sequence of the group.

   \item[2.2B)] $\forall \hspace{0.6mm} i,j \in [q]$, $Z_{\lambda,x,\pmb{s}_i}$ and $Z_{\lambda,x,\pmb{s}_j}$ are not orthogonal to one another.
   
Since we cannot split $Z_{\lambda,x,m}$ by the partition of $C_m^x$ consisting of orbits of the group $\big(\drm{Aut}{\oplus_{\lambda} X_{\lambda,x,m-1}}\big)_{\pmb{s}_i}$, we have to explore those subspaces $Z_{\lambda,x,\pmb{s}_1},\cdots,Z_{\lambda,x,\pmb{s}_q}$ one by one in order to determine the structure of $\drm{Aut}{\oplus_{\lambda} Z_{\lambda,x,m}}$.

Let us pick arbitrarily one subset $\drm{proj}{^{-1}} \pmb{s}_{k_1}$ from the family $\{ \drm{proj}{^{-1}} \pmb{s}_1,\ldots,\drm{proj}{^{-1}} \pmb{s}_q \}$ and work out the information about $\drm{Aut}{\oplus_{\lambda} Z_{\lambda,x,\pmb{s}_{k_1}}}$ reductively. Note that $Z_{\lambda,x,\pmb{s}_i}$ and $Z_{\lambda,x,\pmb{s}_j}$ are not orthogonal, $\forall \hspace{0.6mm} i,j \in [q]$, so we now know the exact way of $\drm{Aut}{\oplus_{\lambda} Z_{\lambda,x,\pmb{s}_{k_1}}}$ action on the family $\{ \drm{proj}{^{-1}} \pmb{s}_1,\ldots,\drm{proj}{^{-1}} \pmb{s}_q \}\setminus\{ \drm{proj}{^{-1}} \pmb{s}_{k_1} \}$. Moreover we have a natural relation among members of $C_m^x$: two vertices $w'$ and $w''$ are said to be related if 
$$
\dproj{\oplus_{\lambda} \big(X_{\lambda,x,m-1} \oplus Z_{\lambda,x,\pmb{s}_{k_1}} \big)}(\pmb{e}_{w'}) =
\dproj{\oplus_{\lambda} \big(X_{\lambda,x,m-1} \oplus Z_{\lambda,x,\pmb{s}_{k_1}} \big)}(\pmb{e}_{w''}). 
$$
Evidently, it is an equivalence relation, so there is a partition $P_{k_1}$ of $C_m^x$ induced from the relation. 

If each cell of $P_{k_1}$ is actually a singleton, we can deal with the rest of subsets $\{ \drm{proj}{^{-1}} \pmb{s}_i : 1\leq i \leq q \mbox{ and } i \neq k_1 \}$ in virtue of the structure of $\drm{Aut}{\oplus_{\lambda} Z_{\lambda,x,\pmb{s}_{k_1}}}$. Accordingly, let us assume there are non-trivial cells in $P_{k_1}$. A moment's reflection would show that one can readily refine $P_{k_1}$ by means of Lemma \ref{Lemma-EquitablePartProj}, so we make a further assumption that the partition $\left\{ \{v\} : v \in \cup_{i=1}^{m-1} C_i^x \right\} \cup P_{k_1}$ of $[n]$ is an equitable one.

Now, we pick arbitrarily a subset $\drm{proj}{^{-1}} \pmb{s}_{k_2}$ from $\{ \drm{proj}{^{-1}} \pmb{s}_1,\ldots,\drm{proj}{^{-1}} \pmb{s}_q \} \setminus \{ \drm{proj}{^{-1}} \pmb{s}_{k_1} \}$ and find the information about $\drm{Aut}{\oplus_{\lambda} Z_{\lambda,x,\pmb{s}_{k_2}}}$ reductively. Then we can obtain a refinement $P_{k_2}$ of the partition $P_{k_1}$ by comparing projections of $C_m^x$ onto the subspace $\oplus_{\lambda} \big(X_{\lambda,x,m-1} \oplus Z_{\lambda,x,\pmb{s}_{k_1}} \oplus Z_{\lambda,x,\pmb{s}_{k_2}} \big)$. Because of the relation that  $Z_{\lambda,x,\pmb{s}_i}$ and $Z_{\lambda,x,\pmb{s}_j}$ are not orthogonal, $\forall \hspace{0.6mm} i,j \in [q]$, we need to expose at most $\min\{q,\lceil \log p \rceil\}$ subsets in the family $\{ \drm{proj}{^{-1}} \pmb{s}_i : i \in [q] \}$ to obtain a partition of $C_m^x$ with all cells singleton. After having exposed $\min\{q,\lceil \log p \rceil\}$ subsets, we deal with the rest of subsets in the family  according to the structure of the group determined by subspaces we have investigated, which is the same as what we did in dealing with the case 2.1).
  
  \end{itemize}

 \end{itemize}

\end{itemize}

Now let us turn back to the 1st case that $|C_m^x| \leq n/2$. Recall that $\Pi[ \oplus V_{\lambda} ; x ]$ is a balanced partition consisting of cells $C^x_1 = \{ x \},C^x_2,\ldots,C^x_m$ such that $m\geq 3$ and $|C^x_2| \leq \cdots \leq |C^x_m|$, and that each eigenspace $V_{\lambda}$ of $\dAM{G}$ possesses an orthogonal decomposition $V_{\lambda,\Pi_x} \oplus Y_{\lambda,x}$. For each $i\in [m]$, we use $X_{\lambda,C^x_i}$ ($\lambda \in \drm{spec}{\dAM{G}}$ and $i \in [m]$) to denote the subspace $\duspan{Y_{\lambda,x}}{C^x_i}$. Since $\mathtt{B} = [n]$, $C^x_2$ cannot be a singleton. We only show in what follows how to work out the information about the structure of $\drm{Aut}{\oplus_{\lambda} \big( X_{\lambda,C^x_2} \oplus X_{\lambda,C^x_3} \big)}$, for we can use the same method to deal with other cells of $\Pi[ \oplus V_{\lambda} ; x ]$.  Apparently, there are two possibilities.

\begin{itemize}

\item[(1)] $\duspan{\oplus_{\lambda} \hspace{0.5mm} X_{\lambda,C^x_2}}{C^x_3}\neq\pmb{0}$.

It is clear that in this case we should use the information about $\drm{Aut}{\oplus_{\lambda} X_{\lambda,C^x_2}}$ to reveal symmetries represented in $\oplus_{\lambda} X_{\lambda,C^x_3}$, which is similar to the case 2).

As we have pointed out in Section 3.1, one may use $C^x_3$ to split the subspace $X_{\lambda,C^x_2}$, so we make a further assumption that $\duspan{X_{\lambda,C^x_2}}{C^x_3} = X_{\lambda,C^x_2}$ for any $\lambda \in \drm{spec}{\dAM{G}}$. On the other hand, it is easy to see that $Y_{\lambda,x}$ and $X_{\lambda,C^x_2}$ are both $\mathfrak{G}_x$-invariant, so is the subspace $Y_{\lambda,x} \ominus X_{\lambda,C^x_2}$ that is the orthogonal complement of $X_{\lambda,C^x_2}$ in $Y_{\lambda,x}$. Let $\hat{X}_{\lambda,C^x_3}$ denote the subspace $\duspan{Y_{\lambda,x} \ominus X_{\lambda,C^x_2}}{C^x_3}$. 

Because the partition $\Pi[ \oplus V_{\lambda} ; x ]$ is equitable, $\dproj{ X_{\lambda,C_2^x} }( \pmb{R}_{C_3^x} ) = \pmb{0}$ for any $\lambda \in \drm{spec}{\dAM{G}}$. Notice that $\duspan{\oplus_{\lambda} \hspace{0.5mm} X_{\lambda,C^x_2}}{C^x_3}\neq\pmb{0}$, so there exist a group of vectors $\pmb{s}_1,\ldots,\pmb{s}_q$ in $\oplus_{\lambda} \hspace{0.5mm} X_{\lambda,C^x_2}$ such that $q\geq 2$ and $\forall \hspace{0.5mm} i \in [q]$, $\exists \hspace{0.5mm} w \in  C_m^x$ {\it s.t.,} $\dproj{\oplus_{\lambda}  X_{\lambda,C^x_2}}(\pmb{e}_w) = \pmb{s}_i$.  

Set $\drm{proj}{^{-1}} \pmb{s}_i = \{ w \in C_3^x : \dproj{\oplus_{\lambda} X_{\lambda,C^x_2}}(\pmb{e}_w) = \pmb{s}_i \}$. Since each subspace $X_{\lambda,C^x_2}$ is $\mathfrak{G}_x$-invariant, if any one of three cases below occurs then $C_3^x$ cannot be an orbit of $\mathfrak{G}_x$:
\begin{enumerate}

\item $\cup_{i=1}^q \hspace{0.5mm} \drm{proj}{^{-1}} \pmb{s}_i \subsetneq C^x_3$;

\item $\exists \hspace{0.5mm} i,j \in [q]$ {\it s.t.,} $|\drm{proj}{^{-1}} \pmb{s}_i| \neq |\drm{proj}{^{-1}} \pmb{s}_j|$;

\item $\exists \hspace{0.5mm} i,j \in [q]$ {\it s.t.,} $\pmb{s}_i$ and $\pmb{s}_j$ do not belong to the same orbit of $\drm{Aut}{\oplus_{\lambda} X_{\lambda,C^x_2}}$.

\end{enumerate}
In what follows, we assume none of cases listed above occurs. Clearly there are again two cases. 

 \begin{itemize}
 
 \item[(1.1)]  $\forall\hspace{0.6mm} x_3,y_3 \in C^x_3$, $\dproj{\oplus_{\lambda} X_{\lambda,C^x_2}}(\pmb{e}_{x_3}) \neq \dproj{\oplus_{\lambda}  X_{\lambda,C^x_2}}(\pmb{e}_{y_3})$.

Obviously, if the partition of $C^x_2$, composed of orbits of $\drm{Aut}{\oplus_{\lambda} X_{\lambda,C^x_2}}$, has only singleton cells, then the action of $\drm{Aut}{\oplus_{\lambda} \big( X_{\lambda,C^x_2} \oplus \hat{X}_{\lambda,C^x_3} \big)}$ on $C^x_3$ is also trivial, so we assume that the action of $\drm{Aut}{\oplus_{\lambda} X_{\lambda,C^x_2}}$ on $C^x_2$ is transitive, otherwise we consider those orbits one by one. For the same reason we suppose that the action of $\drm{Aut}{\oplus_{\lambda} X_{\lambda,C^x_2}}$ on $C^x_3$ is also  transitive. 

Since $\dproj{X_{\lambda,C^x_2}}( \pmb{R}_{C_3^x} ) = \pmb{0}$, $\forall\lambda\in\drm{spec}{\dAM{G}}$, the action of $\big( \drm{Aut}{\oplus_{\lambda} X_{\lambda,C^x_2}} \big)_{u_2}$ on $C_3^x$ possesses at least two orbits, where $u_2$ is a vertex $C_2^x$. We use $T_{u_2}(C_3^x)$ to denote the one of the minimum order, so $|T_{u_2}(C_3^x)| \leq |C_3^x|/2$. Moreover $\bigcup_{u_2 \in C_2^x} T_{u_2}(C_3^x) = C_3^x$, for the action of $\drm{Aut}{\oplus_{\lambda} X_{\lambda,C^x_2}}$ on $C^x_3$ is transitive.  

Set $\hat{X}_{\lambda,C_3^x,u_2} = \duspan{\hat{X}_{\lambda,C_3^x}}{T_{u_2}(C_3^x)}$, where $\lambda \in \drm{spec}{\dAM{G}}$. Then for each $v_2 \in C_2^x$, we work out the information about the group $\drm{Aut}{\oplus_{\lambda} \hat{X}_{\lambda,C_3^x,v_2}}$. Next we conduct a test for consistency of actions of $\drm{Aut}{\oplus_{\lambda} \hat{X}_{\lambda,C_3^x,v_2}}$ and of $\big( \drm{Aut}{\oplus_{\lambda} X_{\lambda,C^x_2}} \big)_{v_2}$ for every $v_2$ in $C_2^x$. 

To be precise we need to determine the orbits of $\big( \drm{Aut}{\oplus_{\lambda} \hat{X}_{\lambda,C_3^x,v_2}} \big) \cap \big( \drm{Aut}{\oplus_{\lambda} X_{\lambda,C^x_2}} \big)_{v_2}$ and a series of partitions of $C_3^x$ associated with a fastening sequence of the group. Note that each member of $T_{v_2}(C_3^x)$ has a representative in $\oplus_{\lambda} X_{\lambda,C_2^x}$, so this could be done efficiently. As a result, we could obtain the information about $\drm{Aut}{\oplus_{\lambda} \big( X_{\lambda,C^x_2} \oplus \hat{X}_{\lambda,C^x_3} \big)}$.

 \item[(1.2)] $\exists\hspace{0.6mm} \pmb{s}_1,\ldots,\pmb{s}_q \in \oplus_{\lambda} X_{\lambda,C^x_2}$ such that $2\leq q \leq |C_3^x|/2$ and $\forall\hspace{0.6mm} w_3 \in C^x_3$, $\exists\hspace{0.6mm} i \in [q]$, 
 $$\dproj{\oplus_{\lambda} X_{\lambda,x,m-1}}(\pmb{e}_{w_3}) = \pmb{s}_i.$$

One can readily see that we can employ the method for dealing with the case 2.2) and (1.1) to work out the information about $\drm{Aut}{\oplus_{\lambda} \big( X_{\lambda,C^x_2} \oplus \hat{X}_{\lambda,C^x_3} \big)}$.
 
 \end{itemize}

\item[(2)] $\oplus_{\lambda} X_{\lambda,C^x_2} \perp \oplus_{\lambda} X_{\lambda,C^x_3}$.

It is clear that if there exists a cell $C_i^x$ ($i\neq 2$ or 3) of $\Pi[ \oplus V_{\lambda} ; x ]$ such that $\duspan{\oplus_{\lambda} X_{\lambda,C^x_2}}{C_i^x} \neq \pmb{0}$ and $\duspan{\oplus_{\lambda} X_{\lambda,C^x_3}}{C_i^x} \neq \pmb{0}$, then we can use that cell to split some of subspaces in the sum $\oplus_{\lambda} X_{\lambda,C^x_2}$ or in the sum $\oplus_{\lambda} X_{\lambda,C^x_3}$. Consequently we assume that for each $\lambda\in\drm{spec}{\dAM{G}}$, $\duspan{X_{\lambda,C^x_2}}{C_i^x} = X_{\lambda,C^x_2}$ and $\duspan{X_{\lambda,C^x_3}}{C_i^x} = X_{\lambda,C^x_3}$. In this case, we can use the method for dealing with the case (1) to work out $\drm{Aut}{\oplus_{\lambda} \big( X_{\lambda,C^x_2} \oplus X_{\lambda,C^x_3} \big)}$ with $C_3^x$ replaced by $C_i^x$.

As a result, we assume that $\forall C_i^x \in \Pi[ \oplus V_{\lambda} ; x ]$, 

$$
\duspan{\oplus_{\lambda} X_{\lambda,C_2^x}}{C_i^x} \neq \pmb{0} \Rightarrow \oplus_{\lambda} X_{\lambda,C_i^x} \perp \oplus_{\lambda} X_{\lambda,C_3^x},$$ or 
$$\duspan{\oplus_{\lambda} X_{\lambda,C_3^x}}{C_i^x} \neq \pmb{0} \Rightarrow \oplus_{\lambda} X_{\lambda,C_i^x} \perp \oplus_{\lambda} X_{\lambda,C_2^x}.$$
In other words, $C_2^x$ and $C_3^x$ are completely irrelevant under the action of $\mathfrak{G}_x$. Then we can cope with $\oplus_{\lambda} X_{\lambda,C_2^x}$ and $\oplus_{\lambda} X_{\lambda,C_3^x}$ separately.

\end{itemize}

\vspace{2mm}
Now let us turn to the case that $\bar{\Pi}[ \oplus V_{\lambda} ] = \{ [n] \}$ but $\mathtt{B} \subsetneq [n]$. Apparently the quotient graph $G/\bar{\Pi}[ \oplus V_{\lambda} ]$ has only one vertex in this case, so $G$ is a regular graph and thus $\lambda \notin \drm{spec}{\dAM{G/\bar{\Pi}[ \oplus V_{\lambda} ]}}$ if $\lambda$ is not the biggest eigenvalue $\lambda_1$ of $\dAM{G}$.  Recall that by carrying out first two operations of outputting $\Pi[ \oplus V_{\lambda} ; \mathtt{B} ]$ on the rest of vertices of $G$, one can obtain a group of subsets $\mathtt{B}_1 = \mathtt{B},\mathtt{B}_2,\ldots,\mathtt{B}_q$, which form a partition of $[n]$. We use $V_{\lambda,\mathtt{B}_i}$ to denote the subspace $\duspan{ V_{\lambda} }{\mathtt{B}_i}$, where $\lambda \in \drm{spec}{\dAM{G}} \setminus \{\lambda_1\}$ and $i = 1,\ldots,q$. Then there are two possibilities. 

\begin{itemize}

\item[I)] $\forall \hspace{0.5mm} i,j \in [q]$, if $i \neq j$ then $\oplus_{\lambda \neq \lambda_1} V_{\lambda,\mathtt{B}_i} \perp \oplus_{\lambda \neq \lambda_1} V_{\lambda,\mathtt{B}_j}$.

It is easy to see that one can use the machinery developed for dealing the case 1) to obtain the information about $\drm{Aut}{\oplus V_{\lambda}}$.

\item[II)] $\exists \hspace{0.5mm} i,j \in [q]$, {\it s.t.,} $i \neq j$ and $\duspan{\oplus_{\lambda \neq \lambda_1} V_{\lambda,\mathtt{B}_i}}{\mathtt{B}_j} \neq \pmb{0} $.

In this case, we first determine the partition $\stackrel{\scriptscriptstyle \Box }{\Pi}[ \oplus V_{\lambda} ; \mathtt{B} ]$ of $\{ \mathtt{B}_1,\ldots,\mathtt{B}_q \}$ induced by $\Pi[ \oplus V_{\lambda} ; \mathtt{B} ]$. Suppose  $\mathtt{L}_1=\mathtt{B}_1,\mathtt{L}_2,\ldots,\mathtt{L}_c$ are cells of $\stackrel{\scriptscriptstyle \Box }{\Pi}[ \oplus V_{\lambda} ; \mathtt{B} ]$ such that $|\mathtt{L}_1| \leq \cdots \leq |\mathtt{L}_c|$, and set $Y_{\lambda,\mathtt{B}_1} = V_{\lambda} \ominus V_{\lambda,\mathtt{B}_1}$. Then we can use those cells to split the subspace $Y_{\lambda,\mathtt{B}_1}$. 

Let $X_{\lambda,\mathtt{L}_i}$ ($i = 2,\ldots,c$) denote the subspace $\duspan{Y_{\lambda,\mathtt{B}_1}}{\mathtt{L}_i}$. Clearly each cell $\mathtt{L}_i$ ($i=2,\ldots,c$) is invariant under the action of $\mathfrak{G}_{\mathtt{B}}$, so is the subspace $\oplus_{\lambda} X_{\lambda,\mathtt{L}_i}$ according to Lemma \ref{ProjOperatorCommutative}. On the other hand, each $\mathtt{L}_i$ may contain some of cells of $\Pi[ \oplus V_{\lambda} ; \mathtt{B} ]$, so we can split $X_{\lambda,\mathtt{L}_i}$ further by means of those cells relevant. Note that $\Pi[ \oplus V_{\lambda} ; \mathtt{B} ]$ is an equitable partition, so Lemma \ref{Lem-SeparatingPseudoOrbit} works well for cells of $\Pi[ \oplus V_{\lambda} ; \mathtt{B} ]$. As a result, we can finally decompose the subspace $Y_{\lambda,\mathtt{B}_1}$ in a way similar to (\ref{SubspaceDecomposition-I}), and accordingly we can use the machinery developed for the case that $\mathtt{B} = [n]$ and $| C_m^x | \leq n/2$ to work out the information about $\drm{Aut}{\oplus V_{\lambda}}$.

\end{itemize}

It is not difficult to verify that in the case that $\bar{\Pi}[ \oplus V_{\lambda} ] = \{ S_1,\ldots,S_t \}$ with $t \geq 2$, one can use the machinery developed for finding the information about $\mathfrak{G}_x$ to obtain the information about $\mathfrak{G}$.

\subsection{Complexity Analysis}

As we have seen in the first two parts of this section, the algorithm $\mathscr{A}$ outputs, by inputting the decomposition $\oplus V_{\lambda}$, the information about $\mathfrak{G}$. Let $f(n)$ denote the number of computations involved by carrying out $\mathscr{A}$. Now we analyze the complexity of the algorithm. 

First of all, it is routine to check that the number of computations involved for obtaining two partitions $\bar{\Pi}[\oplus V_{\lambda}]$ and $\Pi[ \oplus V_{\lambda};\mathtt{B} ]$ is bounded above by $n^K$ for some integer $K$. Suppose the adjacency matrix $\dAM{G}$ possesses $t$ distinct eigenvalues.  

We shall prove by induction on $n$ that $f(n) \leq n^{C \log n}$, where $C$ is a constant not less than $\max\{ K,4 \}$. Let us first consider those three cases relevant to the restriction that $\bar{\Pi}[ \oplus V_{\lambda} ] = \{ [n] \}$, $\mathtt{B} = [n]$ and $| C_m^x | > n/2$. One can readily verify the assertion for $n$ less than 4. We assume the assertion holds for any positive integer not more than $n-1$. 

\begin{itemize}

\item[1)] Let $p$ stand for the order of each cell $E_i^B$, where $i = 1,\ldots,q$ and $q \geq 2$. Consequently, $n = p \cdot q$ and thus 
$$
f(n) \leq n \cdot \left[ n^K + t \cdot {q \choose 2} f(p) \right] 
 \leq n^{1+K} + n^2 \cdot q^2 f(p).
$$
According to the inductive hypothesis, $f(p) \leq p^{C\log p}$. Hence
$$
q^2 f(p) \leq  q^2 \cdot p^{C\log p}
 \leq  (q \cdot p)^{C\log p} 
 \leq  n^{C\log (n/2)} 
 =  n^{C\log n}/n^{C}.
$$
As a result, $f(n) / n^{C\log n} \leq n^{ (1+K) - 2C } + n^{2-C} \leq 1$.

\item[2)] Recall that $\Pi[ \oplus V_{\lambda} ; x ] = \left\{ C^x_1 = \{ x \},C^x_2,\ldots,C^x_m \right\}$  and $|C^x_2| \leq \cdots \leq |C^x_m|$, where $m \geq 3$. Set $s = \left|\cup_{i=2}^{m-1} C_i^x \right|$. Then $s < n/2$ since $|C^x_m| > n/2$.

\begin{itemize}

\item[2.1)] Let $l_k$ ($k=1,\ldots,d$) denote the order of $N_k^+(x)$, which is the set of out-neighbors of $x$ at the $k$-th level in the graph $\mathrm{DPBG}(\mathtt{B})$. Then
\begin{align*}
f(n) \leq &~ n \cdot \left[ n^K + t\cdot \left( f(s) + \sum_{k=1}^d l_k \cdot 2 f(s) \right)  \right]\\
 \leq &~ n \cdot \left[ n^K + t \cdot \left( 1 + \sum_{k} l_k \right) 2 f(s) \right] \\
 \leq &~ n \cdot \left[ n^K + n^2 \cdot 2 f(s) \right]
\end{align*}
According to the inductive hypothesis, $f(s) \leq s^{C\log s} \leq (n/2)^{C\log (n/2)}$. Hence
$$
n^3 \cdot 2f(s) \leq  2 n^3 \cdot  (n/2)^{C\log (n/2)}
 = n^{C \log n} \cdot \frac{2^{C+1}}{n^{2C-3}}.
$$
As a result, $f(n) / n^{C\log n} \leq n^{ (1+K) - 2C } + (2/n)^{C+1} \cdot n^{4-C} \leq 1$.

\item[2.2)] It is easy to see that in the case 2.2A), we can use the argument used in dealing with the case 1) to prove the assertion, so let us consider the case 2.2B). Suppose the order of the subset $\drm{proj}{^{-1}} \pmb{s}_{i}$ is equal to $p$, where $1 \leq i \leq q$. 

Note that after having exposed $\min\{q,\lceil \log p \rceil\}$ subsets, we have a partition of $C_m^x$ with all cells singleton, so we can construct a direct graph like $\mathrm{DPBG}(\mathtt{B})$ to deal with the rest of subsets in the family $\{ \drm{proj}{^{-1}} \pmb{s}_1,\ldots,\drm{proj}{^{-1}} \pmb{s}_q \}$. Again we use $l_k$ ($k=1,\ldots,d$) denote the order of $N_k^+(x)$, which is the set of out-neighbors of $x$ at the $k$-th level in the new graph. Then 
\begin{align*}
f(n) \leq &~ n \cdot \Bigg[ n^K + t\cdot \bigg( f(s) + \prod_{i=1}^{\min\{q,\lceil \log p \rceil\}} (q-i) \cdot f(p)  \\
~ &~ \hspace{1.35cm} + \sum_{k=1}^d l_k \cdot \Big( f(p) + f(s) + \prod_{i=1}^{\min\{q,\lceil \log p \rceil\}} (q-i) \cdot f(p)  \Big)\bigg)  \Bigg]  \\
 \leq &~ n \cdot \left[ n^K + t \cdot  \Big(1 + \sum_{k} l_k \Big) \cdot \Big( f(s) + f(p) + \prod_{i} (q-i) \cdot f(p) \Big)\right] \\
 \leq &~ n \cdot \left[ n^K + n^2 \cdot \Big( f(s) + f(p) + \prod_{i} (q-i) \cdot f(p) \Big) \right]
\end{align*}
According to the inductive hypothesis, we have 
$$
n^3 \cdot \big(f(s) + f(p)\big) \leq 2 n^3\cdot (n/2)^{ C\log (n/2) } = n^{C \log n} \cdot \frac{2^{C+1}}{n^{2C-3}}
$$
and 
$$
n^3 \prod_{i} (q-i) \cdot f(p) \leq n^3 q^{\lceil \log p \rceil} p^{C \log p} \leq n^3 \big( qp \big)^{C \log p} \leq n^3 n^{C\log (n/2)} \leq n^{C\log n} / n^{C-3}.
$$
As a result, $f(n) / n^{C\log n} \leq n^{(1+K)-2C} + (2/n)^{C+1} \cdot n^{4-C} + n^{3-C} \leq 1$.

\end{itemize}

\end{itemize}
Accordingly, $f(n) \leq n^{C \log n}$ if $\bar{\Pi}[ \oplus V_{\lambda} ] = \{ [n] \}$, $\mathtt{B} = [n]$ and $| C_m^x | > n/2$, where $C$ is a constant larger than or equal to $K$. By the same argument, one can easily prove that $f(n) \leq n^{C \log n}$ in the case that $\bar{\Pi}[ \oplus V_{\lambda} ] = \{ [n] \}$, $\mathtt{B} = [n]$ and $| C_m^x | \leq n/2$.

\vspace{6mm}
\noindent{\Large\bf Acknowledgments}

\vspace{3mm}
I would like to express my deep gratitude to Prof. Fu-Ji Zhang, Prof. Xue-Liang Li, Prof. Qiong-Xiang Huang, Prof. Wei Wang, Prof. Sheng-Gui Zhang, Prof. Li-Gong Wang and Prof. Johannes Siemons for their valuable advice which significantly improves the quality of this paper. I also want to thank Prof. Yi-Zheng Fan and Prof. Xiang-Feng Pan for their encouragement and support. Last but not the least, I would like to thank Dr. You Lu, Dr. Yan-Dong Bai, Dr. Bin-Long Li and Dr. Xiao-Gang Liu for helping me verify many parts of this paper.


\begin{thebibliography}{99}

\bibitem{Axler} Sheldon Axler, {\it Linear algebra - done right}, Springer-Verlag New York, 1997.

\bibitem{Babai} L. Babai, Graph isomorphism in quasipolynomial time, arXiv:1512.03547v2.

\bibitem{BaGrMu} L. Babai, D.Yu. Grigoryev and D.M. Mount, Isomorphism of graphs with bounded eigenvalue multiplicity, {\it Proc. 14th ACM Symposium on Theory of Computing} (ACM, New York): 310-324, 1982.

\bibitem{BHZ} R. Boppana, J. Hastad, and S. Zachos, Does co-NP have short interactive proofs? Information Processing Letters, 25(2):27–32, 1987.


\bibitem{DM} John D. Dixon and Brian Mortimer, {\it Permutation groups} (GTM 163), Springer-Verlag New York, 1996.

\bibitem{GodRoy} Chris Godsil and Gordon Royle, {\it Algebraic Graph Theory} (GTM 207), Springer-Verlag New York, 2001.

\bibitem{GolubLoan} Gene H. Golub and Charles F. Van Loan, {\it Matrix Computations}, Johns Hopkins Univ. Press, Baltimore, Maryland, 1996 (3rd edition).

\bibitem{Luks} E. Luks, Isomorphism of bounded valence can be tested in polynomial time, Journal of Computer and System Sciences, 25:42–65, 1982.


\bibitem{McP} Brendan D. McKay and Adolfo Piperno: Practical Graph Isomoprhism, II. arXiv:1301.1493, 2013.

\bibitem{PanChenZheng} Victor Y. Pan, Zhao Q, Chen and Ailong Zheng, The complexity of the algebraic eigenproblem, STOC 1999: 507-516. 

\bibitem{Schonig} U. Sch$\mathrm{\ddot{o}}$ning, Graph isomorphism is in the low hierarchy, Journal of Computer and System Sciences, 37:312–323, 1988.


\end{thebibliography}
\end{document}